\documentclass[a4paper, oneside]{amsart}

\usepackage{amsthm, amsmath, amssymb, amstext, amsfonts}
\usepackage[top=2cm, bottom=2cm]{geometry}
\usepackage{hyperref}
\usepackage{enumitem}
\usepackage[svgnames]{xcolor}
\usepackage{accents}

\definecolor{smoked}{RGB}{216, 212, 204}
\definecolor{mauve}{RGB}{200, 55, 171}
\definecolor{apricot}{RGB}{250, 144, 4}
\definecolor{sky}{RGB}{66, 169, 244}
\definecolor{plum}{RGB}{76, 0, 102}

\definecolor{lightmauve}{RGB}{232, 173, 220}
\definecolor{lightapricot}{RGB}{253, 211, 155}
\definecolor{lightsky}{RGB}{178, 221, 251}
\definecolor{lightplum}{RGB}{184, 153, 192}

\definecolor{darksmoked}{RGB}{198, 194, 176}

\hypersetup{
   colorlinks,
    linkcolor={mauve},
    citecolor={plum},
    urlcolor={plum}
}
\usepackage{graphicx}
\usepackage{tikz,tqft}
\usetikzlibrary{matrix,arrows,bending,cd,decorations.markings,calc,patterns,angles,quotes,backgrounds,math}
\tikzset{>=latex}

\theoremstyle{definition}
\newtheorem{defn}{Definition}
\newtheorem{rem}{Remark}
\theoremstyle{plain}
\newtheorem{lem}{Lemma}
\newtheorem{prop}{Proposition}
\newtheorem{thm}{Theorem}
 \newtheorem{thmx}{Theorem}
 
\newtheorem{cor}{Corollary}
\newtheorem{conj}{Conjecture}
\newtheorem{question}{Question}

\newcommand{\maps}{\colon}

\DeclareMathOperator{\arccosh}{arccosh}

\DeclareMathOperator{\Vol}{Vol}
\DeclareMathOperator{\dev}{dev}
\DeclareMathOperator{\fconf}{\tau}
\DeclareMathOperator{\hol}{hol}

\DeclareMathOperator{\R}{\mathbb{R}}
\DeclareMathOperator{\C}{\mathbb{C}}
\DeclareMathOperator{\Z}{\mathbb{Z}}

\DeclareMathOperator{\CP}{\mathbb{CP}}
\DeclareMathOperator{\HH}{\mathbb{H}}

\newcommand{\psl}{\operatorname{PSL}_2 \R}

\newcommand{\SurfMk}{X}
\newcommand{\SurfMkSmooth}{\SurfMk_\textnormal{sm}}
\newcommand{\SurfPk}{X'}
\newcommand{\SurfUni}{\widetilde{X}}
\newcommand{\SphereMk}{\Sigma}
\newcommand{\SphereMkSmooth}{\SphereMk_\textnormal{sm}}
\newcommand{\SphereExtMk}{S}
\newcommand{\SphereExtMkTop}{\SphereExtMk}
\newcommand{\SphereExtMkSmooth}{\SphereExtMk_\textnormal{sm}}
\newcommand{\SpherePk}{\Sigma'}
\newcommand{\SpherePkSmooth}{\SpherePk_\textnormal{sm}}
\newcommand{\SphereExtPk}{S'}
\newcommand{\SphereExtPkSmooth}{\SphereExtPk_\textnormal{sm}}
\newcommand{\SphereUni}{\widetilde{\Sigma}}
\newcommand{\SphereUniSmooth}{\SphereUni_\textnormal{sm}}
\newcommand{\SphereExtUni}{\widetilde{S}}
\newcommand{\SphereExtUniSmooth}{\SphereExtUni_\textnormal{sm}}
\newcommand{\AssembSphere}[1]{\SphereMk_{#1}}
\newcommand{\AssembSphereExt}[1]{\SphereExtMk_{#1}}
\newcommand{\AssembSpherePk}[1]{\SpherePk_{#1}}
\newcommand{\AssembSphereExtPk}[1]{\SphereExtPk_{#1}}

\DeclareMathOperator{\Rep}{Rep}
\DeclareMathOperator{\IntRep}{R\overset{\circ}{e}p}
\DeclareMathOperator{\Hyp}{Hyp}
\DeclareMathOperator{\Sym}{Sym}
\DeclareMathOperator{\Teich}{Teich}

\newcommand{\HypCone}[1]{\Hyp_{#1}}
\newcommand{\HypConeFree}{\Hyp}
\newcommand{\AugHypCone}[1]{\overline{\Hyp}_{#1}}
\newcommand{\RepDT}[1]{\Rep^\mathrm{DT}_{#1}}
\newcommand{\OldIntRepDT}[1]{\mathrm{R}\overset{\circ}{\mathrm{e}}\mathrm{p}_\alpha^\mathrm{DT}}
\newcommand{\IntRepDT}[2]{{\smash{\IntRep}\vphantom{\Rep}}_{#1, #2}^\mathrm{DT}}
\newcommand{\HypConeStruct}{\mathcal{F}}

\newcommand{\Assemb}{\mathcal{S}}
\newcommand{\AugAssemb}{\overline{\mathcal{S}}}
\newcommand{\Haman}{\mathcal{H}}

\title{The geometry of Deroin--Tholozan representations}

\author{Aaron Fenyes}
\address[A.~Fenyes]{}
\email{aaron.fenyes@fareycircles.ooo}

\author{Arnaud Maret}
\address[A.~Maret]{Sorbonne Université and Université Paris Cité, CNRS, IMJ-PRG, F-75005 Paris, France.}
\email{maret.arnaud@unistra.fr}

\date{\today}

\begin{document}
\begin{abstract}
We present a way to build hyperbolic spheres with conical singularities by gluing together simple building blocks. Our construction provides good control over the holonomy of the resulting hyperbolic cone sphere. In particular, it can be used to realize any Deroin-Tholozan (DT) representation as the holonomy of a hyperbolic cone sphere.

Our construction is inspired by the correspondence between DT representations and chains of triangles in the hyperbolic plane. It gives a geometric interpretation of certain action-angle coordinates on the space of DT representations, which come from this correspondence.
\end{abstract}
\maketitle

\section{Introduction}
\subsection{Motivation and results}\label{sec:motivation}
Every Fuchsian representation arises geometrically as the holonomy of a compact hyperbolic surface, which can be constructed by gluing the edges of a Dirichlet domain. There are many adventures to be had in generalizing this classic result to other kinds of surface group representations into $\operatorname{PSL}_2 \R$ and $\operatorname{PSL}_2 \C$; a few of them are related in~\cite{Far21, branch-sing, Mat12} and their references. We'll focus on {\em Deroin-Tholozan} (DT) representations (Section~\ref{sec:dt-review}), which are special representations of the fundamental group of a punctured sphere into $\operatorname{PSL}_2\R$. These are in some ways the opposites of Fuchsian representations: they're \emph{totally elliptic} instead of totally hyperbolic, and their Toledo numbers are near-zero instead of extremal. Deroin and Tholozan showed that every DT representation arises geometrically as the holonomy of a hyperbolic sphere with conical singularities (Section~\ref{cone-review})~\cite[Equation~14 and Corollary~4.6]{DeTh19}. We'll realize this existence result as a piece-by-piece construction.

Our construction is inspired by Maret's {\em triangle chain} parameterization of DT representations (Section~\ref{sec:dt-review})~\cite{action-angle}. By adding flaps to a chain of triangles that encodes a DT representation $\rho$, we build a net that folds up into a \emph{hyperbolic cone sphere} (Section~\ref{sec:space-of-hyperbolic-cone-metrics}) with holonomy $\rho$.
\begin{center}
\vspace{3mm}
\begin{tikzpicture}
\node[anchor=east] at (-0.5, 0) {\reflectbox{\includegraphics[width=6cm]{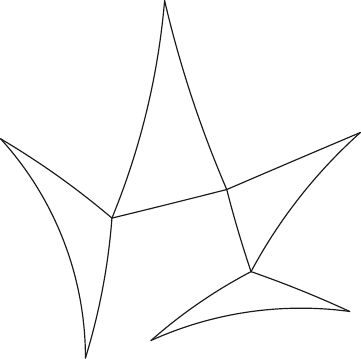}}};
\node[anchor=west] at (0.5, 0) {\reflectbox{\includegraphics[width=6cm]{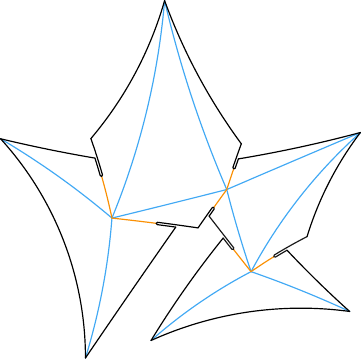}}};
\end{tikzpicture}
\vspace{3mm}
\end{center}
Each individual triangle, with its added flaps, folds up into a three-cornered hyperbolic sphere called a {\em hamantash} (Section~\ref{sec:hamantash})---an especially convenient example of a more general building block called a {\em samosa} (Section~\ref{sec:building-blocks}). The hyperbolic cone sphere that realizes $\rho$ can therefore be seen as a chain of hamantashen which have been slit open and glued corner to corner (Section~\ref{sec:assembly-instructions}). To prove that our construction works, we'll work backwards from a {\em hamantash assembly} of this kind, tuning its parameters so that it unfolds into a net which is built around the desired triangle chain.
\begin{thmx}[Corollary~\ref{cor:surjectivity-holonomies}]\label{thm-intro:surjectivity-holonomies}
Every DT representation is the holonomy of a hyperbolic cone sphere built from a hamantash assembly.    
\end{thmx}
Our construction requires a triangle chain without degenerate triangles or overlaps between adjacent triangles. Fortunately, the triangle chain parameterization of DT representations on an $n$-punctured sphere $\SpherePk$ depends on a geometric presentation of $\pi_1\SpherePk$ and a compatible pants decomposition of $\SpherePk$ (Section~\ref{sec:dt-review}). By choosing these carefully, we can encode any given DT representation as a triangle chain that fulfills our wishes.
\begin{thmx}[Propositions~\ref{prop:non-degenerate-pants-decomp}--\ref{prop:no-overlap}]\label{thm-intro:no-overlap-no-degenerate-triangles}
For every DT representation $\rho\colon\pi_1\SpherePk\to\operatorname{PSL}_2\R$, there exists a system of generators of $\pi_1\SpherePk$ and a compatible pants decomposition of $\SpherePk$ such that in the triangle chain associated to $\rho$, no consecutive triangles overlap and no triangles are degenerate.
\end{thmx}

The triangle chain parameterization encodes a system of action-angle coordinates on the space of DT representations~\cite{action-angle}. These coordinates pull back to simple linear functions on the space of hamantash assemblies (Theorem~\ref{thm:holonomy-hamantash}). They make it easy to describe the deformations of a hamantash assembly that leave the holonomy unchanged, which account for all the local deformations of the underlying conformal structure (Section~\ref{sec:hyperbolic-cone-structures})---parameterizing them in some way (Question~\ref{q:get-coordinates-for-Teich-?}). We might therefore learn a lot by using hamantash assemblies, and {\em samosa assemblies} more generally, to chart the space of hyperbolic cone structures on a sphere (Definition~\ref{defn:hyperbolic-cone-structure}). To formalize this idea, we consider spaces $\Assemb_{\alpha, \Upsilon}^\varepsilon$ of samosa assemblies whose shared combinatorial form $\Upsilon$ and parameter restriction $\varepsilon$ serve to identify all of them, up to isotopy, with the same topological $(2n-3)$-marked sphere $\AssembSphereExt{\Upsilon}$. Each samosa assembly in $\Assemb_{\alpha, \Upsilon}^\varepsilon$ can then be {\em realized} as a hyperbolic cone structure, up to isotopy, on $\AssembSphereExt{\Upsilon}$~(Section~\ref{sec:marking}). The {\em realization map} that describes this process turns out to be a well-behaved chart for the space of hyperbolic cone structures. We've proven this for the simplest choices of $\Upsilon$, and we expect it to hold in general.

\begin{thmx}[Theorem~\ref{thm:realization-map-local-homeo}]\label{thm-intro:realization-map-local-homeo}
For any {\em chained pants assembly} $\Upsilon$ (Section~\ref{sec:assembly-instructions-parametrization}), the realization map from $\Assemb_{\alpha, \Upsilon}^\varepsilon$ to the space of hyperbolic cone structures on $\AssembSphereExt{\Upsilon}$ is a homeomorphism onto its image. Moreover, if Conjecture~\ref{smooth-atlas} holds, the realization map is a diffeomorphism onto its image.
\end{thmx}

The topology we use for the space of hyperbolic cone structures on $\AssembSphereExt{\Upsilon}$ comes from the compact-open topology on the space of developing maps (Section~\ref{sec:space-of-hyperbolic-cone-structures}). This is what we call the \emph{$\mathcal{C}^0$ topology} on the space of hyperbolic cone structures up to isotopy, in contrast to the more traditional $\mathcal{C}^\infty$ topology on a space of hyperbolic cone {\em metrics} (Section~\ref{sec:space-of-hyperbolic-cone-metrics}). We'll prove that the $\mathcal{C}^0$ topology can be expressed in terms of the lengths of finitely many geodesic arcs between cone points (Proposition~\ref{length-homeo}). This correspondence allows for an elementary proof of Theorem~\ref{thm-intro:realization-map-local-homeo} based on hyperbolic trigonometry.

\subsection{Structure of the paper}
We first review the theory of cone metrics on surfaces (Section~\ref{cone-review}), recalling McOwen and Troyanov's uniformization of hyperbolic cone metrics (Corollary~\ref{cor:troyanov}). We then review DT representations, their parameterization by triangle chains, and the corresponding parameterization by action-angle coordinates (Section~\ref{sec:dt-review}).

In Chapter~\ref{chap:param-hyp-cone-metrics}, we discuss topologies on the space of hyperbolic cone metrics and on the space of hyperbolic cone structures. After reviewing the traditional topology, which comes from the $\mathcal{C}^\infty$ topology on the space of developing maps (Section~\ref{sec:space-of-hyperbolic-cone-metrics}), we introduce a coarser topology that comes from the $\mathcal{C}^0$ topology on developing maps (Section~\ref{sec:hyperbolic-cone-structures}). This coarser topology is made-to-measure for our work, and it can be induced by the lengths of geodesic arcs that run between cone points to triangulate the sphere (Section~\ref{sec:parametrization-using-triangulations}).

Our piece-by-piece construction of hyperbolic cone spheres is described in Chapter~\ref{sec:building-cone-surfs}. We introduce the building blocks (Section~\ref{sec:building-blocks}), called samosas, and explain how to assemble samosas into a hyperbolic cone sphere using a conical cut and paste procedure (Section~\ref{sec:conical-cut-paste}). We conclude Chapter~\ref{sec:building-cone-surfs} by proving that the samosa assemblies parameterize hyperbolic cone structures on $\AssembSphereExt{\Upsilon}$ via the realization map (Theorem~\ref{thm-intro:realization-map-local-homeo}). Hamantashen and hamantash assemblies are introduced in Chapter~\ref{sec:unfolding} as a special kind of samosas that can be unfolded.

We prove the main result of the paper (Theorem~\ref{thm-intro:surjectivity-holonomies}) in Chapter~\ref{chap:parametrizing-DT}. The proof unfolds into a bunch of results. We start by describing the image of the realization map restricted to hamantash assemblies (Theorem~\ref{thm:holonomy-hamantash}). We proceed with a proof of Theorem~\ref{thm-intro:no-overlap-no-degenerate-triangles} whose flavour contrasts with the other geometric arguments presented in this paper (Section~\ref{sec:no-overlap-no-degenerate-triangles}).
\subsection{Acknowledgments}
Many thanks to Ursula Hamenst\"{a}dt, Fanny Kassel, Andrea Monti, and Bram Petri for helpful guidance and fruitful conversations. We extend special appreciation to the organizers of the workshop ``Teichm\"{u}ller Theory: Classical, Higher, Super and Quantum'', held at Oberwolfach in August 2023. Their hospitality enabled us to make substantial progress on this work.

AF is grateful for the support of the ReNewQuantum collaboration. This paper is partly a result of the ERC-SyG project, Recursive and Exact New Quantum Theory (ReNewQuantum) which received funding from the European Research Council (ERC) under the European Union's Horizon 2020 research and innovation programme under grant agreement No 810573.

AM acknowledges partial funding by the Deutsche Forschungsgemeinschaft under Germany’s Excellence Strategy EXC2181/1 -- 390900948 (the Heidelberg STRUCTURES Excellence Cluster), the Collaborative Research Center SFB/TRR 191 -- 281071066 (Symplectic Structures in Geometry, Algebra and Dynamics), and the Research Training Group RTG 2229 -- 281869850 (Asymptotic Invariants and Limits of Groups and Spaces).
\section{Background}
\subsection{Review of cone metrics}\label{cone-review}
Let $\SurfMkSmooth$ be a smooth compact surface with a finite set $\mathcal{P}$ of marked points.
\begin{defn}\label{defn:cone-metric}
A \emph{cone metric} on $\SurfMkSmooth$ is a smooth Riemannian metric $h$ on the complement of the marked points which has a special form near each marked point. Near the marked point $p \in \mathcal{P}$, there must be a local chart $z \maps \SurfMkSmooth \to \C$ that sends the marked point to zero and gives
\[
h = e^{2u}\,|z^{-\nu_p/2\pi}\,dz|^2
\]
for some real number $\nu_p < 2\pi$ and some smooth function $u$. The number $\nu_p$ is called the {\em angle defect} of $h$ at the marked point $p$, and the positive number $\theta_p = 2\pi - \nu_p$ is called the {\em cone angle}. The marked points with nonzero angle defects are called the {\em conical singularities} of $h$.
\end{defn}

A cone metric is {\em hyperbolic} if it has constant curvature $-1$ away from the marked points. When $\SurfMkSmooth$ is equipped with a hyperbolic cone metric, it is called a {\em hyperbolic cone surface}. If we give $\SurfMkSmooth$ a hyperbolic cone metric $h$, we can find its area $\Vol(h)$ from the Gauss-Bonnet theorem:
\begin{align*}
\Vol(h) & = 4\pi(g-1) + \sum_{p \in \mathcal{P}} \nu_p \\
& = 4\pi(g-1) + 2\pi|\mathcal{P}| - \sum_{p \in \mathcal{P}} \theta_p,
\end{align*}
where $g$ is the genus of $\SurfMkSmooth$ and $\nu \in (-\infty, 2\pi)^{\mathcal{P}}$ is the tuple of angle defects of $h$. Since the cone angles $\theta_p$ are positive, this formula gives an upper bound on $\Vol(h)$, which is realized in the limit where the conical singularities sharpen to cusps.

If you prefer geometric topology to Riemannian geometry, you can also define a hyperbolic cone surface as a compact surface created by gluing finitely many hyperbolic triangles along isometries between their sides.

Troyanov showed that a negatively curved cone metric is determined uniquely by its curvature, conformal structure, and angle defects. Moreover, the curvature and conformal structure can be chosen freely, and the angle defects are constrained only by the Gauss-Bonnet formula, which becomes an inequality when we ask for the complement of the marked points to have negative total curvature. Here's what Troyanov's theorem says about hyperbolic cone metrics.
\begin{cor}[{\cite[Theorem~A]{Tro91}}]\label{cor:troyanov}
Let $\SurfMkSmooth$ be a smooth finite-type surface with genus $g$ and a finite set $\mathcal{P}$ of marked points. Choose a conformal structure on $\SurfMkSmooth$ and a tuple of angle defects $\nu \in (-\infty, 2\pi)^\mathcal{P}$ with
\[ 4\pi(g-1) + \sum_{p \in \mathcal{P}} \nu_p > 0. \]
There is a unique hyperbolic cone metric on $\SurfMkSmooth$ which induces the chosen conformal structure and has angle defect $\nu_p$ at each marked point $p \in \mathcal{P}$.
\end{cor}

Troyanov actually proved that one can chose any smooth negative function to be the curvature of the cone metric, not only the constant function~$-1$. The special case of Troyanov's theorem presented in Corollary~\ref{cor:troyanov} was proved independently by McOwen~\cite{McOw88}.

\subsection{Review of Deroin--Tholozan representations}\label{sec:dt-review}
Let $\SpherePk$ be an oriented topological sphere with $n \ge 3$ punctures. The representations of the fundamental group $\pi_1\SpherePk$ into $\operatorname{PSL}_2\R$ form a non-compact semi-algebraic variety. Even if you restrict your attention to representations with elliptic holonomies around the punctures, and you fix the rotation angle of the holonomy around each puncture, you might still expect every connected component of the resulting subvariety to be non-compact. In some cases, however, you'll find that one component is compact, against all odds. Deroin--Tholozan established the existence of these compact components for all $n\geq 4$~\cite{DeTh19}. The first examples were spotted by Benedetto--Goldman two decades earlier, in visualizations of the $n=4$ case~\cite{BenGol}.

To formally introduce these compact components, we first describe their setting in more detail. The punctures of $\SpherePk$ form a finite set $\mathcal{P}$. Fix a tuple of angles $\alpha \in (0,2\pi)^\mathcal{P}$, and consider the representations $\pi_1\SpherePk \to \operatorname{PSL}_2\R$ that send counterclockwise loops isolating puncture $p \in \mathcal{P}$ to elliptic elements with rotation angle\footnote{We say an elliptic element of $\operatorname{PSL}_2\R$ has rotation angle $\alpha \in (0, 2\pi)$ if it rotates the hyperbolic plane counterclockwise by angle $\alpha$ around its fixed point---or, equivalently, if it's conjugate to 
\[\pm\left[\begin{array}{cc}
    \hphantom{-}\cos(\alpha/2) & \sin(\alpha/2)\\ -\sin(\alpha/2)&\cos(\alpha/2)
\end{array}\right].\]} $\alpha_p$. We can recognize these representations just by checking one loop around each puncture, since all the counterclockwise loops isolating a given puncture are conjugate to each other. The conjugacy classes of these representations form the \emph{$\alpha$-relative character variety} $\Rep_\alpha(\SpherePk,\operatorname{PSL}_2\R)$.

If we combine the results of Deroin--Tholozan, supported by the observations of Benedetto--Goldman, with the work of Mondello on the topology of relative $\operatorname{PSL}_2\R$ character varieties~\cite[Corollary~4.17]{Mondello}, we get the following result.

\begin{thm}[\cite{DeTh19, BenGol, Mondello}]\label{thm:compact-comp}
The $\alpha$-relative character variety $\Rep_\alpha(\SpherePk,\operatorname{PSL}_2\R)$ has a compact connected component if and only if $\sum_{p \in \mathcal{P}} \alpha_p$ is in one of these sets:
\begin{center}
\begin{tabular}{ll}
$2\pi(0, 1)$ & small angle \\
$2\pi(n-1, n)$ & large angle \\
$2\pi\{1, 2, \ldots, n-1\}$ & whole angle.
\end{tabular}
\end{center}
When it exists, the compact component is unique.
\end{thm}

In the ``whole angle'' case, the compact component is an isolated point in $\Rep_\alpha(\SpherePk,\operatorname{PSL}_2\R)$: the conjugacy class where all elements of $\pi_1\SpherePk$ go to rotations around the same point in the hyperbolic plane $\HH$.

In the more interesting ``small angle'' and ``large angle'' cases, the compact component is denoted by $\RepDT{\alpha}(\SpherePk)$, and its elements are called {\em Deroin--Tholozan representations}---or {\em DT representations}, for short.\footnote{\emph{Supra-maximal} is another name for these representations that can also be found in the literature.} Deroin and Tholozan proved that $\RepDT{\alpha}(\SpherePk)$ is isomorphic, as a symplectic toric manifold, to $\CP^{n-3}$. DT representations are {\em totally elliptic}: they map every simple closed curve of $\SpherePk$ to an elliptic element of $\operatorname{PSL}_2\R$~\cite[Lemma~3.5]{DeTh19}. They are always Zariski dense and generically faithful, but almost never discrete. 

Building on the work of Deroin--Tholozan, Maret described explicit, global action-angle coordinates for $\RepDT{\alpha}(\SpherePk)$~\cite{action-angle}. The coordinate system depends on a choice of {\em geometric presentation} of $\pi_1 \SpherePk$ and a compatible choice of pants decomposition of $\SpherePk$.\footnote{We show an example of a geometric presentation in the proof of Proposition~\ref{prop:non-degenerate-pants-decomp}. The compatible pants decompositions are the ones obtained by playing the game described later in the proof, without worrying about trying to win.} A geometric presentation of $\pi_1 \SpherePk$ is one with generators $c_1,\ldots,c_n$ satisfying the relation $c_1\cdots c_n =1$, where each $c_i$ is the homotopy class of a loop around a different puncture. The \emph{standard pants decomposition} associated with a geometric presentation of $\pi_1\SpherePk$---the only one considered in~\cite{action-angle}---is the one whose pants curves are represented by the fundamental group elements $b_1, \ldots, b_{n-3}$ given by $b_i = (c_1 c_2 \cdots c_i c_{i+1})^{-1}$.
\begin{center}
\vspace{2mm}
\begin{tikzpicture}[scale=1.1, decoration={
    markings,
    mark=at position 0.6 with {\arrow{>}}}]
  \draw[postaction={decorate}] (0,-.5) arc(-90:-270: .25 and .5) node[midway, left]{$c_1$};
  \draw[black!40] (0,.5) arc(90:-90: .25 and .5);
  \draw[apricot, postaction={decorate}] (2,.5) arc(90:270: .25 and .5) node[midway, left]{$b_1$};
  \draw[lightapricot] (2,.5) arc(90:-90: .25 and .5);
  \draw[apricot, postaction={decorate}] (4,.5) arc(90:270: .25 and .5) node[midway, left]{$b_2$};
  \draw[lightapricot] (4,.5) arc(90:-90: .25 and .5);
  \draw[apricot, postaction={decorate}] (6,.5) arc(90:270: .25 and .5) node[midway, left]{$b_3$};
  \draw[lightapricot] (6,.5) arc(90:-90: .25 and .5);
  \draw[postaction={decorate}] (8,.5) arc(90:270: .25 and .5) node[midway, left]{$c_6$};
  \draw (8,.5) arc(90:-90: .25 and .5);
  
  \draw (.5,1) arc(180:0: .5 and .25) node[midway, above]{$c_2$};
  \draw[postaction={decorate}] (.5,1) arc(-180:0: .5 and .25);
  \draw (2.5,1) arc(180:0: .5 and .25)node[midway, above]{$c_3$};
  \draw[postaction={decorate}] (2.5,1) arc(-180:0: .5 and .25);
  \draw (4.5,1) arc(180:0: .5 and .25)node[midway, above]{$c_4$};
  \draw[postaction={decorate}] (4.5,1) arc(-180:0: .5 and .25);
  \draw (6.5,1) arc(180:0: .5 and .25)node[midway, above]{$c_5$};
  \draw[postaction={decorate}] (6.5,1) arc(-180:0: .5 and .25);
   
  \draw (0,.5) to[out=0,in=-90] (.5,1);
  \draw (1.5,1) to[out=-90,in=180] (2,.5);
  \draw (0,-.5) to[out=0,in=180] (2,-.5);
  
  \draw (2,.5) to[out=0,in=-90] (2.5,1);
  \draw (3.5,1) to[out=-90,in=180] (4,.5);
  \draw (2,-.5) to[out=0,in=180] (4,-.5);
  
  \draw (4,.5) to[out=0,in=-90] (4.5,1);
  \draw (5.5,1) to[out=-90,in=180] (6,.5);
  \draw (4,-.5) to[out=0,in=180] (6,-.5);
  
  \draw (6,.5) to[out=0,in=-90] (6.5,1);
  \draw (7.5,1) to[out=-90,in=180] (8,.5);
  \draw (6,-.5) to[out=0,in=180] (8,-.5);
\end{tikzpicture}
\vspace{2mm}
\end{center}

To find the coordinates of a representation $\rho \maps \pi_1\SpherePk \to \operatorname{PSL}_2 \R$, we start like this:
\begin{enumerate}[resume]
    \item Since DT representations are totally elliptic, $\rho(c_1), \ldots,\rho(c_n)$ and $\rho(b_1),\ldots,\rho(b_{n-3})$ act on $\HH$ by rotation. Consider their fixed points points $C_1,\ldots,C_n$ and $B_1,\ldots,B_{n-3}$, keeping in mind that some---but not all---might coincide.
    \item Each pair of pants is bounded by three curves, which are either pants curves or loops around punctures. Connect the fixed points that come from these three curves with geodesic segments. This gives a chain of $n-2$ hyperbolic triangles---one for each pair of pants. When two pairs of pants meet at a pants curve, the corresponding triangles share a vertex: the fixed point that comes from the common pants curve. The shared vertices are $B_1,\ldots,B_{n-3}$.
\end{enumerate}
\begin{center}
\begin{tikzpicture}[font=\sffamily,decoration={
    markings,
    mark=at position 1 with {\arrow{>}}}]]
    
\node[anchor=south west,inner sep=0] at (0,0) {\includegraphics[width=9cm]{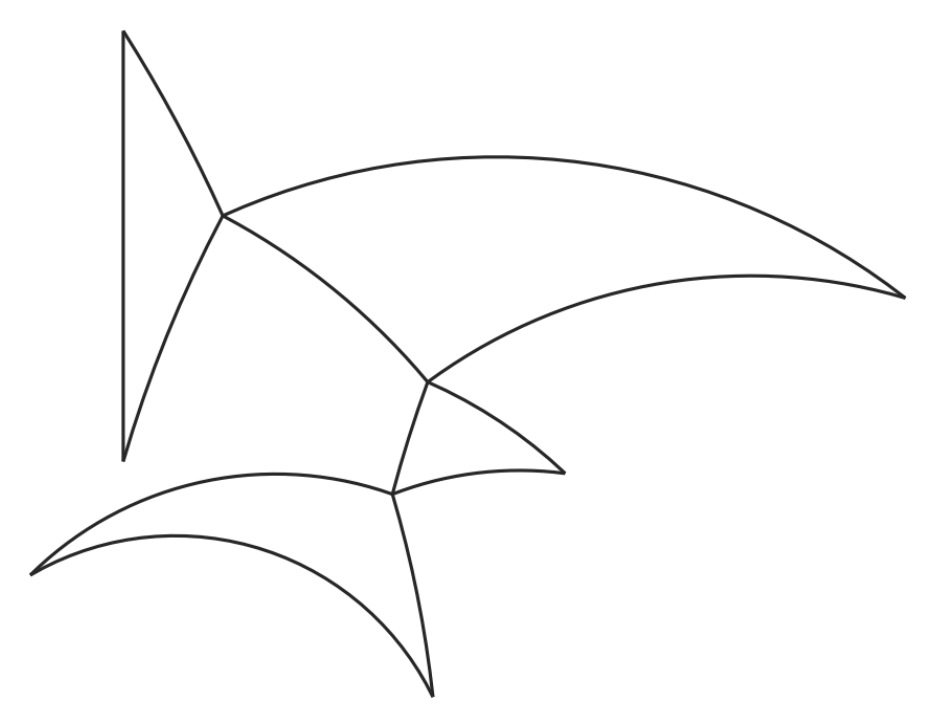}};

\begin{scope}
\fill (1.19,2.53) circle (0.07) node[left]{$C_1$};
\fill (1.19,6.62) circle (0.07) node[left]{$C_2$};
\fill (8.65,4.08) circle (0.07) node[below right]{$C_3$};
\fill (5.4,2.4) circle (0.07) node[below right]{$C_4$};
\fill (4.14,0.28) circle (0.07) node[below right]{$C_5$};
\fill (.31,1.43) circle (0.07) node[left]{$C_6$};
\end{scope}

\begin{scope}[apricot]
\fill (2.15,4.86) circle (0.07);
\fill (2.25,5.2) node{$B_1$};
\fill (4.1,3.27) circle (0.07) node[left]{$B_2$};
\fill (3.76,2.2) circle (0.07);
\fill (4.05,2.05) node{$B_3$};
\end{scope}
\end{tikzpicture}
\end{center}

We have used our pants decomposition to turn $\rho$ into a {\em triangle chain}. By construction, a triangle chain always has the following important geometric features~\cite[Corollary~3.6]{action-angle}:
\begin{itemize}
    \item The two interior angles at a shared vertex (which corresponds to a pants curve) sum to~$\pi$.
\end{itemize}
The next features are specific to the ``large angle'' case: 
\begin{itemize}
    \item The {\em terminal} triangles---the ones at the start and end of the chain---have vertex triples $(C_1, C_2, B_1)$ and $(B_{n-3},C_{n-1},C_n)$, respectively. These triples are ordered clockwise.
    \item The {\em intermediate} triangles---that is, all the others---have vertex triples $(B_i,C_{i+2},B_{i+1})$ for $i \in \{1,\ldots,n-4\}$. These triples are also ordered clockwise.
    \item The interior angle at a non-shared vertex (which corresponds to a loop around a puncture $p$) is $\pi-\alpha_p/2$.
\end{itemize}
In the ``small angle'' case, the same triples are ordered counterclockwise and the angle at a non-shared vertex is $\alpha_p/2$.

The action-angle coordinates of $\rho$ can now be obtained by measuring some angles in the triangle chain. At the fixed point $B_i$ of $\rho(b_i)$, which joins two consecutive triangles in the chain, we measure an action coordinate $\beta_i \in (0, 2\pi)$ and an angle coordinate $\gamma_i \in \R/2\pi\Z$. We get the action coordinates $\beta=(\beta_1,\ldots,\beta_{n-3})$ by writing the first interior angle adjacent to $B_i$ as $\pi-\beta_i/2$. That means the other interior angle adjacent to $B_i$ is $\beta_i/2$. Each of the angle coordinates $\gamma=(\gamma_1,\ldots,\gamma_{n-3})$ is the angle $\gamma_i$ between the two triangles that meet $B_i$.\footnote{To be precise, $\gamma_k$ is the angle $\angle(C_{i+1}, B_i, C_{i+2})$.}
\begin{center}
\begin{tikzpicture}[font=\sffamily,decoration={markings, mark=at position 1 with {\arrow{>}}}]
\node[anchor=south west,inner sep=0] at (0,0) {\includegraphics[width=9cm]{fig/fig-triangles-black}};

\begin{scope}
\fill (1.19,2.53) circle (0.07) node[left]{$C_1$};
\fill (1.19,6.62) circle (0.07) node[left]{$C_2$};
\fill (8.65,4.08) circle (0.07) node[below right]{$C_3$};
\fill (5.4,2.4) circle (0.07) node[below right]{$C_4$};
\fill (4.14,0.28) circle (0.07) node[below right]{$C_5$};
\fill (.31,1.43) circle (0.07) node[left]{$C_6$};
\end{scope}

\begin{scope}[apricot]
\fill (2.15,4.86) circle (0.07);
\fill (2.25,5.2) node{$B_1$};
\fill (4.1,3.27) circle (0.07) node[left]{$B_2$};
\fill (3.76,2.2) circle (0.07);
\fill (4.05,2.05) node{$B_3$};
\end{scope}

\draw[thick, postaction={decorate}, sky] (2.8,5.1) arc (10:115:.7) node[near end, above right]{$\gamma_1$};
\draw[thick, postaction={decorate}, sky] (4.7,3) arc (-20:32:.7) node[near start, above right]{$\gamma_2$};
\draw[thick, postaction={decorate}, sky] (3.95,1.6) arc (-80:15:.6) node[near end, below right]{$\gamma_3$};

\draw[thick, apricot] (1.9,5.3) arc (109:250:.5);
\draw[apricot] (2.4,4.2) node{\tiny $\pi-\beta_1/2$};
\draw[thick, apricot] (4.5,3.55) arc (25:129:.5) node[midway, above]{\small $\pi-\beta_2/2$};
\draw[thick, apricot] (4.25,2.35) arc (5:78:.4) node[at end, left]{\small $\pi-\beta_3/2$};
\end{tikzpicture}
\end{center}

The triangles in a chain can overlap; this causes no inconvenience. If we want to, however, we can get rid of overlaps between adjacent triangles by changing the geometric presentation of $\pi_1\SpherePk$ while keeping the same pants decomposition. This is done by applying Dehn twists along the pants curves to the generators $c_1,\ldots,c_n$, as explained in Proposition~\ref{prop:no-overlap}.

Triangles can also degenerate to single points; this must be handled more carefully. Fortunately, it only happens in rare cases: for each pants decomposition $\mathcal B$ that can be used to construct triangle chains, there's a dense open subset $\IntRepDT{\alpha}{\mathcal B}(\SpherePk)\subset \RepDT{\alpha}(\SpherePk)$ consisting of representations whose triangle chains are free of degenerate triangles. We can avoid degenerate triangles by choosing a different pants decomposition compatible with the same geometric presentation of $\pi_1\SpherePk$, as explained in Proposition~\ref{prop:non-degenerate-pants-decomp}. We can then switch the generating set such that the new pants decomposition is the standard one, as described in Proposition~\ref{prop:non-degenerate-standard-pants-decomp}.

\begin{cor}[{\cite[Theorem~A]{action-angle}}]\label{cor:action-angle}
The coordinates $(\beta,\gamma)$ completely parameterize $\IntRepDT{\alpha}{\mathcal B}(\SpherePk)$, turning it into a Lagrangian torus bundle over the interior of a polytope\footnote{The precise inequalities that describe $\Delta$ only depend on $\alpha$ and are written down in~\cite[Equation~3.11]{action-angle}. See also the discussion after Definition~\ref{defn:samosa-assembly}.} $\Delta\subset\R^{n-3}$:
\[
(\beta,\gamma)\colon \IntRepDT{\alpha}{\mathcal B}(\SpherePk) \overset{\cong}{\longrightarrow} \overset{\circ}{\Delta}\times (\R/2\pi\Z)^{n-3}.
\]
\end{cor}
Theorem~A in \cite{action-angle} shows that $(\beta,\gamma)$ extend to global action-angle coordinates for $\RepDT{\alpha}(\SpherePk)$. This involves extending the definitions of $\beta$ and $\gamma$ to the cases where triangles degenerate, allowing $\beta$ to take values on the boundary of $\Delta$. One can then promote the isomorphism from Corollary~\ref{cor:action-angle} to a symplectomorphism $\RepDT{\alpha}(\SpherePk) \to \CP^{n-3}$~\cite[Theorem~B]{action-angle}.

Changing the geometric presentation of $\pi_1\SpherePk$, as we'd do to get rid of overlaps between adjacent triangles, gives us a new but equally good action-angle coordinate system on $\RepDT{\alpha}(\SpherePk)$, which enjoys all the properties described in~\cite{action-angle}. Changing the pants decomposition, as we'd do to avoid degenerate triangles, should also give us a new action-angle coordinate system, good enough for our purposes in this paper. We expect most of the arguments in~\cite{action-angle} to generalize to all compatible pants decompositions.

\section{Parameterizing hyperbolic cone surfaces}\label{chap:param-hyp-cone-metrics}
\subsection{The space of hyperbolic cone metrics on a smooth sphere}\label{sec:space-of-hyperbolic-cone-metrics}
In Section~\ref{sec:dt-review}, we chose an oriented topological sphere $\SpherePk$ with a finite set $\mathcal P$ of $n\geq 3$ punctures. Add $n-3$ extra punctures, making a new punctured sphere $\SphereExtPk$ with an inclusion $\SphereExtPk \hookrightarrow \SpherePk$. It'll often be helpful to fill in the punctures, turning them into marked points. That gives us a forgetful map $\SphereExtMk \to \SphereMk$ of spheres with marked points, which erases the $n-3$ extra marked points on $\SphereExtMk$. Give $\SphereMk$ a smooth structure, which the other spheres inherit. We'll refer to the smooth versions of the spheres as $\SphereMkSmooth, \SphereExtMkSmooth, \SpherePkSmooth, \SphereExtPkSmooth$.

Fix a tuple of angle defects $\alpha \in (0,2\pi)^\mathcal{P}$ with
\begin{equation}\label{cond:defect-sum}
\sum_{p\in\mathcal P}\alpha_p > 2\pi(n-1).
\end{equation}
In other words, choose $\alpha$ from the ``large angle'' case of Theorem~\ref{thm:compact-comp}. Let $\HypCone{\alpha}(\SphereExtMkSmooth)$ be the space of hyperbolic cone metrics on $\SphereExtMkSmooth$, up to smooth isotopy, with cone angles $2\pi-\alpha_p$ at each of the original marked point $p\in\mathcal P$ and $4\pi$ at the $n-3$ extra marked points. We'll call these the {\em fractional} and {\em whole} singularities, respectively. Given a hyperbolic cone metric $h$ on $\SphereExtMkSmooth$, we'll use $[h]$ to denote its isotopy class, which is a point in $\HypCone{\alpha}(\SphereExtMkSmooth)$. A hyperbolic cone metric on $\SphereExtMkSmooth$ induces a \emph{branched hyperbolic structure}, as defined for instance in~\cite{Far21}, away from the fractional singularities.

Consider the smooth universal coverings $\SphereUniSmooth \to \SpherePkSmooth$ and $\SphereExtUniSmooth \to \SphereExtPkSmooth$. For each hyperbolic cone metric $h$ on $\SphereExtMkSmooth$, we can find a {\em developing map}: a local isometry $\dev_h \maps \SphereExtUniSmooth \to \HH$. The developing map is unique up to postcomposition by isometries of $\HH$, and it fully describes $h$. Precomposing $\dev_h$ by an isotopy of $\SphereExtMkSmooth$, lifted to $\SphereExtUniSmooth$, has the same effect as pulling $h$ back along that isotopy. With that in mind, consider all the developing maps of all the metrics representing elements of $\HypCone{\alpha}(\SphereExtMkSmooth)$. Postcomposition by isometries of $\HH$ and precomposition by isotopies of $\SphereExtMkSmooth$ lifted to $\SphereExtUniSmooth$ both act on this set, and the quotient can be identified with $\HypCone{\alpha}(\SphereExtMkSmooth)$. Following~\cite{cone-metric}, we give $\HypCone{\alpha}(\SphereExtMkSmooth)$ the quotient topology that comes from the $\mathcal{C}^\infty$ topology\footnote{ The \emph{compact-open $\mathcal{C}^\infty$ topology}, or in short \emph{$C^\infty$ topology}, on the set $\mathcal{C}^{\infty}(X,Y)$ of smooth maps between two smooth manifolds $X$ and $Y$ is generated by the following subbase. Let $f\in \mathcal{C}^{\infty}(X,Y)$, and let $(\varphi, U)$, $(\psi, V)$ be two charts on $X$ and $Y$. Let $K\subset U$ be a compact subset with $f(K)\subset V$, and let $\varepsilon >0$. Define a sub-basic neighbourhood of $f$ as the set of all $g\in \mathcal{C}^{\infty}(X,Y)$ such that $g(K)\subset V$ and such that the derivatives of $f$ and $g$ of every order are uniformly $\varepsilon$-close on $K$~\cite[Chap.~2]{hirsch-diff-topology}. The $\mathcal{C}^\infty$ topology on $\mathcal{C}^{\infty}(X,Y)$ shouldn't be confused with the finer Whitney $\mathcal{C}^\infty$ topology.} on the space of smooth maps $\SphereExtUniSmooth \to \HH$.

Let $\Teich(\SphereExtMkSmooth)$ be the {\em Teichm\"{u}ller space} of $\SphereExtMkSmooth$: the space of conformal structures on $\SphereExtMkSmooth$ up to isotopy. By taking the conformal class of a hyperbolic cone metric on $\SphereExtMkSmooth$ and then erasing the extra marked points, we get maps
\[ \HypCone{\alpha}(\SphereExtMkSmooth) \overset{\cong}{\longrightarrow} \Teich(\SphereExtMkSmooth) \longrightarrow \Teich(\SphereMkSmooth). \]
The invertibility of the first map is a consequence of Troyanov's theorem, via Corollary~\ref{cor:troyanov}.  The hypothesis of the corollary is Condition~\eqref{cond:defect-sum}, rewritten as $-4\pi + \sum_{p\in\mathcal{P}}\alpha_p - 2\pi(n-3)> 0$, and the conclusion is that each conformal structure on $\SphereExtMkSmooth$ comes from a unique cone metric representing an element of $\HypCone{\alpha}(\SphereExtMkSmooth)$. Further argument shows that $\HypCone{\alpha}(\SphereExtMkSmooth) \to\Teich(\SphereExtMkSmooth)$ is actually a homeomorphism~\cite[Theorem~3.2.8 and preceding paragraph]{cone-metric}.

By allowing the whole singularities to coalesce with the fractional singularities and each other, we can extend $\HypCone{\alpha}(\SphereExtMkSmooth)$ to a larger space $\AugHypCone{\alpha}(\SphereExtMkSmooth)$. Since we keep the fractional singularities away from each other, we still have a projection
\[ \AugHypCone{\alpha}(\SphereExtMkSmooth) \overset{\fconf}{\longrightarrow} \Teich(\SphereMkSmooth). \]
Like before, we topologize $\AugHypCone{\alpha}(\SphereExtMkSmooth)$ using the $\mathcal{C}^\infty$ topology on the space of developing maps, making the inclusion $\HypCone{\alpha}(\SphereExtMkSmooth) \hookrightarrow \AugHypCone{\alpha}(\SphereExtMkSmooth)$ continuous. Using Troyanov's theorem, we can identify each fiber of the map $\tau$ with $\Sym_{n-3}(\CP^1)$---the configuration space of $n-3$ points on $\CP^1$.\footnote{\label{foot:bundle} This identification is not canonical, because it involves choosing a Riemann surface isomorphism $\SphereMkSmooth \to \CP^1$ for each complex structure on $\SphereMkSmooth$. It should be possible to do the identification continuously across fibers, turning $\AugHypCone{\alpha}(\SphereExtMkSmooth)$ into a trivial topological bundle over $\Teich(\SphereMkSmooth)$ whose fibers are homeomorphic to $\Sym_{n-3}(\CP^1)$. However, we won't try to prove this.} It turns out that $\Sym_{n-3}(\CP^1)$ is a complex manifold biholomorphic to $\CP^{n-3}$.
                        
Each developing map is equivariant with respect to a unique representation $\pi_1\SphereExtPk \to \psl$, called the {\em holonomy} of $h$. To be precise, the holonomy is the representation $\rho_h$ defined by the following property: for each path $\widetilde{\zeta}$ on $\SphereExtUniSmooth$ that projects to an element $\zeta$ of $\pi_1\SphereExtPk$, the isometry $\rho_h(\zeta)$ sends the end point of $\dev_h \circ \widetilde{\zeta}$ to the start point. The holonomy around each whole singularity of $h$ is the identity, so the holonomy factors through the projection $\pi_1\SphereExtPk \to \pi_1\SpherePk$ to become a representation $\pi_1\SpherePk\to \psl$, which we also refer to as the holonomy of $h$. Different choices of developing map give the same holonomy up to conjugacy, so we can think of taking the holonomy as map $\HypCone{\alpha}(\SphereExtMkSmooth) \to \Rep_{\alpha}(\SpherePk, \psl)$. Its image, as Deroin and Tholozan show, falls within $\RepDT{\alpha}(\SpherePk)$~\cite[\S 4]{DeTh19}. Nothing goes wrong when the whole singularities coalesce with the fractional singularities or each other, so we end up with a map
\[ \AugHypCone{\alpha}(\SphereExtMkSmooth) \overset{\hol}{\longrightarrow} \RepDT{\alpha}(\SpherePk). \]
Together, the maps $\fconf$ and $\hol$ turn out to express $\AugHypCone{\alpha}(\SphereExtMkSmooth)$ as a product of sets~\cite[Equation~14 and Corollary~4.6]{DeTh19}:
\[ \AugHypCone{\alpha}(\SphereExtMkSmooth) \overset{\cong}{\longrightarrow} \Teich(\SphereMkSmooth) \times \RepDT{\alpha}(\SpherePk). \]
If the topological claim in footnote~\ref{foot:bundle} is true, then this bijection is actually a homeomorphism.

Our goal, as described in Section~\ref{sec:motivation}, is to construct a simple family of hyperbolic cone spheres whose holonomies can realize any DT representation. Our construction charts part of the space of hyperbolic cone metrics with coordinates that relate directly to Maret's action-angle coordinates.
\subsection{The space of hyperbolic cone structures on a topological sphere}\label{sec:hyperbolic-cone-structures}
In Section~\ref{sec:space-of-hyperbolic-cone-metrics}, we introduced the smooth structure $\SphereExtMkSmooth$ so that we could look at hyperbolic cone spheres from a Riemannian point of view---the perspective used most often in the literature. However, this smooth structure will be irrelevant to most of our reasoning, and the $\mathcal{C}^\infty$ topology on $\HypCone{\alpha}(\SphereExtMkSmooth)$ will be inconvenient to work with. For this reason, we will state and prove most of our results in terms of {\em hyperbolic cone structures} on the topological sphere $\SphereExtMk$, which form a space $\HypCone{\alpha}(\SphereExtMk)$ with a coarser and more convenient topology.
\subsubsection{Hyperbolic cone structures on topological surfaces}
Recall from Definition~\ref{defn:cone-metric} that for smooth cone metrics, the prototypical cone point of angle defect $\nu_0 < 2\pi$ is the origin in $\C$ with the Riemannian metric $|z^{-\nu_0/2\pi}\,dz|^2$.
\begin{defn}\label{defn:hyperbolic-cone-structure}
We define a {\em hyperbolic cone structure} on a topological surface as a maximal atlas of charts that each map an open subset of the surface to an open subset of~$\HH$ or a neighborhood of a prototypical cone point, subject to the condition that the transition maps between charts must be isometries.
\end{defn}
A hyperbolic cone structure gives a length metric on the domain of every chart~\cite[Chapter~I.3]{metric-non-pos}. Since these local metrics agree on overlaps, we can use them to measure arc lengths of paths, inducing a global length metric on each connected component of the surface. Since the transition maps are smooth, a hyperbolic cone structure on a topological surface $X$ also induces a smooth structure and a hyperbolic cone metric on $X$. This means that we can alternatively think of a hyperbolic cone structure on $X$ as an equivalence class of homeomorphisms from $X$ to hyperbolic cone spheres, where two homeomorphisms are equivalent if they differ by an isometry of hyperbolic cone spheres.

Consider a topological surface $\SurfMk$ with a hyperbolic cone structure. Let $\SurfPk \subset \SurfMk$ be the complement of the cone points, and let $\SurfUni \to \SurfPk$ be its universal covering. We can always find a local isometry $\SurfUni \to \HH$, which is called a {\em developing map} for the hyperbolic cone structure on $\SurfMk$. The developing map is unique up to postcomposition by isometries of $\HH$, and it fully describes the hyperbolic cone structure. It's equivariant with respect to a unique representation $\pi_1 \SurfPk \to \psl$, called the {\em holonomy} of the hyperbolic cone structure. If we use the induced smooth structure to think of $\SurfMk$ as a smooth surface with a hyperbolic cone metric, the topological developing maps discussed here are the same as the smooth developing maps discussed in Section~\ref{sec:space-of-hyperbolic-cone-metrics}.
\subsubsection{The space of hyperbolic cone structures}\label{sec:space-of-hyperbolic-cone-structures}
At the beginning of Section~\ref{sec:space-of-hyperbolic-cone-metrics}, we introduced an oriented topological sphere $\SphereExtMk$ with $2n - 3$ marked points: a set $\mathcal{P}$ of $n$ marked points and $n-3$ extra marked points. We also fixed a tuple of angle defects $\alpha \in (0,2\pi)^\mathcal{P}$ satisfying condition~\eqref{cond:defect-sum}. We define $\HypCone{\alpha}(\SphereExtMk)$ be to be the space of hyperbolic cone structures on $\SphereExtMk$, up to isotopy, with cone angles $2\pi-\alpha_p$ at each of the first $n$ marked points and $4\pi$ at the $n-3$ extra marked points. Given a hyperbolic cone structure $\HypConeStruct$ on $\SphereExtMk$, we'll use $[\HypConeStruct]$ to denote its isotopy class, which is a point of $\HypCone{\alpha}(\SphereExtMk)$.

Puncturing $\SphereExtMk$ at all $2n - 3$ of its marked points gives back the punctured sphere $\SphereExtPk$ introduced in Section~\ref{sec:space-of-hyperbolic-cone-metrics}. Like before, let $\SphereExtUni \to \SphereExtPk$ be the universal covering. Let $\mathcal{C}^0(\SphereExtUni, \HH)$ be the space of continuous maps $\SphereExtUni \to \HH$, equipped with the compact-open topology. Consider all the developing maps of all the hyperbolic cone structures representing elements of $\HypCone{\alpha}(\SphereExtMk)$. They form a subset $\mathcal{D} \subset \mathcal{C}^0(\SphereExtUni, \HH)$. Just like in the smooth case, postcomposition by isometries of $\HH$ and precomposition by isotopies of $\SphereExtMk$ lifted to $\SphereExtUni$ both act on $\mathcal{C}^0(\SphereExtUni, \HH)$, and the quotient of $\mathcal{D}$ by these actions can be identified with $\HypCone{\alpha}(\SphereExtMk)$. The quotient topology on $\mathcal{D}$ can thus be seen as a topology on $\HypCone{\alpha}(\SphereExtMk)$, which we'll call the {\em $\mathcal{C}^0$ topology}. Like we did for hyperbolic cone spheres, we can extend $\HypCone{\alpha}(\SphereExtMk)$ to a larger space $\AugHypCone{\alpha}(\SphereExtMk)$ by allowing whole singularities to coalesce with fractional singularities and each other. We topologize $\AugHypCone{\alpha}(\SphereExtMk)$ in the same way as $\HypCone{\alpha}(\SphereExtMk)$, using the $\mathcal{C}^0$ topology on the space of developing maps.

Recall from Section~\ref{sec:space-of-hyperbolic-cone-metrics} that $\SphereMk$ is the sphere we get from $\SphereExtMk$ by forgetting the $n-3$ extra marked points, and $\SpherePk$ is the punctured sphere we get from $\SphereMk$ by removing the marked points. We can define conformal structures on $\SphereMk$ in terms of homeomorphisms to conformal spheres with marked points, and we can think of hyperbolic cone structures on $\SphereExtMk$ in terms of homeomorphisms to hyperbolic cone spheres. Each hyperbolic cone sphere has the underlying structure of a conformal sphere with marked points, giving a continuous projection
\[
\AugHypCone{\alpha}(\SphereExtMk)\overset{\tau}{\longrightarrow}\Teich(\SphereMk)
\]
that forgets the extra marked points of $\SphereExtMk$. A hyperbolic cone structure on $\SphereExtMk$, like a hyperbolic cone metric on $\SphereExtMkSmooth$, comes with a holonomy representation, defining a continuous map
\[
\AugHypCone{\alpha}(\SphereExtMk)\overset{\hol}{\longrightarrow} \RepDT{\alpha}(\SpherePk).
\]
\subsubsection{Hyperbolic cone structures from hyperbolic cone metrics}
Every hyperbolic cone metric on $\SphereExtMkSmooth$ has a smooth developing map $\SphereExtMkSmooth \to \HH$, as discussed in Section~\ref{sec:space-of-hyperbolic-cone-metrics}. The developing map specifies a hyperbolic cone structure on the underlying topological sphere $\SphereExtMk$. This gives a map
\[ \HypCone{\alpha}(\SphereExtMkSmooth) \to \HypCone{\alpha}(\SphereExtMk), \]
which turns out to be a continuous inclusion. It's an inclusion---that is, an injective map---by the Myers-Steenrod theorem. It's continuous because each of the seminorms that defines the $\mathcal{C}^\infty$ topology can be expressed a $\mathcal{C}^0$ seminorm plus a series of seminorms of derivatives.

\begin{question}
Is this continuous inclusion a homeomorphism?
\end{question}

We can show that the inclusion is at least invertible as a map of sets. To invert it, we just need a way to take an isotopy class of homeomorphisms from $\SphereExtMk$ to a given hyperbolic cone sphere and upgrade it to an isotopy class of diffeomorphisms from $\SphereExtMkSmooth$ to the same hyperbolic cone sphere. There is indeed a way to do this, as explained for instance in~\cite[Section~1.4.2]{mcg-primer}. The question is whether the resulting inverse is continuous.

\subsection{Parameterization using triangulations}\label{sec:parametrization-using-triangulations}
Hyperbolic metrics on closed surfaces are determined by the geodesic lengths of finitely many simple closed curves. The minimal number of such curves for a closed surface of genus $g$ is $6g-5$ \cite{Sch93} (see also \cite[Theorem~10.7]{mcg-primer} for a parameterization with $9g-9$ simple closed curves). The construction in \cite{mcg-primer} generalizes to hyperbolic surfaces with cusps: a hyperbolic metric on a surface of genus $g$ with $m$ cusps is determined by the length of $9g-9+3m$ simple closed curves. Alternatively, Hamenst\"{a}dt showed that $6g-5+2m$ (non-closed) simple curves suffice \cite[Proposition~3.1]{Ham03}. In Hamenst\"{a}dt's construction, the curves are chosen to start and end in the same cusp and decompose the surface into ideal triangles and once-punctured disks.

Our study of $\HypCone{\alpha}(\SphereExtMkTop)$ requires an analogous parameterization of hyperbolic cone structures. Recall that a hyperbolic cone structure on $\SphereExtMkTop$ induces a smooth structure, turning $\SphereExtMkTop$ into a smooth sphere with a hyperbolic cone metric. Thus, in the presence of a hyperbolic cone structure, we can talk about geodesics on $\SphereExtMkTop$. However, we must keep in mind that a geodesic arc joining two cone points on a hyperbolic cone sphere might be {\em broken}, changing direction sharply where it passes through another cone point. To get around this problem, we produce local parameterizations of $\HypCone{\alpha}(\SphereExtMkTop)$ which are indexed by the triangulations of $\SphereExtMkTop$, rather than a global parameterization of the kind described earlier. We only consider triangulations where the set of vertices is the set of marked points of $\SphereExtMkTop$. Such a triangulation always consists of $4n-10$ triangles and $6n-15$ triangle sides.

Given a triangulation $\mathcal{T}$, we consider the subset $U_\mathcal{T}\subset \HypCone{\alpha}(\SphereExtMkTop)$ comprising all hyperbolic cone structures, up to isotopy, in which the edges of $\mathcal{T}$ are realized as unbroken geodesic arcs between cone points. Under each cone structure in $U_\mathcal{T}$, the faces of $\mathcal{T}$ are realized as non-degenerate hyperbolic triangles and lift to a geodesic triangulation $\widetilde{\mathcal{T}}$ of $\SphereExtUni$. Pick an edge of $\widetilde{\mathcal{T}}$ and map it isometrically to a hyperbolic segment into $\HH$. This map can be extended to a locally isometric developing map $\SphereExtUni\to \HH$. Unbroken geodesic segments remain unbroken under small deformations of the developing map in the space of continuous developing maps arising from hyperbolic cone structures on $\SphereExtMkTop$. This shows that $U_\mathcal{T}$ is an open subset of $\HypCone{\alpha}(\SphereExtMkTop)$. Every hyperbolic cone structure on $\SphereExtMkTop$ has a geodesic triangulation (see for instance~\cite[Proposition~3.1]{Thurston} for the analogue statement about Euclidean cone metrics on spheres), so the open sets $U_\mathcal T$ give an open covering of $\HypCone{\alpha}(\SphereExtMkTop)$.

Define a map $\varphi_\mathcal{T}\colon U_\mathcal{T}\to \R_{>0}^{6n-15}$ by sending an isotopy class of hyperbolic cone structures on $\SphereExtMkTop$ to the vector listing the lengths of the geodesic segments realizing the edges of $\mathcal T$. Since a hyperbolic triangle is uniquely determined, up to orientation, by the lengths of its sides, the map $\varphi_{\mathcal T}$ is injective.
\begin{prop}\label{length-homeo}
The maps $\varphi_{\mathcal{T}}$ are homeomorphisms onto their images.
\end{prop}
\begin{proof}
Fix a triangulation $\mathcal T$ of $\SphereExtMkTop$ and pick a hyperbolic cone structure $\HypConeStruct$ whose isotopy class $[\HypConeStruct]$ lies in $U_\mathcal{T}$. By the definition of $U_\mathcal{T}$, the edges of $\mathcal{T}$ snap tight to unbroken geodesic arcs on the hyperbolic cone sphere determined by $[\HypConeStruct]$. Like we did before, lift $\mathcal{T}$ to a geodesic triangulation $\widetilde{\mathcal T}$ of~$\SphereExtUni$ in which each triangle is mapped isometrically into $\HH$ by the developing map $\dev_{\HypConeStruct}\colon \SphereExtUni\to \HH$.

We start by proving that $\varphi_\mathcal{T}\colon U_\mathcal{T}\to \R_{>0}^{6n-15}$ is continuous at $[\HypConeStruct]$. Pick an edge $e$ of $\mathcal T$ and an edge $\tilde e$ of $\widetilde{\mathcal T}$ that lifts $e$. Let $\pi_e\colon \R_{>0}^{6n-15}\to \R_{>0}$ be the projection that picks out the $e$ component, so $\pi_e\circ\varphi_\mathcal{T}([\HypConeStruct])$ is the $\HypConeStruct$-length of $e$. To show that the $e$ component of $\varphi_\mathcal{T}([\HypConeStruct])$ depends continuously on $[\HypConeStruct]$, recall that $\dev_{\HypConeStruct}$ is a local isometry. The length of the geodesic arc $\dev_\HypConeStruct(\tilde e)$ depends continuously on its two endpoints in $\HH$, and therefore depends continuously on the developing map in the $\mathcal{C}^0$ topology, because convergence in this topology implies pointwise convergence of the images of the cone points.

We now prove that the inverse map is continuous at $\ell=\varphi_{\mathcal T}([\HypConeStruct])$. For every $\ell'\in\varphi_{\mathcal T}(U_{\mathcal T})$ close to $\ell$, our plan is to deform $\dev_{\HypConeStruct}$ into a developing map $\dev_{\HypConeStruct'}$ of a new hyperbolic cone structure $\HypConeStruct'$ representing the isotopy class $\varphi_\mathcal{T}^{-1}(\ell')$. We'll then prove that $\dev_{\HypConeStruct'}$ converges to $\dev_\HypConeStruct$ in the $\mathcal{C}^0$ topology as $\ell'$ goes to $\ell$.

We'll build $\dev_{\HypConeStruct'}$ in pieces, using diffeomorphisms $f_{\widetilde{T},\ell'}$ indexed by the faces $\widetilde{T}$ of the triangulation $\widetilde{\mathcal{T}}$. Let $T$ be the image of $\widetilde T$ under $\dev_{\HypConeStruct}$. The side lengths of $T$ are given by three components of $\ell$. For every $\ell'\in\varphi_{\mathcal T}(U_{\mathcal T})$ close to $\ell$, pick some hyperbolic triangle $T'\subset \HH$ whose side lengths are given by the same three components of $\ell'$. We can get a diffeomorphism $f_{\widetilde{T}, \ell'} \maps T \to T'$ that maps edges to edges by identifying points with the same barycentric coordinates, as defined in Defintion~\ref{defn:barycentric-coordinates} from Appendix~\ref{apx:barycentric-coordinates}. By composing $f_{\widetilde{T}, \ell'}$ with the inclusion $T' \hookrightarrow \HH$, we can think of it as a continuous map $T \to \HH$, which depends continuously on the vertices of $T'$ with respect to to the $\mathcal{C}^0$ topology on $\mathcal{C}^0(T, \HH)$. This continuity property follows from the continuity of the barycentric parameterization with respect to a triangle's vertices. It implies that $f_{\widetilde{T}, \ell'}\colon T\to \HH$ converges to $T\hookrightarrow \HH$ in $\mathcal{C}^0(T, \HH)/\psl$ as $\ell'$ approaches $\ell$.

For every $\ell'$, we build a developing map $\dev_{\HypConeStruct'}$ by patching together the maps $f_{\widetilde T, \ell'}\circ\dev_\HypConeStruct$ over all the triangles $\widetilde T$. We start with an arbitrary first triangle $\widetilde T_0$ of $\mathcal{\widetilde{T}}$, defining $\dev_{\HypConeStruct'}$ on $\widetilde T_0$ to be $f_{\widetilde T_0, \ell'}\circ\dev_{\HypConeStruct}$. Now, for any other triangle $\widetilde T_1$ that shares an edge with $\widetilde T_0$, there's a unique isometry $A_{\widetilde{T_1},\ell'}$ of $\HH$ for which $T_0'$ and $A_{\widetilde{T_1},\ell'} T_1'$ share the same edge. We can extend $\dev_{\HypConeStruct'}$ continuously to $\widetilde T_1$ by $A_{\widetilde{T_1},\ell'}\circ f_{\widetilde T_1, \ell'}\circ\dev_{\HypConeStruct}$, because when two hyperbolic triangles share an edge, their barycentric parameterizations match along that edge. By repeating this patching procedure across every edge of $\widetilde{\mathcal{T}}$, we get a continuous map $\dev_{\HypConeStruct'}\colon \SphereExtUni\to \HH$.\footnote{With respect to the smooth structure on $\SphereExtUni$ induced by the hyperbolic cone structure $\HypConeStruct$, the map $\dev_{\HypConeStruct'}$ will be smooth on the interior of each triangle in $\widetilde{\mathcal{T}}$. However, we can't expect $\dev_{\HypConeStruct'}$ to be globally smooth, because the gluing between triangles is only smooth in exceptional cases.} By construction, $\dev_{\HypConeStruct'}$ is a developing map for a hyperbolic cone structure $\HypConeStruct'$ with length vector $\ell'$, and it converges to $\dev_{\HypConeStruct}$ in $\mathcal{C}^0(\SphereExtUni,\HH)/\psl$ as $\ell'$ approaches $\ell$.
\end{proof}
\begin{rem}
The sets $U_\mathcal{T}$ also define an open covering of $\HypCone{\alpha}(\SphereExtMkSmooth)$ equipped with the $\mathcal{C}^\infty$ topology from Section~\ref{sec:parametrization-using-triangulations}. The maps $\varphi_\mathcal{T}$ can be defined analogously and remain injective and continuous with respect to the $\mathcal{C}^\infty$ topology on $U_{\mathcal{T}}$. However, it's not clear to us that their inverses should be continuous in the finer $\mathcal{C}^\infty$ topology. The developing map $\dev_{\HypConeStruct'}$ that we constructed in the proof of Proposition~\ref{length-homeo} by patching up triangles gives a piecewise smooth map $\SphereExtUniSmooth\to \HH$, but this map generally isn't globally smooth, because gluing along the edges doesn't extend it smoothly across triangles in general.
\end{rem}

The atlas $\{(U_{\mathcal{T}},\varphi_{\mathcal{T}})\}$, indexed by the triangulations of $\SphereExtMkTop$, give $\HypCone{\alpha}(\SphereExtMkTop)$ the structure of a topological manifold. We believe we can give $\HypCone{\alpha}(\SphereExtMkTop)$ a smooth manifold structure that makes each map $\varphi_\mathcal{T}$ a diffeomorphism onto its image. One approach is to first allow the angle defects to vary, defining an augmented space $\HypConeFree(\SphereExtMkTop)$ that comprises all of the spaces $\HypCone{\alpha}(\SphereExtMkTop)$. The maps $\varphi_\mathcal{T}$ extend naturally to $\HypConeFree(\SphereExtMkTop)$, and our proof of Proposition~\ref{length-homeo} generalizes to this extension. The maps $\varphi_\mathcal{T}$ should form a smooth atlas for $\HypConeFree(\SphereExtMkTop)$, giving it a smooth manifold structure. To see that the transition maps are smooth, first recall that any two triangulations of $\SphereExtMkTop$ are connected by a finite sequence of edge flips~\cite[Proposition~3.8]{triangulations}. That means we just need to show that the coordinates change smoothly when we flip an edge. We can do this by observing that the length of a diagonal of a hyperbolic quadrilateral depends smoothly on the lengths of the other diagonal and the four sides.

We can now see each $\HypCone{\alpha}(\SphereExtMkTop)$ as a codimension-$(2n-3)$ slice of $\HypConeFree(\SphereExtMkTop)$, cut out by the equations that make the face angles of a chosen triangulation $\mathcal{T}$ sum to the desired cone angle around each vertex. Since the face angles depend on the edge lengths through the law of cosines, these $2n-3$ equations are non-linear but smooth. For at least some values of $\alpha$, the implicit function theorem should guarantee that $\HypCone{\alpha}(\SphereExtMkTop)$ is a smooth submanifold of $\HypConeFree(\SphereExtMkTop)$, giving it a smooth structure in which $\varphi_\mathcal{T}$ is a diffeomorphim onto its image. This argument works for any triangulation we pick, so it actually gives a smooth structure in which every map $\varphi_\mathcal{T}$ is a diffeomorphism onto its image.
\begin{conj}\label{smooth-atlas}
Under some constraint on cone angles, we can give $\HypCone{\alpha}(\SphereExtMkTop)$ a smooth manifold structure that makes each map $\varphi_\mathcal{T}$ a diffeomorphism onto its image. In particular, we expect this to happen when the first tuple of angle defects $\alpha\in (0, 2\pi)^\mathcal{P}$ satisfies
\[ \sum_{p\in\mathcal{P}}\alpha_p > 2\pi(n-1), \]
as we're assuming; the cone angles at the last $n-3$ marked points are near $4\pi$; and the inequality from Corollary~\ref{cor:troyanov} is satisfied.
\end{conj}

\section{Building hyperbolic cone spheres}\label{sec:building-cone-surfs}
\subsection{Overview}
We present a way of building hyperbolic cone spheres by gluing together elementary pieces, called samosas. The gluing is done along slits that we cut open on each samosa, as described in Section~\ref{sec:conical-cut-paste}. The original inspiration behind this construction was Maret's triangle chain parameterization of DT representations. We'll see that by constructing hyperbolic cone spheres in this way, we get good control on their holonomy. Furthermore, it is general enough to chart the whole space of hyperbolic cone structures on a model punctured sphere.
\subsection{Building blocks}\label{sec:building-blocks}
Let's start with the simplest example of a hyperbolic cone surface.
\begin{defn}\label{defn:samosa}
A {\em samosa} is a hyperbolic sphere with three conical singularities, which we'll call its {\em corners}. The sphere is oriented, and the corners come with a cyclic order.
\end{defn}
\begin{center}
\includegraphics[width=8cm]{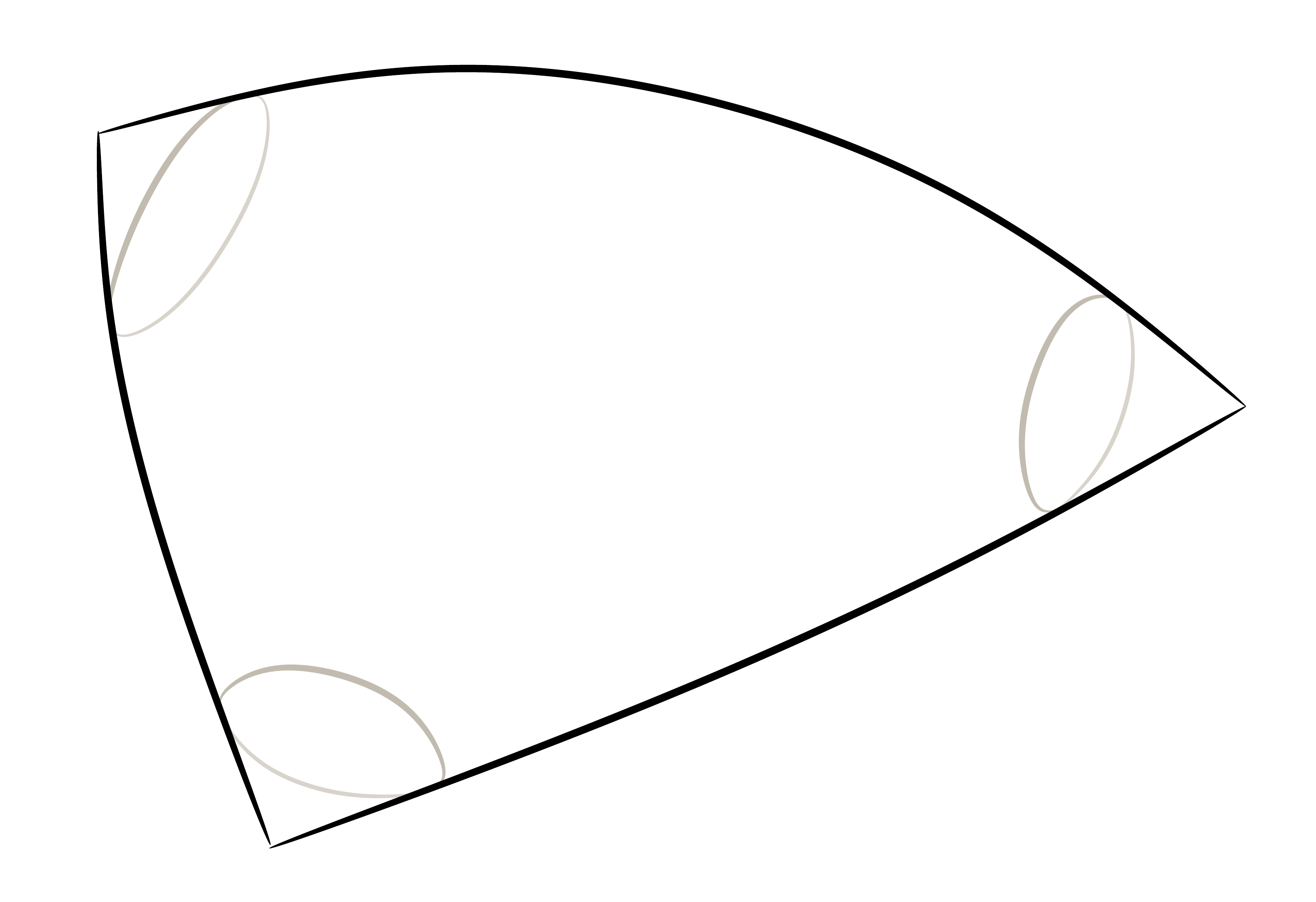}
\end{center}
By the Gauss-Bonnet theorem, the angles of the conical singularities must sum to less than $2\pi$.\footnote{If the sum is equal to $2\pi$, then the samosa is flat, instead of hyperbolic.}  By Corollary~\ref{cor:troyanov}, these {\em corner angles} determine the samosa uniquely, up to isometry. We can build the samosa with corner angles $\theta = (\theta_1, \theta_2, \theta_3)$ by gluing two hyperbolic triangles with interior angles $\theta/2$ along their edges. This construction is the reverse of an intrinsic decomposition of a samosa. A three-punctured sphere has one simple curve, up to isotopy, between each pair of punctures~\cite[Fig.~3]{arc+curve}, so a samosa has one simple geodesic between each pair of corners~\cite[\S 1.2.7]{mcg-primer}. The corners and the simple geodesics between them form a topological circle, which we'll call the samosa's {\em equator}. The equator is oriented to match the cyclic order of the corners. Cutting along the equator decomposes the samosa into a pair of hyperbolic triangles, which are isometric---with opposite orientations---because their sides have equal lengths. We'll call these triangles the samosa's {\em hemispheres}. The equator runs clockwise around the {\em northern} hemisphere, and counterclockwise around the {\em southern} hemisphere. At corner $j$, each hemisphere has a vertex with interior angle $\theta_j/2$.
\begin{center}
\begin{tikzpicture}
\node at (0, 0) {\includegraphics[width=8cm]{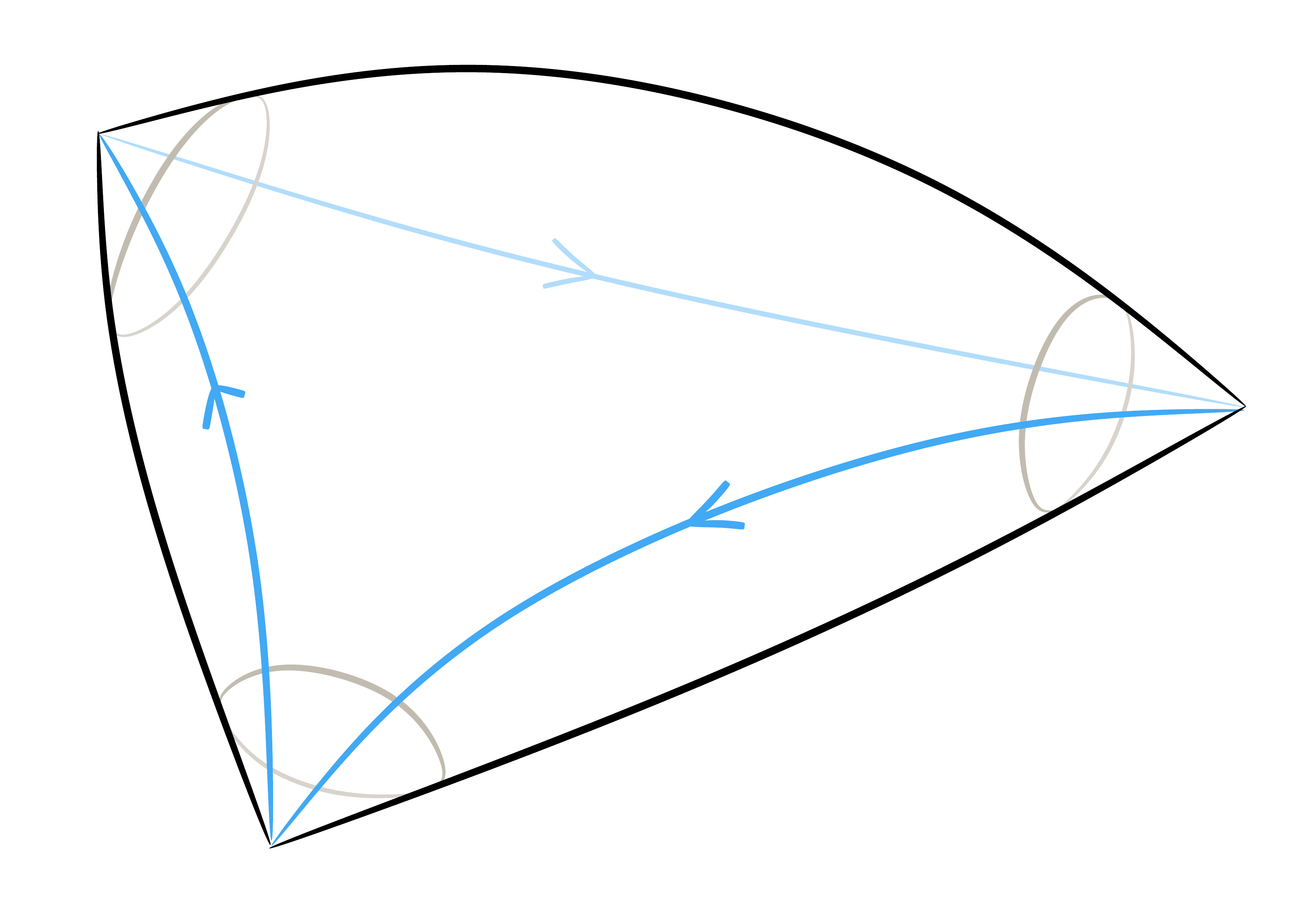}};
\node (N) at (3.9,2.5) {\text{northern hemisphere}};
\draw (N.west) edge[out=180,in=65,->] (0.8, 1.6);
\node (S) at (3.5,-2) {\text{southern hemisphere}};
\draw (S.west) edge[out=180,in=-60,->] (0.2,-1.0);
\end{tikzpicture}
\end{center}

\begin{defn}\label{defn:slit}
A {\em slit} on a samosa is a geodesic segment that starts at a corner and stays within a hemisphere.
\end{defn}
A typical slit ends in the interior of a hemisphere. However, a slit can also end on the equator or at a corner. A slit starting at corner $j$ is parameterized by its length $\ell > 0$ and its angle $\phi$ from the negative side of the equator. The angle $\phi$ is a circle-valued parameter inside $\R/\theta_j\Z$. We usually think of $\phi$ as a number in $[-\tfrac{1}{2} \theta_j, \tfrac{1}{2} \theta_j]$. We call $\ell$ the \emph{slit length} and $\phi$ the \emph{slit angle}.
\begin{center}
\includegraphics[width=8cm]{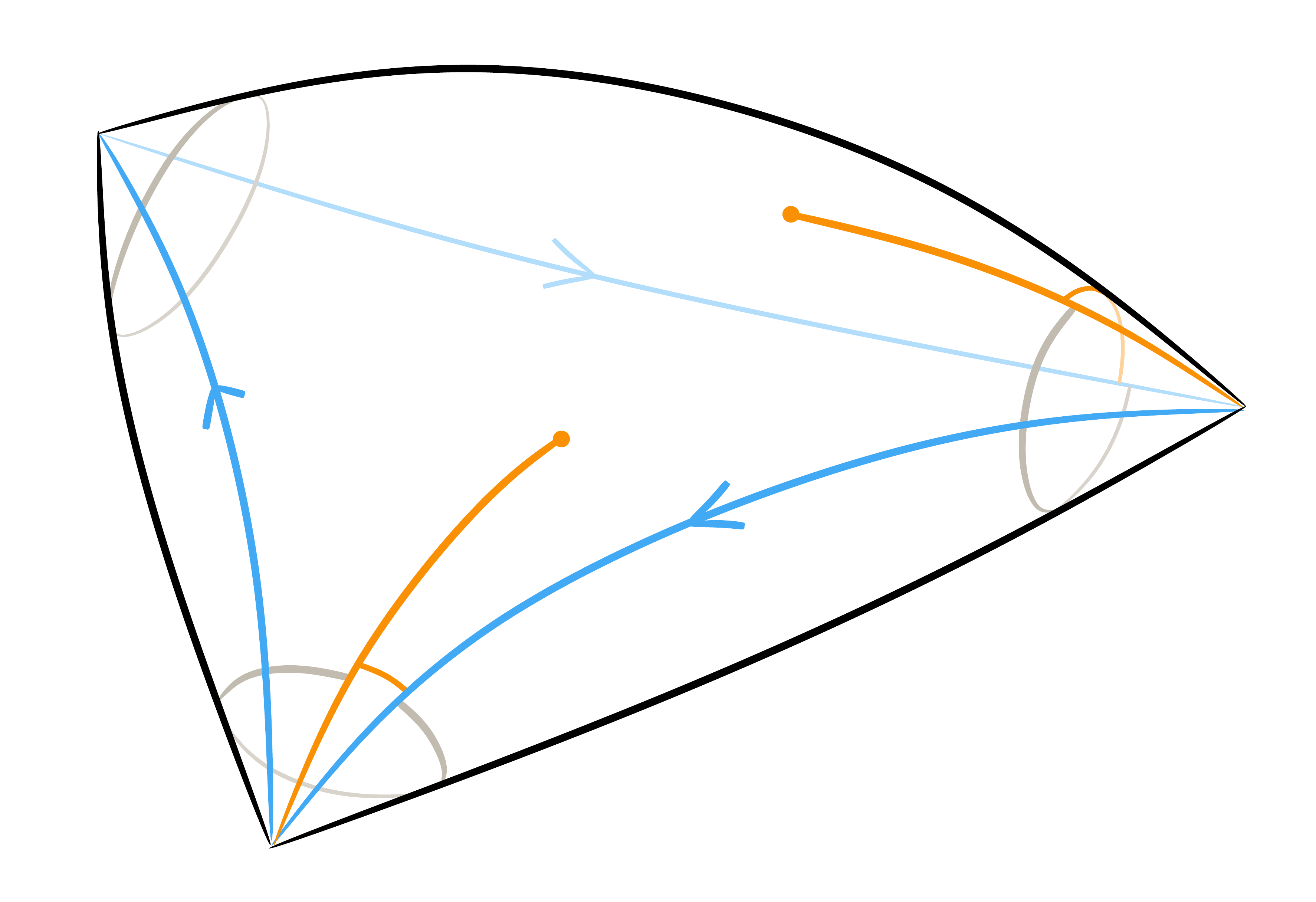}
\end{center}
When $\phi = 0$, the slit runs along the equator in the negative direction. A slit in the northern hemisphere has $\phi \ge 0$, and a slit in the southern hemisphere has $\phi \le 0$.

\begin{lem}\label{lem:bound-on-l-one-slit}
The slit at corner $j$ stays within one hemisphere as long as its length doesn't exceed a certain bound $\ell_\text{max}$, which depends on the slit angle $\phi_j$ and the shape of the samosa. Explicitly, $\ell_\text{max}$ is given by the formula
\[
\coth(\ell_{max})=\frac{1}{\sinh(d_{j-1})}\big(\cos(\phi_j)\cosh(d_{j-1})+|\sin(\phi_j)|\cot(\theta_{j+1}/2)\big),
\]
where $d_{j-1}>0$ is the length of the equator segment joining the corners $j$ and $j+1$, and $\theta_{j+1}/2$ is the angle of the corner $j+1$.
\end{lem}
\begin{proof}
We apply the four-parts formula---formula~\eqref{eq:four-part-formula} from Appendix~\ref{apx:trig-formulae}---to the triangle determined by the slit and the equator segment joining the corners $j$ and $j+1$. In this triangle, sides of length $\ell_{max}$ and $d_{j-1}$ meet at an angle of $\phi_j$. The side of length $\ell_{max}$ is opposite to the corner $j+1$, which has angle $\theta_{j+1}/2$. The absolute value on $\sin(\phi_j)$ makes the formula valid no matter which hemisphere the slit lies in.
\begin{center}
\begin{tikzpicture}[scale=.9, every node/.style={inner sep=0.5mm}]
\node[anchor=south west, inner sep=0mm] at (-3.21, -0.03) {\reflectbox{\includegraphics[width=4.05cm, angle=-90]{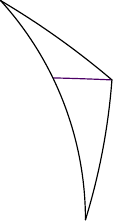}}};
\coordinate (A) at (5.59, 1.06);
\coordinate (B) at (-3.18, 4.46);
\coordinate (M) at (-0.08, 2.36);
\draw ++(14:0.3) arc (14:88:0.3) node[midway, anchor=south west] {$\phi_j$};
\draw (A) ++(180:1) arc (180:193:1) node[midway, inner sep=2.3mm, anchor=5] {$\theta_{j+1}/2$};
\path (0, 0) -- (M) node[plum, inner sep=1mm, pos=0.55, anchor=west] {$\ell_\text{max}$};
\path (0, 0) edge[draw=none, bend right=6] node[inner sep=1.2mm, midway, anchor=120] {$d_{j-1}$} (A);
\node[anchor=north, inner sep=1mm] {$j$};
\node[anchor=west, inner sep=1mm] at (A) {$j+1$};
\node[anchor=south, inner sep=1mm] at (B) {$j-1$};
\end{tikzpicture}
\end{center}
\end{proof}

The hyperbolic law of cosines---formula~\eqref{eq:hyperbolic-law-of-cosines-angles} from Appendix~\ref{apx:trig-formulae}---expresses $\cosh(d_{j-1})$ in terms of the corner angles of the samosa. So, it's also possible to formulate the bound from Lemma~\ref{lem:bound-on-l-one-slit} in terms of the slit angle $\phi_j$ and the corner angles instead of $d_{j-1}$.

\begin{defn}\label{defn:slit-samosa}
A {\em slit samosa} is a samosa with one to three slits, which start at different corners and don't collide. That means none of the slits passes through or ends on another slit.

The slits are allowed to share endpoints, and to end at the corners. We'll say the slit samosa is {\em degenerate} when that happens.
\end{defn}

\subsection{Assembly instructions}\label{sec:assembly-instructions}
We can build more complicated hyperbolic cone spheres by gluing slit samosas together. The gluing operation we'll use is known as the \emph{conical cut and paste}---described, for example, in \cite[§2]{branch-sing}. Section~\ref{sec:conical-cut-paste} explains how to apply this operation to samosas. Gluing data will be organized into an object called a {\em samosa assembly}, defined in Section~\ref{sec:assembly-instructions-parametrization}. The gluing process, described in Section~\ref{sec:assembly-instructions-realization}, can be seen as a geometric realization that turns abstract gluing data into an actual hyperbolic cone sphere. When two samosa assemblies have the same combinatorics, their realizations are typically identified through a distinguished homotopy class of homeomorphisms, as described in Section~\ref{sec:marking}.

\subsubsection{Conical cut and paste for samosas}\label{sec:conical-cut-paste} To join two slit samosas using a {\em conical cut and paste}, we start by cutting along two slits of equal length, one on each of the two samosas. The incision runs from the origin of the slit to its end point, separating it into a left lip and a right lip. We then glue each slit's left lip to the other slit's right lip. In doing so, we trade the conical singularities of angles $\theta$ and $\theta'$ at the slit corners for a conical singularity of angle $\theta+\theta'$ at the origin of the slit and another conical singularity of angle $4\pi$ at the end of the slit. Later, in Section~\ref{sec:assembly-instructions-realization}, we'll primarily be interested in the case where $\theta+\theta'=2\pi$. In that case, the conical cut and paste procedure turns two samosas into one hyperbolic cone sphere with four fractional singularities and one whole singularity. 

\subsubsection{Parameterization}\label{sec:assembly-instructions-parametrization}
Combinatorially, the gluing of a collection of slit samosas is described by a genus-0 {\em pants assembly:} a bunch\footnote{We have finite sets in mind, but our construction should work for a countably infinite sets too. Countability ensures that the pants assembly describes a second-countable surface.} of pairs of pants with a relation that says certain pairs of cuffs should be glued together, forming a connected surface of genus 0. The glued cuffs become pants curves that form a pants decomposition of the assembled surface, and the unglued cuffs are contracted to marked points. To get a sphere with $n$ marked points, we have to glue $n-2$ pairs of pants along $n-3$ pants curves.

For convenience, we'll give our pants assemblies some additional structure.
\begin{itemize}
\item Each pair of pants comes with an orientation and a cyclic ordering of its cuffs. The pants are meant to be glued with matching orientations.
\item The pants curves on the assembled surface are oriented. Since the assembled surface is oriented too, this gives each pants curve a left side and a right side.
\end{itemize}

\begin{defn}\label{defn:chained-pants-decomp}
We'll say a pair of pants in a pants decomposition is \emph{terminal} if two of its cuffs are unglued. Any pants decomposition has at least two terminal pairs of pants. A pants decomposition with only two terminal pairs of pants will be called a \emph{chained} pants decomposition.
\end{defn}
An example of a chained pants decomposition is the standard pants decomposition associated with a geometric presentation of $\pi_1\SpherePk$, introduced in Section~\ref{sec:dt-review}.

\begin{defn}\label{defn:samosa-assembly}
A {\em samosa assembly} is a genus-0 pants assembly with the following extra information, which tells us how to realize the pairs of pants as slit samosas glued along slits.
\begin{itemize}
    \item A tuple of angle defects $\alpha \in (0, 2\pi)^\mathcal{P}$, where ${\mathcal P}$ is the set of marked points on the assembled sphere. When we realize the pairs of pants as samosas, each marked point $p$ becomes an un-slit corner with cone angle $2\pi - \alpha_p$.
    \item A tuple of angles $\beta \in (0, 2\pi)^{\mathcal B}$, where ${\mathcal B}$ is the set of pants curves on the assembled sphere. When we realize the pairs of pants as samosas, the cuff to the left of pants curve $b$ becomes a corner with cone angle $\beta'_b := 2\pi - \beta_b$, and the cuff to the right becomes a corner with cone angle $\beta_b$.
    \item A tuple of slit lengths $\ell \in (0, \infty)^{\mathcal B}$. When we realize the pairs of pants as slit samosas, the two corners coming from the cuffs on both sides of the pants curve $b$ become slit corners with slits of common length $\ell_b$.
    \item A tuple of slit angles $\phi' \in \prod_{b \in {\mathcal B}} (\R/\beta_b'\Z)$ for the cuffs to the left of the pants curves.
    \item A tuple of slit angles $\phi \in \prod_{b \in {\mathcal B}} (\R/\beta_b\Z)$ for the cuffs to the right of the pants curves.
\end{itemize}
These parameters must satisfy the following conditions.
\begin{itemize}
    \item (Cone angle condition) The three cone angles attached to each pair of pants sum to less than $2\pi$. That makes them the corner angles of a unique samosa, as discussed after Definition~\ref{defn:samosa}.
    \item (Slit condition) Once each pair of pants is realized as a slit samosa, its slit parameters specify valid slits that don't collide, as described in and between Definitions~\ref{defn:slit} and~\ref{defn:slit-samosa}. See also Lemma~\ref{lem:bound-on-l-one-slit}. 
\end{itemize}

 A samosa assembly is \emph{degenerate} if it includes a pair of pants which is realized by a degenerate slit samosa, as defined in Definition~\ref{defn:slit-samosa}. A samosa in a samosa assembly is said to be \emph{terminal} if it realizes a terminal pair of pants of the underlying pants assembly, as defined in Definition~\ref{defn:chained-pants-decomp}.
\end{defn}
The cone angle condition takes different forms depending on how the pants cuffs are glued. For example, a pair of pants with two unglued cuffs contracted to marked points $p,q\in \mathcal{P}$ and one cuff glued to the right side of the pants curve $b \in {\mathcal B}$ imposes the inequality
\[(2\pi - \alpha_p) + (2\pi - \alpha_q) + \beta_b < 2\pi. \]
A pair of pants with one unglued cuff contracted to a marked point $p\in\mathcal{P}$, one cuff glued to the left side of $b \in {\mathcal B}$, and one cuff glued to the right side of $a \in {\mathcal B}$ imposes the inequality
\[ (2\pi - \alpha_p) + \beta'_b + \beta_a < 2\pi. \]
Summing these inequalities over all $n-2$ pairs of pants, and recalling that $\beta_b + \beta'_b = 2\pi$ for each $b \in {\mathcal B}$, we get Condition~\eqref{cond:defect-sum}:
\[ 2\pi n - \sum_{p\in\mathcal{P}}\alpha_p + 2\pi(n-3) < 2\pi(n-2). \]

Let $\Assemb_{\alpha, \Upsilon}$ be the space of non-degenerate samosa assemblies with angle defects $\alpha$ and underlying pants assembly $\Upsilon$.\footnote{As we just saw, the definition of a samosa assembly forces $\alpha$ to satisfy Condition~\eqref{cond:defect-sum}. That means $\Assemb_{\alpha, \Upsilon}$ is empty unless Condition~\eqref{cond:defect-sum} holds.} The space $\Assemb_{\alpha, \Upsilon}$ is paramaterized by the $4(n-3)$ parameters $\beta_b, \ell_b, \phi_b, \phi'_b$ indexed by the pants curves $b\in {\mathcal B}$. The range of possible values for $\beta$ is the interior of a polytope in $\R^{n-3}$ which is determined in terms of $\alpha$ by the inequalities above and matches the polytope $\Delta$ of Corollary~\ref{cor:action-angle}. The angles $\phi_b$ and $\phi_b'$ can take any values on their respective circles of length $\beta_b$ and $\beta_b'$. The slit condition in the definition of samosa assemblies (Definition~\ref{defn:samosa-assembly}) can be formulated in terms of an interval of admissible values for $\ell_b$. All small enough positive values of $\ell_b$ give valid slits. The maximum slit lengths depend on $\beta, \phi, \phi'$. Lemma~\ref{lem:bound-on-l-one-slit} gives the necessary and sufficient length bound for a samosa with one slit, and Lemma~\ref{lem:bound-on-l-two-slits} gives a sufficient length bound for a samosa with two slits.
\subsubsection{Realizing a samosa assembly}\label{sec:assembly-instructions-realization} Each non-degenerate samosa assembly in $\Assemb_{\alpha, \Upsilon}$ can be realized as a sphere with a hyperbolic cone metric. This process has two steps.
\begin{enumerate}
    \item Realize each pair of pants as a slit samosa. Its shape is determined, as described in Definition~\ref{defn:samosa-assembly}, by the parameter lists $\beta$, $\ell$, $\phi$, $\phi'$.
    \item Carry out each gluing using a conical cut and paste, as described in Section~\ref{sec:conical-cut-paste}.
\end{enumerate}
For each pants curve $b\in {\mathcal B}$, the assembled surface has a pair of samosas glued along slits of length $\ell_b$. The corners where the slits begin have cone angles $\beta_b' = 2\pi - \beta_b$ and $\beta_b$, so they fit together into a non-singular point. The points where the slits end are non-singular, so they fit together into a whole singularity: a point with cone angle $4\pi$.
\begin{center}
\begin{tikzpicture}[every node/.style={inner sep=0.5mm}]
\node[anchor=south west, inner sep=0mm] at (-6.12, 4.05) {\includegraphics[width=6cm]{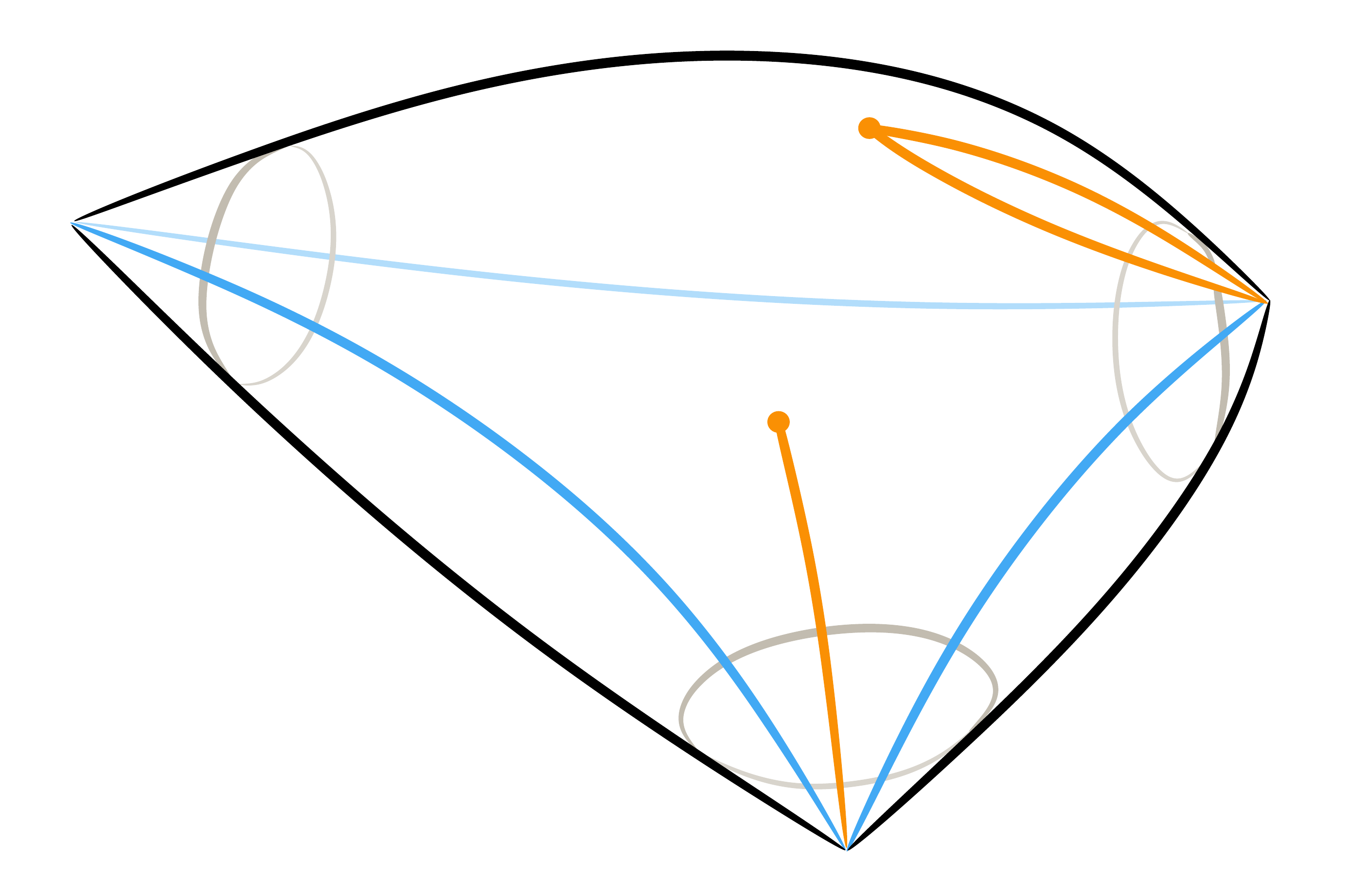}};
\node[anchor=south west, inner sep=0mm] at (-0.27, 7.65) {\includegraphics[width=6cm]{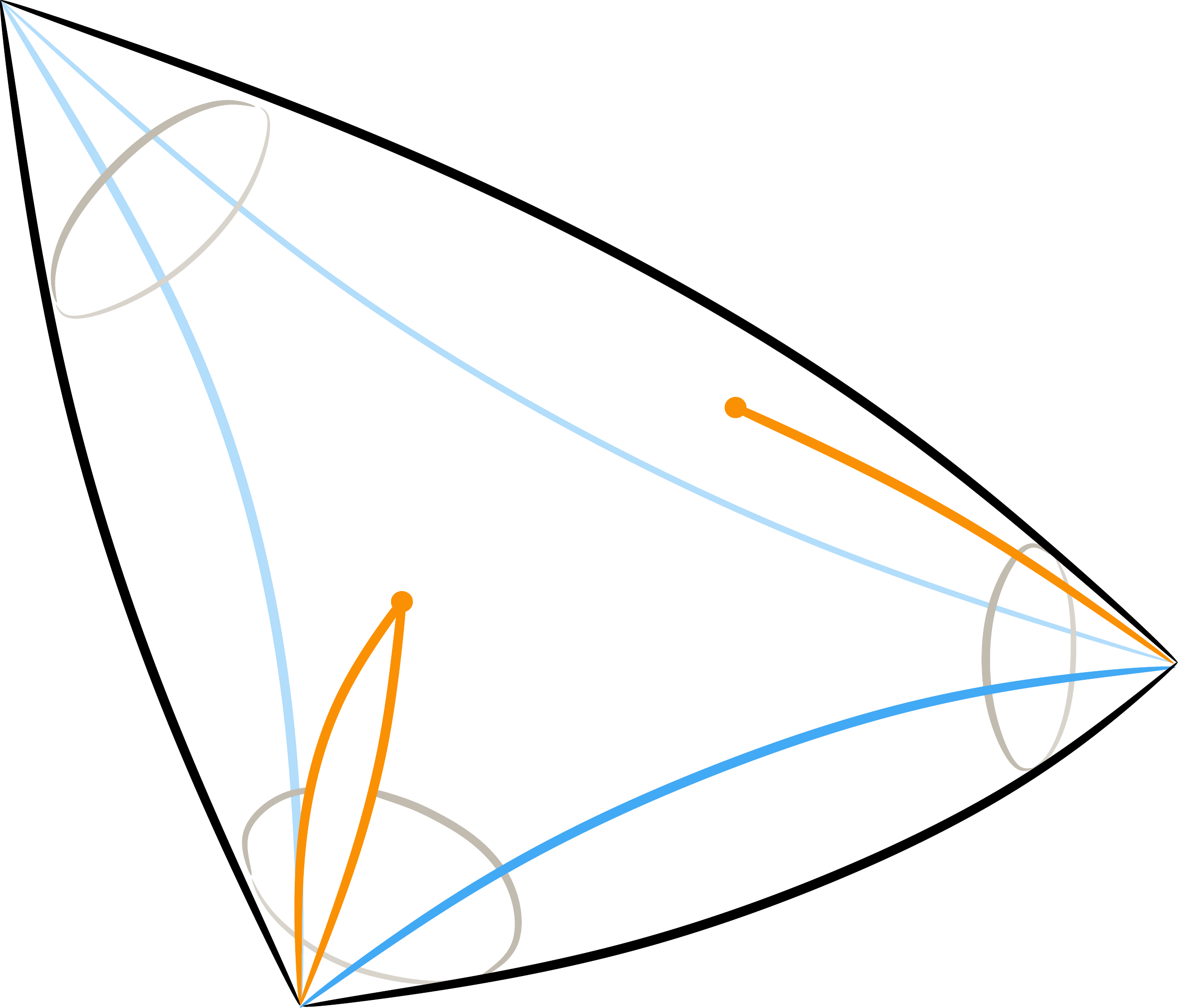}};
\draw[apricot, <->] (-1.7, 7.37) to[bend left=15] (1.43, 9.3);
\draw[apricot, <->] (-1.4, 7.03) to[bend left=15] (1.60, 8.82);
\end{tikzpicture}
\end{center}
For our illustrations, we've chosen samosas from the class defined in Section~\ref{sec:hamantash}, which can be neatly cut and unfolded onto the hyperbolic plane. This view makes it particularly easy to see how the beginnings and endings of slits become non-singular points and whole singularities, respectively. Keep in mind that not all samosas can be unfolded in this way.
\begin{center}
\begin{tikzpicture}[scale=.9, every node/.style={inner sep=0.5mm}]
\node[anchor=south west, inner sep=0mm] at (-4.96, -3.07) {\reflectbox{\includegraphics[width=9cm]{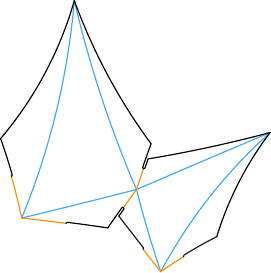}}};
\coordinate (S) at (-0.25, 0.75);
\coordinate (T) at (0.48, -0.67);
\path (0, 0) -- (S) node[apricot, pos=0.6, anchor=east] {$\ell_b$};
\path (0, 0) -- (T) node[apricot, inner sep=0.2mm, pos=0.6, anchor=north east] {$\ell_b$};
\draw ++(-47:0.3) arc (-47:103:0.3) node[pos=0.4, inner sep=0.25mm, anchor=south west] {$\beta_b$};
\node[align=center, anchor=north] at (0.5, -3.4) {\footnotesize The two geodesic arcs of length $\ell_b$ form a loop representing the pants curve $b$.
};
\end{tikzpicture}
\end{center}

For each unglued cuff $p$, the assembled surface has a fractional singularity, with angle defect $\alpha_p$. The glued edges of the slits associated with a pants curve $b \in \mathcal{B}$ form a closed curve representing $b$---a piecewise geodesic curve of length $2\ell_b$ that runs through the whole singularity, with angle $2\pi$ between the in- and out-going segments, and turns by $\beta_b - \pi$ at the non-singular point where the corners of the samosas were glued.

\subsubsection{Marking}\label{sec:marking}
If $\Upsilon$ is a chained pants assembly (Definition~\ref{defn:chained-pants-decomp}), the realizations of most samosa assemblies in $\Assemb_{\alpha, \Upsilon}$ are homeomorphic, up to isotopy, in a canonical way.\footnote{In fact, these realizations should even be diffeomorphic in a canonical way, allowing us to see $\Assemb_{\alpha, \Upsilon}$ as parameterizing different hyperbolic cone metrics on the same smooth sphere. However, these canonical diffeomorphisms seem tricky to construct and work with, so we'll stick to homeomorphisms for simplicity.} In other words, they each come with a distinguished isotopy class of homeomorphisms to a ``generic realization'' sphere $\AssembSphereExt{\Upsilon}$, which corresponds to an element of $\HypCone{\alpha}(\AssembSphereExt{\Upsilon})$. The distinguished isotopy classes come from a disconnected ribbon graph, called the {\em skeleton}, that sits within the pieces of a chained samosa assembly. When you realize a chained samosa assembly as a hyperbolic cone sphere, the pieces of the skeleton fit together into a connected ribbon graph that divides the sphere into disks. We'll first describe the skeleton of a generic samosa assembly, which we'll use to construct $\AssembSphereExt{\Upsilon}$. Then we'll examine the samosa assemblies whose skeletons aren't isomorphic to the generic one.

Within a non-terminal samosa, the skeleton has six vertices: the three corners of the samosa, the two points where the slits end, and one point on the equator segment connecting the slit corners. These are connected by five edges. Two edges, which belong to what we'll call the {\em spine}, run along the equator segment connecting the slit corners. Two more edges, called {\em half-ribs}, run along the slits. The last edge runs along the southern hemisphere from the middle vertex of the spine to the un-slit corner.
\begin{center}
\includegraphics[width=8cm]{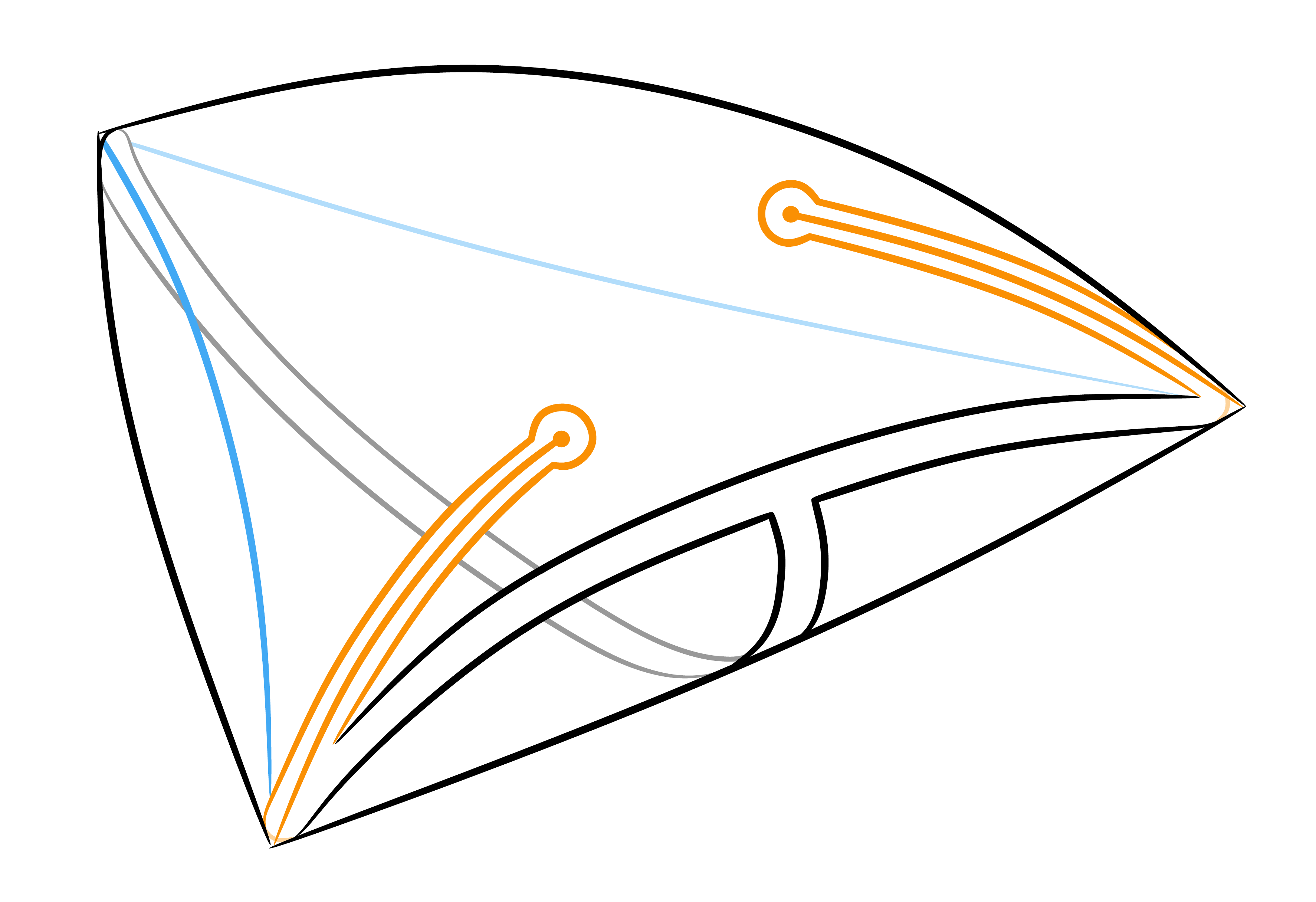}
\end{center}

Within a terminal samosa, the skeleton has four vertices: the three corners of the samosa and the point where the slit ends. These are connected by three edges: one half-rib, which runs along the slit, and two other edges, which run along the equator segments connecting the slit vertex to the un-slit vertices.
\begin{center}
\includegraphics[width=8cm]{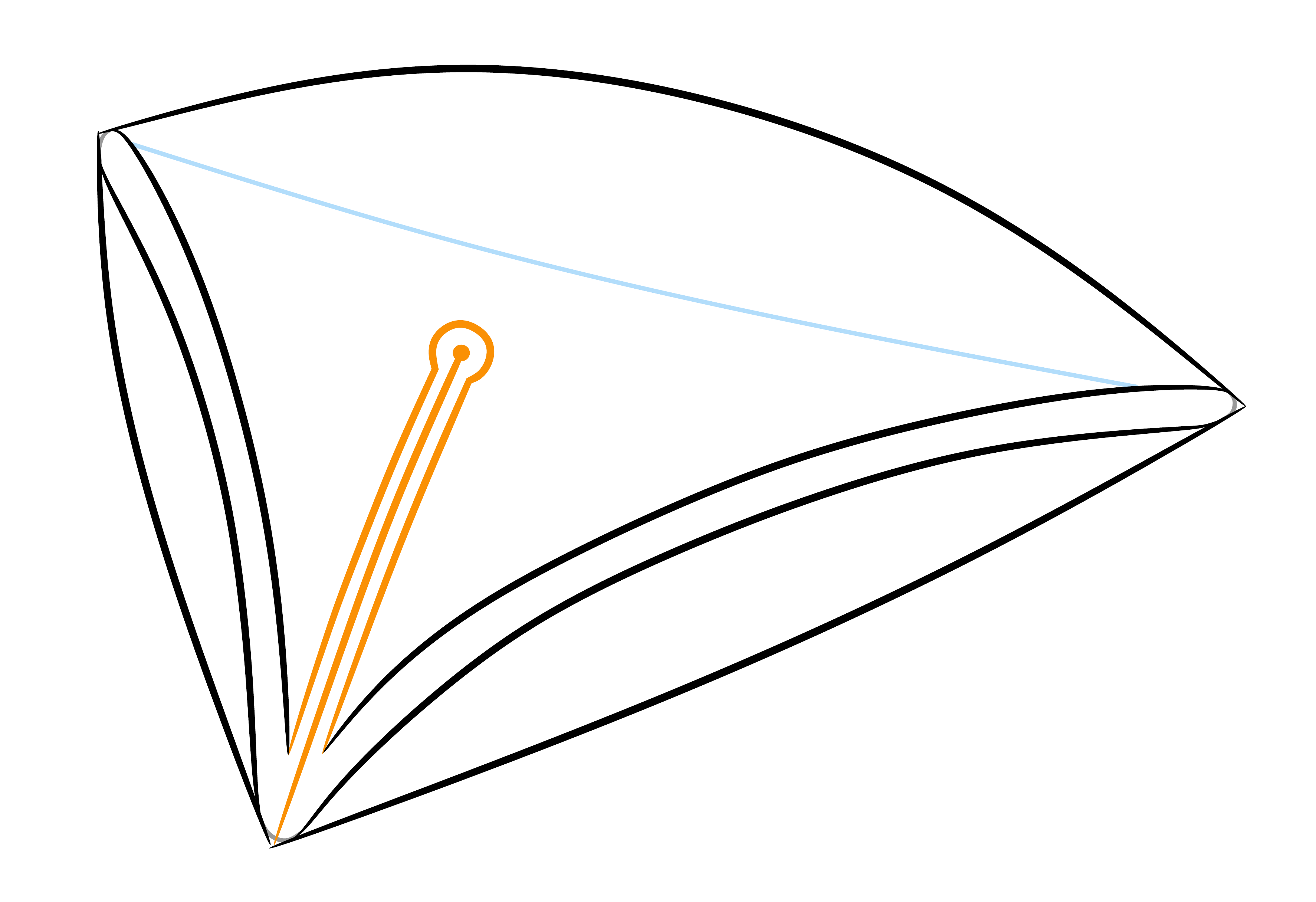}
\end{center}

To realize a samosa assembly, we first cut along the slits. This opens each half-rib into a half-ribbon loop, which looks like one of the halves you'd get if you took a loop made of two ribbon segments and cut it lengthwise down the middle. When we glue the samosas together, each gluing fuses two half-ribs into a {\em rib}, which really is a loop made of two ribbon segments. This joins the pieces of the skeleton into a connected ribbon graph on the assembled sphere.

The construction above only looks at combinatorial features of the samosa assembly. If we apply it to all of the samosa assemblies in $\Assemb_{\alpha, \Upsilon}$ that it can be applied to, we'll end up putting together vertices and edges in the same way each time. That gives us commuting isomorphisms between all of the resulting skeletons. We can thus treat these skeletons like a single ``generic skeleton'' for samosa assemblies modeled on $\Upsilon$.

The skeleton of a realized samosa assembly divides the realization into disks---one for each samosa. That means we can reconstruct the realization, up to isotopy, from the skeleton. More formally, the realization comes with an isotopy class of homeomorphisms to the sphere created from the skeleton by gluing a disk along each boundary component.

Guided by the last two paragraphs, let's build a {\em generic realization} sphere $\AssembSphereExt{\Upsilon}$ by gluing a disk along each boundary component of the generic skeleton for samosa assemblies modeled on $\Upsilon$. If the skeleton of a samosa assembly in $\Assemb_{\alpha, \Upsilon}$ can be constructed using the generic process described earlier, the realization of that samosa assembly comes with an isotopy class of homeomorphisms to $\AssembSphereExt{\Upsilon}$.
\tikzset{%
  on layer/.code={
    \pgfonlayer{#1}\begingroup
    \aftergroup\endpgfonlayer
    \aftergroup\endgroup
  },
  ribbon/.style={
    postaction={
      draw, white,
      line width=0.9mm,
      on layer=foreground
    }
  }
}
\newcommand{\ribbonsing}[2]{
  \fill[#2] (#1) circle (0.165);
  \fill[white, on layer=foreground] (#1) circle (0.12);
}
\begin{center}
\vspace{4mm}
\begin{tikzpicture}[scale=0.9]
  \pgfmathsetmacro{\rearth}{4}
  \pgfmathsetmacro{\tilt}{5}
  \pgfmathsetmacro{\aleaf}{20}
  \pgfmathsetmacro{\awhole}{-30}
  \pgfmathsetmacro{\aspine}{30}
  
  \pgfdeclarelayer{background}
  \pgfdeclarelayer{foreground}
  \pgfsetlayers{background, main, foreground}

  \newcommand{\spherept}[3]{\coordinate (#1) at ({-\rearth*sin((#2)+\tilt)*(1-abs(sin((#3)+\tilt))*sin(\tilt))}, {-\rearth*sin((#3)+\tilt)*cos((#2)+\tilt)});}
  
  \newcommand{\draweq}{\draw ellipse ({\rearth} and {\rearth*sin(\tilt)});}
  \pgfmathsetmacro{\nudge}{-0.015*\rearth*sin(\aspine)}
  \spherept{spine right}{-\aspine}{0}
  \spherept{C1}{55}{30}
  \spherept{C2}{55}{-30}
  \spherept{Cm1}{-55}{30}
  \spherept{Cm2}{-55}{-30}
  \begin{scope}[line width=1.8mm]
    \draw[domain=-\aspine+\tilt:\aspine+\tilt, ribbon] plot ({-\rearth*sin(\x)}, {-\rearth*cos(\x)*sin(\tilt)}) coordinate (spine left);
    \draw[ribbon] (C1) to[bend left=2] (spine left) to[bend left=4] (C2);
    \draw[ribbon] (Cm2) to[bend left=4] (spine right) to[bend right=2] (Cm1);
    
    \foreach \pt in {C1, C2, Cm2, Cm1} {
      \ribbonsing{\pt}{black}
    }
  \end{scope}
  
  \newcommand{\drawcirc}[1]{\draw[#1] ({\zcirc+\nudge}, 0) ellipse ({\rcirc*sin(\tilt)} and {\rcirc});}
  \foreach \lat in {30, 0, -30} {
    \pgfmathsetmacro{\rcirc}{\rearth*cos(\lat)}
    \pgfmathsetmacro{\zcirc}{\rearth*sin(\lat)}
    \pgfmathsetmacro{\nudge}{-0.015*\zcirc}
    
    \coordinate (W) at ({\zcirc-\rcirc*cos(\awhole)*sin(\tilt)+\nudge}, {-\rearth*sin(\awhole)});
    \ribbonsing{W}{apricot}
    
    \begin{pgfonlayer}{background}
    \begin{scope}[lightapricot, line width=0.45mm]
      \clip ({\zcirc+\nudge}, -\rcirc) rectangle (\rearth, \rcirc);
      \drawcirc{double distance=0.6mm}
    \end{scope}
    \end{pgfonlayer}
    \begin{scope}[apricot, line width=0.45mm]
      \clip ({\zcirc+\nudge}, -\rcirc) rectangle (-\rearth, \rcirc);
      \drawcirc{double distance=0.9mm}
    \end{scope}
  }
  
  \foreach \lat in {15, -15} {
    \pgfmathsetmacro{\rcirc}{\rearth*cos(\lat)}
    \pgfmathsetmacro{\zcirc}{\rearth*sin(\lat)}
    \pgfmathsetmacro{\xleaf}{\rearth*sin(\aleaf+\tilt)}
    \pgfmathsetmacro{\zleaf}{\rearth*cos(\aleaf)}
    \pgfmathsetmacro{\nudge}{-0.015*\zcirc}
    
    \draw[domain=0:\aleaf, ribbon, line width=1.8mm] plot ({\zcirc-\rcirc*cos(\x)*sin(\tilt)+\nudge}, {-\rearth*sin(\x+\tilt)}) coordinate (leaf);
    
    
    \ribbonsing{leaf}{}
  }
  
  \draw[line width=0.45mm, on layer=foreground] circle (\rearth);

  \node[align=center, anchor=north] at (0, {-(\rearth+0.5)}) {\footnotesize The generic realization sphere $\AssembSphereExt{\Upsilon}$ for a chained pants assembly $\Upsilon$ with four pairs of pants.};
\end{tikzpicture}
\vspace{4mm}
\end{center}

Now, let's look at the samosa assemblies whose skeletons can't be constructed using the generic process described earlier. These are the ones where a slit lies on an equator segment opposite from an un-slit corner, which would cause a half-rib to collide with another ribbon in the usual construction. To build a skeleton for one of these samosas, we first place the half-ribs along the slits, as usual. Then, whenever an equatorial ribbon would collide with a half-rib, we shorten it to avoid the collision. The offending half-rib ends up attached the wrong way around: the shortened ribbon meets it at the point where the slit ends, rather than at the corner where the slit starts. The construction and assembly of the skeleton then proceeds as usual.

Homotopically, we can connect this construction to the generic one by sliding the attachment point of each shortened ribbon along the boundary of the offending half-rib, bringing it to the corner where the slit starts.
\begin{center}
\vspace{6mm}
\newcommand{\ribbonpoints}{
  \coordinate (C2) at (-1.6, 0);
  \coordinate (A) at (0, 0);
  \coordinate (A') at (2.2, 0);
}
\newcommand{\ribbonvertices}{
  \ribbonsing{C2}{apricot}
  \ribbonsing{A}{apricot}
  \ribbonsing{A'}{black}
}
\begin{tikzpicture}[line width=1.8mm, every edge/.append style={line width=1.8mm}]
\pgfdeclarelayer{background}
\pgfdeclarelayer{foreground}
\pgfsetlayers{background, main, foreground}

\ribbonpoints
\draw[ribbon, apricot] (C2) -- (A);
\draw[ribbon] (A) -- (A');
\ribbonvertices

\begin{scope}[shift={(-3, -1.6)}, line width=1.8mm]
\ribbonpoints
\draw[ribbon, apricot] (C2) -- (A);
\draw (C2) edge[ribbon, out=40,in=160] (A');
\ribbonvertices
\path (C2) ++(16:0.6) coordinate (slideto);
\draw[->, thin, looseness=0.9] (A) ++(20:0.3) to[out=110, in=16] (slideto);
\end{scope}

\begin{scope}[shift={(3, -1.6)}, line width=1.8mm]
\ribbonpoints
\draw[ribbon, apricot] (C2) -- (A);
\draw (C2) edge[ribbon, out=-40,in=200] (A');
\ribbonvertices
\path (C2) ++(-16:0.6) coordinate (slideto);
\draw[->, thin, looseness=0.9] (A) ++(-20:0.3) to[out=-110, in=-16] (slideto);
\end{scope}
\end{tikzpicture}
\end{center}
That's equivalent to doing a half Dehn twist along the corresponding rib of the assembled skeleton. We can slide the attachment point along either side of the offending rib; equivalently, we can do the half Dehn twist in either direction. Choosing a half-twist direction for each offending half-rib picks out an isotopy class of homeomorphisms from the realized samosa assembly to the generic realization $\AssembSphereExt{\Upsilon}$. We can switch any of the chosen half-twist directions by sliding the attachment point back up the half-rib and down the other side, changing the isotopy class of homeomorphisms by a full Dehn twist along the corresponding rib. As another way to visualize the choice of identification, you can imagine pushing each offending slit infinitesimally off the equator, in one direction or the other. The equatorial ribbon can then extend to the corner like it does in the generic construction.

Let's say $\Upsilon$ has $n$ unglued cuffs, so $\AssembSphereExt{\Upsilon}$ has $n + (n-3)$ marked points. Erasing the $n-3$ marked points on the ribs gives a sphere $\AssembSphere{\Upsilon}$ with $n$ marked points, which comes with an inclusion $\AssembSphereExt{\Upsilon} \hookrightarrow \AssembSphere{\Upsilon}$. The generic skeleton picks out a distinguished geometric presentation of $\pi_1\AssembSpherePk{\Upsilon}$. To see it, first construct $\pi_1\AssembSpherePk{\Upsilon}$ using the spine as an extended base point. We can do this because the spine is compact and simply connected. We can now draw a counterclockwise loop around each puncture by following the boundary of the ribbon that connect the vertex at that puncture to the spine. Ordering the set of loops counterclockwise around the spine will pick out a distinguished geometric presentation $c_1, \ldots, c_n$ of $\pi_1\AssembSpherePk{\Upsilon}$, once we refine that cyclic ordering to a linear one. We choose the unique linear ordering in which the first two and last two loops enclose the punctures on the terminal samosas. Each loop encloses a different puncture of $\AssembSphere{\Upsilon}$, identifying the set of punctures with $\{1, \ldots, n\}$.

Recall when we realize a genus-0 pants assembly by gluing pairs of pants together, the glued cuffs become pants curves that form a pants decomposition of the assembled sphere. Correspondingly, the ribs form a pants decomposition of $\AssembSpherePk{\Upsilon}$, which turns out to be the standard pants decomposition associated with the distinguished geometric presentation of $\pi_1\AssembSpherePk{\Upsilon}$ described above.

We can now think of the realization process as a {\em realization map}
\[ R \maps \Assemb^\text{gen}_{\alpha, \Upsilon} \to \HypCone{\alpha}(\AssembSphereExt{\Upsilon}), \]
where $\Assemb^\text{gen}_{\alpha, \Upsilon} \subset\Assemb_{\alpha, \Upsilon}$ is the space of samosa assemblies in which no slit lies on an equator segment opposite from an un-slit corner and $\HypCone{\alpha}(\AssembSphereExt{\Upsilon})$ is the space of hyperbolic cone structures on $\AssembSphereExt{\Upsilon}$ introduced in Section~\ref{sec:hyperbolic-cone-structures}. This formalizes the intuition that samosa assemblies parameterize hyperbolic cone structures.
\subsection{Parameterizing hyperbolic cone structures}
We'd like to understand the realization map as a parameterization of hyperbolic cone structures (Definition~\ref{defn:hyperbolic-cone-structure}). From now on, for simplicity, we'll only consider the case where $\Upsilon$ is a chained pants assembly (Definition~\ref{defn:chained-pants-decomp}). That will be enough to accomplish the original goal of understanding Maret's action-angle coordinate system.

We expect the realization map $R$ to be injective on its whole domain $\Assemb^\text{gen}_{\alpha, \Upsilon}$, but we won't try to prove this. For our purposes, it's enough to show that $R$ is injective on each of the open subsets $\Assemb_{\alpha, \Upsilon}^\varepsilon\subset \Assemb^\text{gen}_{\alpha, \Upsilon}$ determined by making a choice of hemisphere interior for each slit, represented by a tuple $\varepsilon\in\{\text{north},\,\text{south}\}^{2(n-3)}$. The hemisphere choice $\varepsilon$ limits the domain of each slit angle to positive or negative values.

The $4(n-3)$ parameters that define a samosa assembly give a natural inclusion of $\Assemb_{\alpha, \Upsilon}$ into a product of $\R^{2(n-3)}$ with $2(n-3)$ circles, whose lengths depend on $\beta$. The real-valued parameters are $\beta$ and $\ell$, and the circle-valued ones are $\phi$ and $\phi'$. Restricting each slit to the interior of a particular hemisphere refines this inclusion to $\Assemb_{\alpha, \Upsilon}^\varepsilon\subset \R^{4(n-3)}$. From this perspective, the realization map $R\colon \Assemb_{\alpha, \Upsilon}^\varepsilon \to \HypCone{\alpha}(\AssembSphereExt{\Upsilon})$ turns out to be a chart for $\HypCone{\alpha}(\AssembSphereExt{\Upsilon})$.
\begin{thm}\label{thm:realization-map-local-homeo}
For any chained pants assembly $\Upsilon$, the realization map $R\colon \Assemb_{\alpha, \Upsilon}^\varepsilon \to \HypCone{\alpha}(\AssembSphereExt{\Upsilon})$ is a homeomorphsim onto its image with respect to the $\mathcal{C}^0$ topology on $\HypCone{\alpha}(\AssembSphereExt{\Upsilon})$, defined in Section~\ref{sec:space-of-hyperbolic-cone-structures}. Moreover, if Conjecture~\ref{smooth-atlas} holds, $R$ is a diffeomorphism onto its image.
\end{thm}
\begin{proof}
The proof unfolds across several sections. We start by proving in Section~\ref{sec:proof-thm-2-injectivity} that the realization map is injective. Next, in Section~\ref{sec:proof-thm-2-continuity}, we show that the realization map is continuous, using an explicit triangulation of $\AssembSphereExt{\Upsilon}$ that we describe in Sections~\ref{sec:proof-thm-2-triangulating-joker's-hats}--\ref{sec:proof-thm-2-triangulating-V-pieces-case-V2}. Finally, we show in Section~\ref{sec:proof-thm-2-continuity-inverse-map} that the inverse of the realization map is also continuous, completing the proof that the realization map is a homeomorphism onto its image. We go on to show in Section~\ref{sec:proof-thm-2-smoothness} that if Conjecture~\ref{smooth-atlas} holds, the realization map is a diffeomorphism onto its image. Several of our arguments will rely on two trigonometric lemmas, proven in Section~\ref{sec:proof-thm-2-two-trigonometric-lemmas}.
\end{proof}
\subsubsection{Two trigonometric lemmas}\label{sec:proof-thm-2-two-trigonometric-lemmas}
To prove Theorem \ref{thm:realization-map-local-homeo}, we need the following lemmas.
\begin{lem}\label{lem:kink-angle}
A pants curve $b$ on the assembled sphere
is associated with a slit of length $\ell_b$ on each adjacent samosa. Gluing along these slits creates a whole singularity. Let $c_b$ be the length of the shortest representative of $b$ that goes through the whole singularity. This curve can, and typically will, be kinked at the whole singularity.\footnote{This curve is homotopic to, but different from, the piecewise geodesic curve formed by the edges of the slit, which we discussed in Section~\ref{sec:assembly-instructions-realization}.} Let $\kappa_b\in (-\pi,\pi)$ be the kink angle---signed positive if the curve turns left at the singularity, and negative if the curve turns right.\footnote{In the illustration, the curve turns left.}
\begin{center}
\begin{tikzpicture}[every node/.style={inner sep=0.5mm}]
\node[anchor=south west, inner sep=0mm] at (-3.76, -3.06) {\includegraphics[width=6cm]{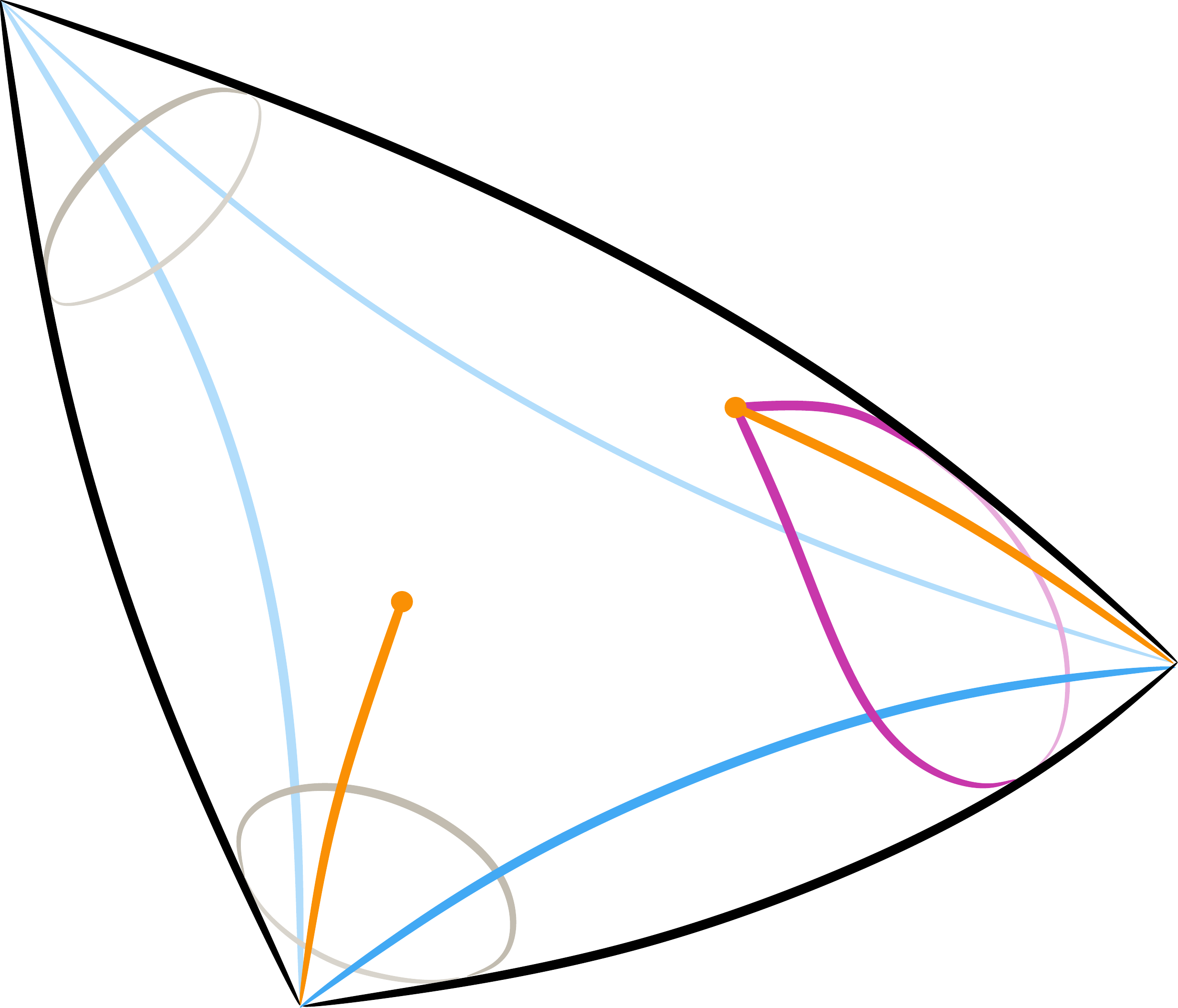}};
\node[anchor=south west, inner sep=0mm] at (2, -0.5) {\includegraphics[width=8cm]{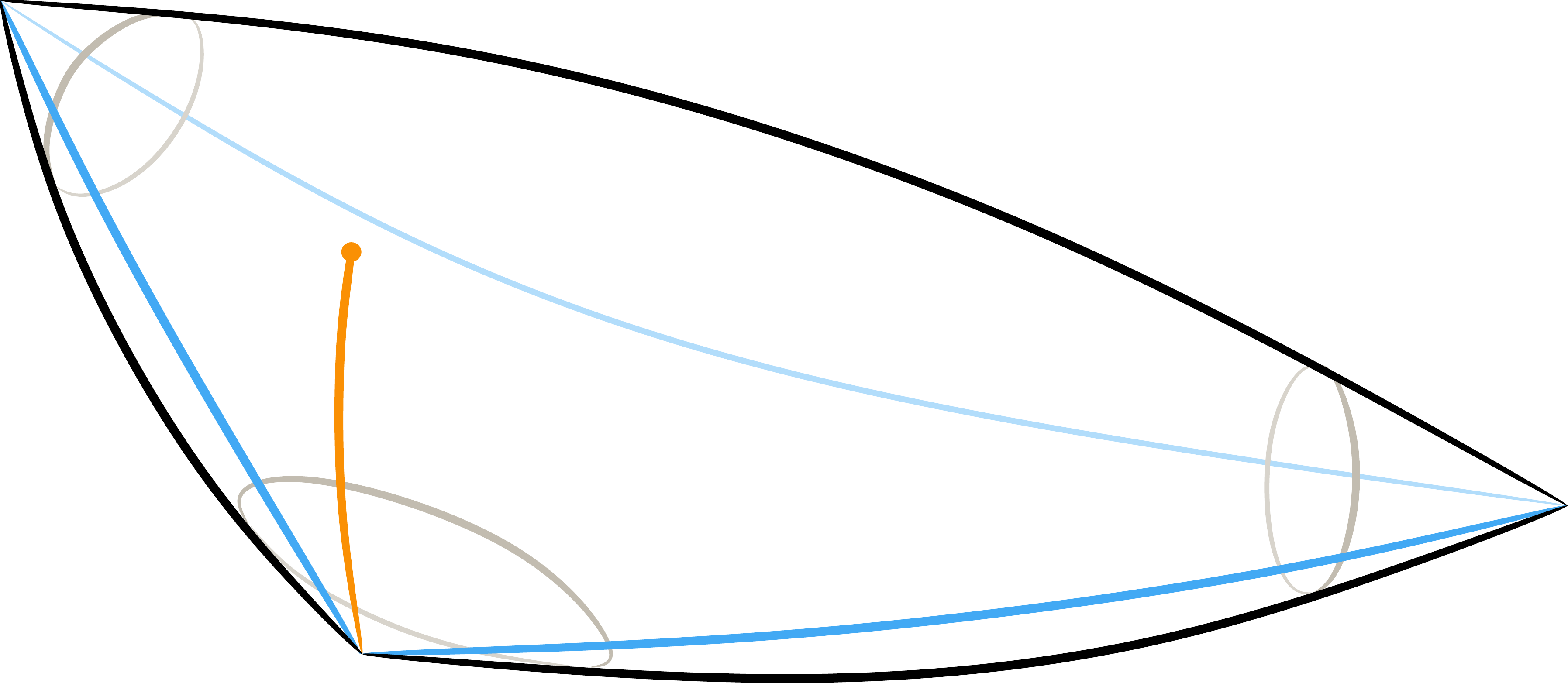}};

\coordinate (B) at (2.24, -1.3);
\draw[mauve] ++(-1:0.3) arc (-1:-60:0.3) node[inner sep=0.2mm, outer sep=0.5mm, fill=white, rounded corners=0.5mm, anchor=30] {$\kappa_b$};
\node[mauve] at (0.3, -1.3) {$c_b$};
\path (0, 0) edge[draw=none, bend left=6] node[apricot, inner sep=0.2mm, outer sep=1.2mm, fill=white, rounded corners=0.5mm, midway, anchor=north] {$\ell_b$} (B);

\draw[apricot] (1.35, -0.1) edge[<->, bend left=10] node[circle, anchor=-50] {\textnormal{glue}} (2.5, 0.7);
\end{tikzpicture}

\begin{tikzpicture}[scale=.8, every node/.style={inner sep=0.5mm}]
\node[anchor=south west, inner sep=0mm] at (-9.01, -5.38) {\reflectbox{\includegraphics[width=8cm, angle=-90]{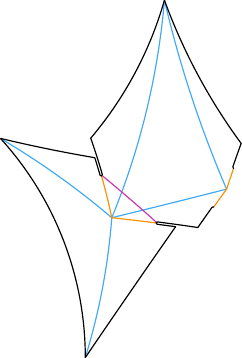}}};
\coordinate (A) at (-1.77, 0.44);
\coordinate (B) at (0.23, -1.88);
\path (A) -- (B) node[mauve, midway, anchor=north east] {$c_b$};
\path (0, 0) -- (A) node[apricot, pos=0.6, anchor=south west] {$\ell_b$};
\path (0, 0) -- (B) node[apricot, pos=0.45, anchor=west] {$\ell_b$};
\draw ++(160:0.3) arc (160:-75:0.3) node[midway, inner sep=1.2mm, anchor=-150] {$\beta_b$};
\draw[mauve] (A) ++(-20:0.4) arc (-20:-46:0.4);
\draw[mauve] (A) ++(-46:0.4) node[anchor=north east] {$\kappa_b/2$};
\draw[mauve] (B) ++(102:0.4) arc (102:127:0.4);
\draw[mauve] (B) ++(102:0.4) node[inner sep=1.2mm, anchor=west] {$\kappa_b/2$};
\end{tikzpicture}
\end{center}
The samosa assembly parameters $\ell_b$, $\beta_b$ are related to $c_b$, $\kappa_b$ by the equations
\begin{align}
\cosh (c_b) & = \cosh (\ell_b)^2-\sinh(\ell_b)^2\cos(\beta_b)\label{eq:length-kinked-geodesic}\\
\tan (\kappa_b/2) & = \frac{1+\cos(\beta_b)}{\sin(\beta_b)\cosh(\ell_b)}. \label{eq:kinked-angle}
\end{align}
If $\beta_b=\pi$, then $c_b=2\ell_b$ and $\kappa_b =0$.
\end{lem}

\begin{proof}
We have an isosceles triangle with two sides of length $\ell_b$, which meet at an interior angle of $2\pi-\beta_b$. The third side has length $c_b$, and it makes an angle of $\kappa_b/2$ with each of the other two sides. The hyperbolic laws of cosines---formulas~\eqref{eq:hyperbolic-law-of-cosines-lengths} and~\eqref{eq:hyperbolic-law-of-cosines-angles} from Appendix~\ref{apx:trig-formulae}---give
\begin{align*}
\cosh(c_b) & = \cosh (\ell_b)^2-\sinh(\ell_b)^2\cos(\beta_b) \\
\sin (|\kappa_b|/2)\,|\sin (\beta_b)|\,\cosh(\ell_b) & = \cos (\kappa_b/2) +\cos (\kappa_b/2)\cos(\beta_b).
\end{align*}
Our convention for the sign of the kink angle ensures that $\kappa_b$ is positive when $\beta_b<\pi$ and negative when $\beta_b>\pi$. That lets us rewrite the second relation without the absolute values:
\[ \sin (\kappa_b/2) \sin (\beta_b) \cosh(\ell_b)=\cos (\kappa_b/2) +\cos (\kappa_b/2)\cos(\beta_b). \]
Dividing by $\cos (\kappa_b/2)$ gives the desired formula.
\end{proof}

\begin{rem}\label{rem:angle-beta-and-length-ell}
Equations \eqref{eq:length-kinked-geodesic} and \eqref{eq:kinked-angle} can be rewritten to express the samosa assembly parameters $\ell_b$, $\beta_b$ in terms of $c_b$, $\kappa_b$ instead. This can be seen directly using hyperbolic trigonometry in the same isosceles triangles. The equations are
\begin{align}
\cos(\beta_b) & = \sin(\kappa_b/2)^2\cosh(c_b)-\cos(\kappa_b/2)^2. \label{eq:angle-beta}
\end{align}
Plugging the value of $\cos(\beta_b)$ into~\eqref{eq:length-kinked-geodesic} and simplifying shows that
\begin{align}
 \cosh(\ell_b)^2& = \frac{\cos(\kappa_b/2)^2}{\cos(\kappa_b/2)^2+\tanh (c_b/2)^2}. \label{eq:length-ell}   
\end{align}
If we know $c_b$, $\kappa_b$, we can determine $\beta_b$ from the sign of $\kappa_b$ and Equation~\eqref{eq:angle-beta}, and we can determine $\ell_b$ from Equation~\eqref{eq:length-ell}. The dependence of $\beta_b$ on the sign of $\kappa_b$ will be made explicit later, in Equation~\eqref{eq:beta-as-function-of-h}.
\end{rem}

\begin{lem}\label{lem:slit-angle}
On a samosa hemisphere containing one or two slits, the following formulas relate the slit angles to the slit lengths, the lengths of equator segments, and the distances from the slit endpoints to un-slit corners and to each other.
\begin{itemize}
    \item First, consider a hemisphere containing a single slit at a corner $b$. The slit angle $\phi_b$ is related to the length $d_{b}$ of the equator segment joining the slit corner to an un-slit corner with apex $p\in\mathcal{P}$, the slit length $\ell_b$, and length $\lambda_b$ of the shortest geodesic arc joining the endpoint of the slit to $p$ by the equation
    \begin{equation}\label{eq:slit-angle-case-one-slit}
    \cos(\phi_b)=\frac{\cosh(d_{b})\cosh(\ell_b)-\cosh(\lambda_b)}{\sinh(d_{b})\sinh(\ell_b)}.
    \end{equation}
    \begin{center}
    \begin{tikzpicture}[scale=.9, every node/.style={inner sep=0.5mm}]
    \node[anchor=south west, inner sep=0mm] at (-3.21, -0.03) {\reflectbox{\includegraphics[width=4.05cm, angle=-90]{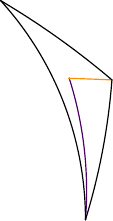}}};
    \coordinate (A) at (5.59, 1.06);
    \coordinate (B) at (-3.18, 4.46);
    \coordinate (S) at (-0.06, 1.7);
    \draw[apricot] ++(14:0.3) arc (14:88:0.3) node[midway, anchor=south west] {$\phi_b$};
    \path (0, 0) -- (S) node[apricot, inner sep=1.2mm, pos=0.55, anchor=east] {$\ell_b$};
    \path (0, 0) edge[draw=none, bend right=6] node[inner sep=1.2mm, midway, anchor=120] {$d_{b}$} (A);
    \path (S) edge[draw=none, bend right=9] node[plum, inner sep=1.2mm, pos=0.4, anchor=80] {$\lambda_b$} (A);
    \node[anchor=west, inner sep=1mm] at (A) {$p$};
    \end{tikzpicture}
    \end{center}
    \item Now, consider a hemisphere containing two slits, at two corners $a$, $b$ with cone angles $\beta_a$, $\beta_b'$. We'll distinguish two configurations, both shown in the figure below. In the first configuration, neither of the two geodesic arcs joining the un-slit corner to one of the slit endpoints crosses the other slit. In that case, the slit angles $\phi_a$, $\phi_b'$ are related to $d$, $\ell$, and $\lambda$ by the equations
    \begin{align}
    \cos\left(\beta_a/2-|\phi_a|\right)&=\frac{\cosh(d_a)\cosh(\ell_a)-\cosh(\lambda_a)}{\sinh(d_a)\sinh(\ell_a)} \label{eq:second-slit-angle-case-two-slits}\\
    \cos(\phi_b')&=\frac{\cosh(d_b)\cosh(\ell_b)-\cosh(\lambda_b)}{\sinh(d_b)\sinh(\ell_b)}. \label{eq:first-slit-angle-case-two-slits}    
    \end{align} 
    For the second configuration, we'll assume that the geodesic arc joining the un-slit corner to the endpoint of the slit at corner $b$ crosses the other slit. In that case, we prefer the following alternative formula to compute $\phi_b'$. We introduce two extra parameters: the length $\xi$ of the shortest geodesic arc joining the endpoints of the two slits, and the angle $\eta$ between the two geodesic arcs from the endpoint of the slit based at corner $a$ to the un-slit corner and to the endpoint of the other slit.
    \begin{equation}
        \cos(\phi_b')=\frac{\cosh(d_b)\cosh(\ell_b)-\cosh(\xi)\cosh(\lambda_a)+\sinh(\xi)\sinh(\lambda_a)\cos(\eta)}{\sinh(d_b)\sinh(\ell_b)}. \label{eq:first-slit-angle-special-case-two-slits}
    \end{equation}
    \begin{center}
    \begin{tikzpicture}[scale=.9, every node/.style={inner sep=1mm}]
    \node[anchor=south west, inner sep=0mm] at (-2.44, -8.47) {\reflectbox{\includegraphics[width=4.05cm]{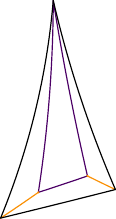}}};
    \coordinate (A) at (-2.4, -7.33);
    \coordinate (B) at (2.04, -8.43);
    \coordinate (sA) at (-1.32, -6.8);
    \coordinate (sB) at (0.59, -7.44);
    
    \draw[apricot] (A) ++(22:0.6) arc (22:-10:0.6) node[inner sep=1.5mm, anchor=100] {$\phi_a$};
    \draw (A) ++(-9:0.3) arc (-9:61:0.3) node[anchor=-10] {$\beta_a/2$};
    \draw (B) ++(116:0.3) arc (116:163:0.3) node[anchor=60] {$\beta_b'/2$};
    \draw[apricot] (B) ++(144:0.6) arc (144:115:0.6) node[anchor=200] {$\phi_b'$};
    
    \path (A) edge[draw=none, bend right=4] node[midway, inner sep=1.2mm, anchor=east] {$d_a$} (0, 0);
    \path (B) edge[draw=none, bend left=7] node[midway, inner sep=1.2mm, anchor=west] {$d_b$} (0, 0);
    \path (sA) edge[draw=none, bend right=2] node[plum, inner sep=0.5mm, pos=0.3, anchor=west] {$\lambda_a$} (0, 0);
    \path (sB) edge[draw=none, bend left=4] node[plum, inner sep=0.5mm, pos=0.3, anchor=east] {$\lambda_b$} (0, 0);
    \path (sA) -- (sB) node[plum, midway, inner sep=0.8mm, anchor=north east] {$\xi$};
    \path[apricot] (A) -- (sA) node[inner sep=0.3mm, pos=0.8, anchor=south east] {$\ell_a$};
    \path[apricot] (B) -- (sB) node[midway, anchor=south] {$\ell_b$};
    
    \begin{scope}[shift={(8, 0)}]
    \node[anchor=south west, inner sep=0mm] at (-2.44, -8.47) {\reflectbox{\includegraphics[width=4.05cm]{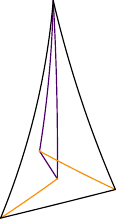}}};
    \coordinate (A) at (-2.4, -7.33);
    \coordinate (B) at (2.04, -8.43);
    \coordinate (sA) at (0.53, -5.84);
    \coordinate (sB) at (-0.15, -6.92);
    
    \draw[apricot] (A) ++(22:0.6) arc (22:-10:0.6) node[inner sep=1.5mm, anchor=100] {$\phi_a$};
    \draw (A) ++(-9:0.3) arc (-9:61:0.3) node[anchor=-10] {$\beta_a/2$};
    \draw (B) ++(116:0.3) arc (116:163:0.3) node[anchor=60] {$\beta_b'/2$};
    \draw[apricot] (B) ++(144:0.6) arc (144:115:0.6) node[anchor=200] {$\phi_b'$};
    \draw[plum] (sA) ++(105:0.3) arc (105:232:0.3) node[midway, anchor=-70] {$\eta$};
    
    \path (A) edge[draw=none, bend right=4] node[midway, inner sep=1.2mm, anchor=east] {$d_a$} (0, 0);
    \path (B) edge[draw=none, bend left=7] node[midway, inner sep=1.2mm, anchor=west] {$d_b$} (0, 0);
    \path (sA) edge[draw=none, bend right=2] node[plum, inner sep=0.4mm, pos=0.1, anchor=west] {$\lambda_a$} (0, 0);
    \path (sA) -- (sB) node[plum, midway, inner sep=0.5mm, anchor=135] {$\xi$};
    \path[apricot] (A) -- (sA) node[inner sep=0.3mm, pos=0.55, anchor=south east] {$\ell_a$};
    \path[apricot] (B) -- (sB) node[midway, anchor=south] {$\ell_b$};
    \end{scope}
    \end{tikzpicture}
    \end{center}
    \end{itemize}
    The hyperbolic law of cosines---formula~\eqref{eq:hyperbolic-law-of-cosines-angles} from Appendix~\ref{apx:trig-formulae}---expresses the hyperbolic cosine of an equator segment in terms of the corner angles of the samosa. So, similarly as for Lemma~\ref{lem:bound-on-l-one-slit}, it's possible to formulate Equations~\eqref{eq:slit-angle-case-one-slit}, \eqref{eq:second-slit-angle-case-two-slits}, \eqref{eq:first-slit-angle-case-two-slits}, \eqref{eq:first-slit-angle-special-case-two-slits} in terms of the corner angles of the samosa rather than the lengths of the equator segments.
\end{lem}
\begin{proof}
Each statement is a direct application of the hyperbolic law of cosines---formula~\eqref{eq:hyperbolic-law-of-cosines-lengths} from Appendix~\ref{apx:trig-formulae}. 
\end{proof}

\begin{rem}\label{rem:sign-slit-angle}
Equations~\eqref{eq:slit-angle-case-one-slit}, \eqref{eq:second-slit-angle-case-two-slits}, \eqref{eq:first-slit-angle-case-two-slits}, \eqref{eq:first-slit-angle-special-case-two-slits} determine the absolute values of the slit angles. The sign of each slit angle is determined by which hemisphere it lies in. 
\end{rem}
\subsubsection{Injectivity of the realization map}\label{sec:proof-thm-2-injectivity}
With Lemmas \ref{lem:kink-angle} and \ref{lem:slit-angle} in hand, we're ready to tackle Theorem~\ref{thm:realization-map-local-homeo}. We start by proving that the realization map $R\colon \Assemb_{\alpha, \Upsilon}^\varepsilon \to \HypCone{\alpha}(\AssembSphereExt{\Upsilon})$ is injective. In other words, we show that parameters of a samosa assembly are uniquely determined by the associated isotopy class $[\HypConeStruct]\in R(\Assemb_{\alpha, \Upsilon}^\varepsilon)$ of hyperbolic cone structures on $\AssembSphereExt{\Upsilon}$. Our strategy is to express each of the samosa assembly parameters as a function of \emph{intrinsic parameters} of $[\HypConeStruct]$, by which we mean quantities that can be measured using the metric on the realized cone sphere without knowing how the realization was assembled from samosas.

We first show that $\ell$ and $\beta$ are determined by $[\HypConeStruct]$. Each pants curve $b\in {\mathcal B}$ is represented by a unique kinked geodesic from the associated whole singularity to itself, as described in Lemma~\ref{lem:kink-angle}. Its length $c_b$ and kink angle $\kappa_b$ are both intrinsic parameters of $[\HypConeStruct]$. Using~\eqref{eq:length-ell}, we can express the value of $\ell_b$ in terms of $c_b$, $\kappa_b$:
\begin{equation}\label{eq:slit-length-as-function-of-h}
    \ell_b = \arccosh{\sqrt{\frac{\cos(\kappa_b/2)^2}{\cos(\kappa_b/2)^2+\tanh (c_b/2)^2}}}. 
\end{equation}
The value of $\cos(\beta_b)$ is determined by~\eqref{eq:angle-beta}, and the value of $\beta_b$ is then determined by the sign of~$\kappa_b$: 
\begin{equation}\label{eq:beta-as-function-of-h}
    \beta_b = \left\{\begin{array}{ll}
    \arccos\big(\sin(\kappa_b/2)^2\cosh(c_b)-\cos(\kappa_b/2)^2\big),&\text{ if } \kappa_b\geq 0,\\
    2\pi-\arccos\big(\sin(\kappa_b/2)^2\cosh(c_b)-\cos(\kappa_b/2)^2\big),&\text{ if } \kappa_b\leq 0.
\end{array}\right.
\end{equation}
We now see that both $\beta_b$ and $\ell_b$ are uniquely determined by $[\HypConeStruct]$ for every pants curve $b\in {\mathcal B}$.
    
Next, we prove that $[\HypConeStruct]$ also determines the slit angles. We're assuming that the underlying pants assembly $\Upsilon$ is chained, so there are never more than two slits on a given samosa. That means we can find expressions for all the slit angles using Lemma~\ref{lem:slit-angle}:
\begin{align}
    \phi_b'&=\pm\arccos \left(\frac{\cosh(d_b)\cosh(\ell_b)-\cosh(\lambda_b)}{\sinh(d_b)\sinh(\ell_b)}\right) \label{eq:slit-angle-1-as-function-of-h} \\
    \phi_a&=\pm\left(\beta_a/2-\arccos \left(\frac{\cosh(d_a)\cosh(\ell_a)-\cosh(\lambda_a)}{\sinh(d_a)\sinh(\ell_a)}\right)\right).\label{eq:slit-angle-2-as-function-of-h}
\end{align}
The sign ambiguity in these formulas comes from the one thing Lemma~\ref{lem:slit-angle} can't tell us: which hemisphere a given slit is in. We get that information from $\varepsilon$, as explained in Remark~\ref{rem:sign-slit-angle}.

Now we just need to confirm that $[\HypConeStruct]$ determines all of the lengths and angles that appear in our expressions for the slit angles. Recall from Lemma~\ref{lem:slit-angle} that $d_a$ and $d_b$ can be expressed in terms of the corner angles $\beta$ and $\alpha$ of the samosas. On a hemisphere with a single slit, or a hemisphere with two slits in which neither slit blocks the arc from the end of the other slit to the un-slit corner, $\lambda_a$ and $\lambda_b$ can both be measured as the lengths of geodesic arcs between cone points. This makes them intrinsic parameters of $[\HypConeStruct]$. On the other hand, if each slit does block the arc from the end other slit, only one of $\lambda_a$ and $\lambda_b$ can be measured directly in this way. In this case, fortunately, the other quantity can be expressed in terms of $\xi$ and $\eta$, as described in Lemma~\ref{lem:slit-angle}. Both $\xi$ and $\eta$ are intrinsic parameters of $[\HypConeStruct]$, so we conclude in either case that $[\HypConeStruct]$ determines the values of $\lambda_a$ and $\lambda_b$. Finally, we've already shown that $\ell$ and $\beta$ are determined by $[\HypConeStruct]$. This concludes the proof that the values of the slit angles are determined by $[\HypConeStruct]$, which in turn completes the proof that the realization map is injective.

\subsubsection{Continuity of the realization map}\label{sec:proof-thm-2-continuity}
Next, we prove that the realization map is continuous. In other words, we prove that when we realize a samosa assembly $\sigma \in \Assemb_{\alpha, \Upsilon}^\varepsilon$, its image $[\HypConeStruct] = R(\sigma)$ depends continuously on the samosa assembly parameters $(\beta_b,\ell_b,\phi_b,\phi'_b)_{b\in {\mathcal B}}$. We use the cone structure $\HypConeStruct$ to represent $R(\sigma)$, inducing a smooth structure and a hyperbolic cone metric $h_\HypConeStruct$ on $\AssembSphereExt{\Upsilon}$. Identifying the realization of $\sigma$ with the cone sphere $(\AssembSphereExt{\Upsilon}, h_\HypConeStruct)$ allows us to talk about geodesics on $\AssembSphereExt{\Upsilon}$. Working locally near some initial choice of $\sigma$, we'll introduce a triangulation $\mathcal{T}$ of $\AssembSphereExt{\Upsilon}$ for which $[\HypConeStruct]$ lies in the domain of the parameterization $\varphi_\mathcal{T}\colon U_\mathcal{T}\to \R_{>0}^{6n-15}$ from Section~\ref{sec:parametrization-using-triangulations}. We'll then prove that the composition $\varphi_{\mathcal{T}}\circ R$ is continuous at $\sigma$. This implies that $R$ is continuous at $\sigma$, because by Proposition~\ref{length-homeo} the map $\varphi_{\mathcal T}$ is a homeomorphism onto its image.

We'll construct $\mathcal{T}$ in two steps. First, we'll decompose $\AssembSphereExt{\Upsilon}$ into elementary pieces whose boundaries are made of valid triangulation edges. Then, we'll triangulate each piece individually. The decomposition goes like this. For each pants curve $b\in {\mathcal B}$, take the simple closed geodesic that starts and ends at the whole singularity and lies in the homotopy class of $b$, as described in Lemma~\ref{lem:kink-angle}. Its length is $c_b$. Cutting $\AssembSphereExt{\Upsilon}$ along each of these geodesics splits it into two \emph{joker's hats} and $n-4$ \emph{V-pieces}, as defined for instance in~\cite{DryPar}. Equation \eqref{eq:length-kinked-geodesic} shows that $c_b$ depends continuously on the samosa assembly parameters $\ell_b$ and $\beta_b$ since we can write
\begin{equation}\label{eq:length-kinked-geodesic-isolated}
c_b=\arccosh\big( \cosh(\ell_b)^2-\sinh(\ell_b)^2\cos(\beta_b)\big).
\end{equation}

\subsubsection{Triangulating joker's hats}\label{sec:proof-thm-2-triangulating-joker's-hats}
We triangulate each joker's hat using three hyperbolic triangles. Here's a picture of a joker's hat, which is already triangulated as a preview.
\begin{center}
\vspace{2mm}
\begin{tikzpicture}[scale=1.1, every node/.style={inner sep=0.2mm}]
\node[anchor=south west, inner sep=0mm] {\includegraphics[width=8.8cm]{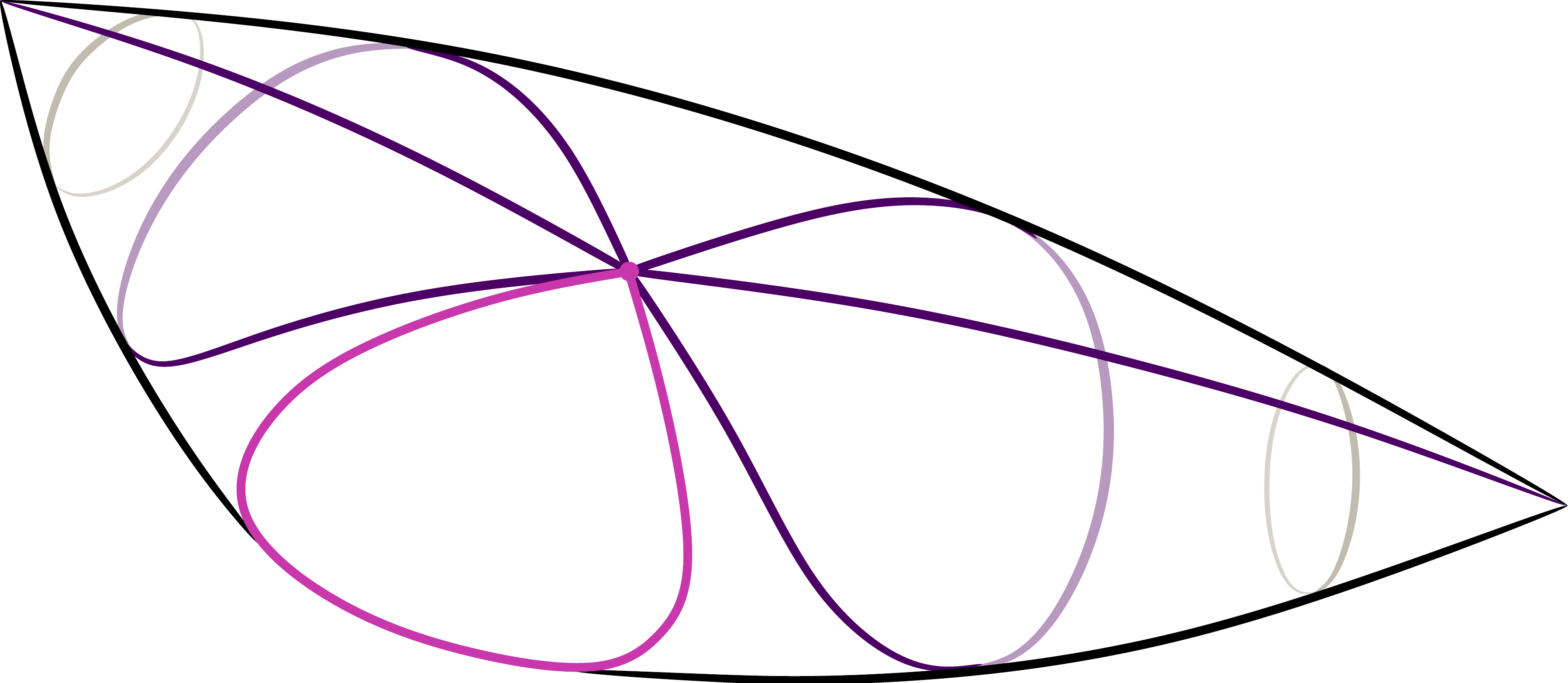}};

\node[black, anchor=east] at (-0.05, 3.55) {$q$};
\node[black, anchor=west] at (8.05, 0.85) {$p$};

\node[plum, anchor=north] at (4.8, 1.85) {$\lambda_p$};
\node[plum, anchor=north] at (1.9, 2.75) {$\lambda_q$};
\node[plum, anchor=north] at (4.4, 0.8) {$\delta_p$};
\node[plum, anchor=north] at (0.9, 2.15) {$\delta_q$};
\node[mauve, anchor=north east] at (2, 0.3) {$c_b$};

\node[darksmoked, anchor=north east] at (0.2, 2.7) {$2\pi - \alpha_q$};
\node[darksmoked, anchor=north west] at (6.8, 0.4) {$2\pi - \alpha_p$};
\end{tikzpicture}
\vspace{2mm}
\end{center}
We label the un-slit corners $p$ and $q$. The boundary is a geodesic homotpic to the pants curve $b\in {\mathcal B}$. The first two edges of the triangulation are the geodesic loops that enclose the corners $p$ and $q$, starting and ending at the whole singularity. Let's call their lengths $\delta_{p}$ and $\delta_{q}$, respectively. They split the joker's hat into two cones, with cone angles $2\pi - \alpha_p$ and $2\pi - \alpha_q$, and one triangle, with side lengths, $(\delta_p, \delta_q, c_b)$. The other two edges are the geodesic segments that run from the whole singularity to $p$ and $q$. Let's call their lengths $\lambda_p$ and $\lambda_q$, respectively. The edge to $p$ cuts the cone around $p$ into an isosceles triangle with side lengths $(\lambda_p, \lambda_p, \delta_p)$, where the matching sides make an angle of $2\pi - \alpha_p$. We get an analogous triangle at corner $q$.

The geodesics from the whole singularity to the corners $p$ and $q$ lie within a single samosa, even though the joker's hat might spill over the boundary of that samosa a bit. Their lengths, $\lambda_p$ and $\lambda_q$, are determined by the parameters of the samosa they lie in. Specifically, their lengths are determined by the trigonometry of the hemisphere they lie in, shown below.
\begin{center}
\vspace{1mm}
\begin{tikzpicture}[scale=.9, every node/.style={inner sep=1.2mm}]
\node[anchor=south west, inner sep=0mm] at (-3.21, -0.03) {\reflectbox{\includegraphics[width=4.05cm, angle=-90]{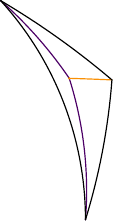}}};
\coordinate (P) at (5.59, 1.06);
\coordinate (Q) at (-3.18, 4.46);
\coordinate (S) at (-0.06, 1.7);
\draw[apricot] ++(12:0.45) arc (12:88:0.45) node[apricot, midway, anchor=south west] {$\phi_b$};
\draw ++(14:0.3) arc (14:121:0.3) node[anchor=15, inner sep=1.5mm] {$\beta_b/2$};
\path (0, 0) -- (S) node[apricot, pos=0.55, anchor=east] {$\ell_b$};
\path (0, 0) edge[draw=none, bend right=6] node[midway, anchor=120] {$d_q$} (P);
\path (0, 0) edge[draw=none, bend left=4.5] node[pos=0.6, anchor=60] {$d_p$} (Q);
\path (S) edge[draw=none, bend right=9] node[plum, pos=0.4, anchor=80] {$\lambda_p$} (P);
\path (S) edge[draw=none, bend left=6.5] node[plum, pos=0.25, anchor=70, inner sep=0.5mm] {$\lambda_q$} (Q);
\node[anchor=west] at (P) {$p$};
\node[anchor=east] at (Q) {$q$};
\end{tikzpicture}
\end{center}

With a calculation similar to the one we did in Lemma~\ref{lem:slit-angle}, we can now write manifestly continuous formulas for $\lambda_p$ and $\lambda_q$ in terms of the samosa assembly parameters:
\begin{align}
\lambda_p&=\arccosh\big( \cosh(d_{q})\cosh(\ell_b)-\cos(\phi_b)\sinh(d_{q})\sinh(\ell_b)\big ) \label{eq:length-lambda-geodesic-1-isolated} \\
\lambda_q&=\arccosh\big( \cosh(d_{p})\cosh(\ell_b)-\cos(\beta_b/2-|\phi_b|)\sinh(d_{p})\sinh(\ell_b)\big ). \label{eq:length-lambda-geodesic-2-isolated}
\end{align}
Similarly, the hyperbolic law of cosines---formula~\eqref{eq:hyperbolic-law-of-cosines-lengths} in Appendix~\ref{apx:trig-formulae}---applied to the isosceles triangle with side lengths $(\lambda_p, \lambda_p, \delta_p)$ gives
\begin{equation}
\delta_p=\arccosh\big( \cosh(\lambda_p)^2-\cos(\alpha_p)\sinh(\lambda_p)^2\big).\label{eq:length-delta-geodesic-isolated}
\end{equation}
The analogous formula for $\delta_q$ holds too. We conclude that the length of each arc used to triangulate each joker's hat depends continuously on the samosa assembly parameters.

\subsubsection{Triangulating V-pieces}\label{sec:proof-thm-2-triangulating-V-pieces} We triangulate each V-piece using four hyperbolic triangles. A V-piece sits mostly or entirely within a samosa, and the triangulation we use depends on the shape of that samosa. Let's say $a, b \in \mathcal{B}$ are the pants curves homotopic to the boundary components of the V-piece. We look at four geodesic arcs connecting the whole singularities to the un-slit corner:
\begin{itemize}
\item[] $\zeta_a$, $\zeta_b$. The arcs that stay within the same hemisphere.
\item[] $\zeta'_a$, $\zeta'_b$. The arcs that cross into the other hemisphere and stay there.
\end{itemize}
Each arc is labeled with the pants curve associated with the whole singularity it starts at. The arrangement of these four arcs determines which triangulation we use. There are two cases to consider:
\begin{itemize}
\item[] \textbf{Case~$\text{V}_1$.} The slits don't block any of the arcs described above.
\item[] \textbf{Case~$\text{V}_2$.} Each slit blocks one of the arcs described above.
\end{itemize}
When the slits are in the same hemisphere, these cases correspond to the two samosa configurations illustrated below Equation~\eqref{eq:first-slit-angle-special-case-two-slits} in the statement of Lemma~\ref{lem:slit-angle}. We'll describe the triangulation carefully in Case~$\text{V}_1$, and we'll sketch it in Case~$\text{V}_2$.

\subsubsection{Triangulating a V-piece in Case~$\text{V}_1$}\label{sec:proof-thm-2-triangulating-V-pieces-case-V1} Let's start with Case~$\text{V}_1$. The first two edges of the triangulation are the boundary components; call their lengths $c_a$, $c_b$. The next four edges are $\zeta_a$, $\zeta_b$, $\zeta'_a$, $\zeta'_b$; call their lengths $\lambda_{a}$,$\lambda_b$, $\lambda'_a$, $\lambda'_b$. The last edge is the shortest geodesic arc connecting the two whole singularities; call its length $\xi$.
\begin{center}
\begin{tikzpicture}[scale=1.2, every node/.style={inner sep=0.2mm}]
\node[anchor=south west, inner sep=0mm] {\includegraphics[width=7.2cm]{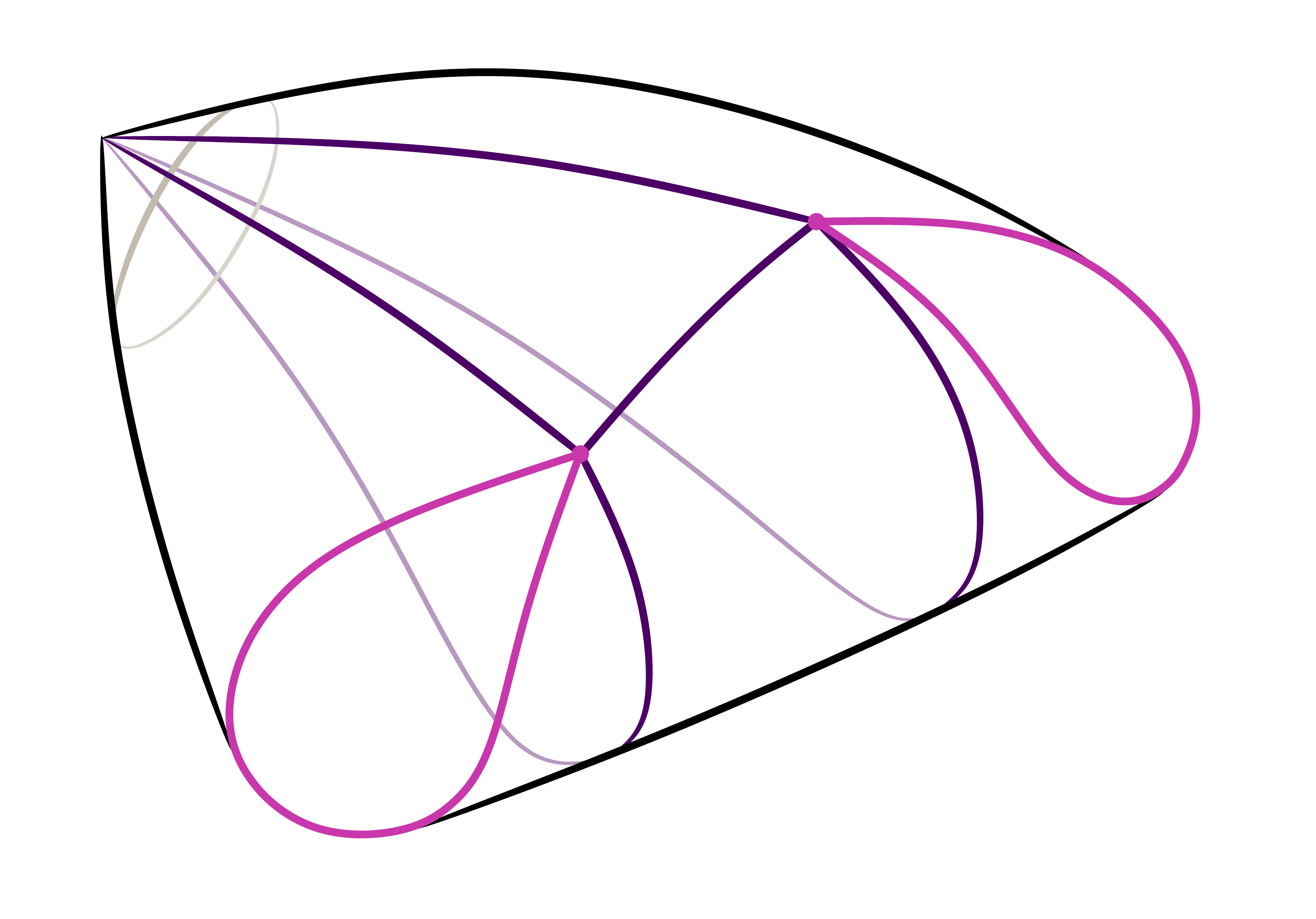}};

\node[plum, anchor=north east] at (1.95, 2.75) {$\lambda_a$};
\node[plum, anchor=north east] at (2.9, 3.35) {$\lambda_b$};
\node[plum, anchor=north west] at (3.2, 2.75) {$\xi$};
\node[plum, anchor=south west] at (3.0, 1.2) {$\lambda'_a$};
\node[plum, anchor=north east] at (4.35, 2.4) {$\lambda'_b$};
\node[mauve, anchor=north east] at (1.5, 0.4) {$c_a$};
\node[mauve, anchor=north east] at (5.7, 2.8) {$c_b$};

\begin{scope}[shift={(5.6, -0.27)}]
\node[anchor=south west, inner sep=0mm] {\includegraphics[width=7.41cm]{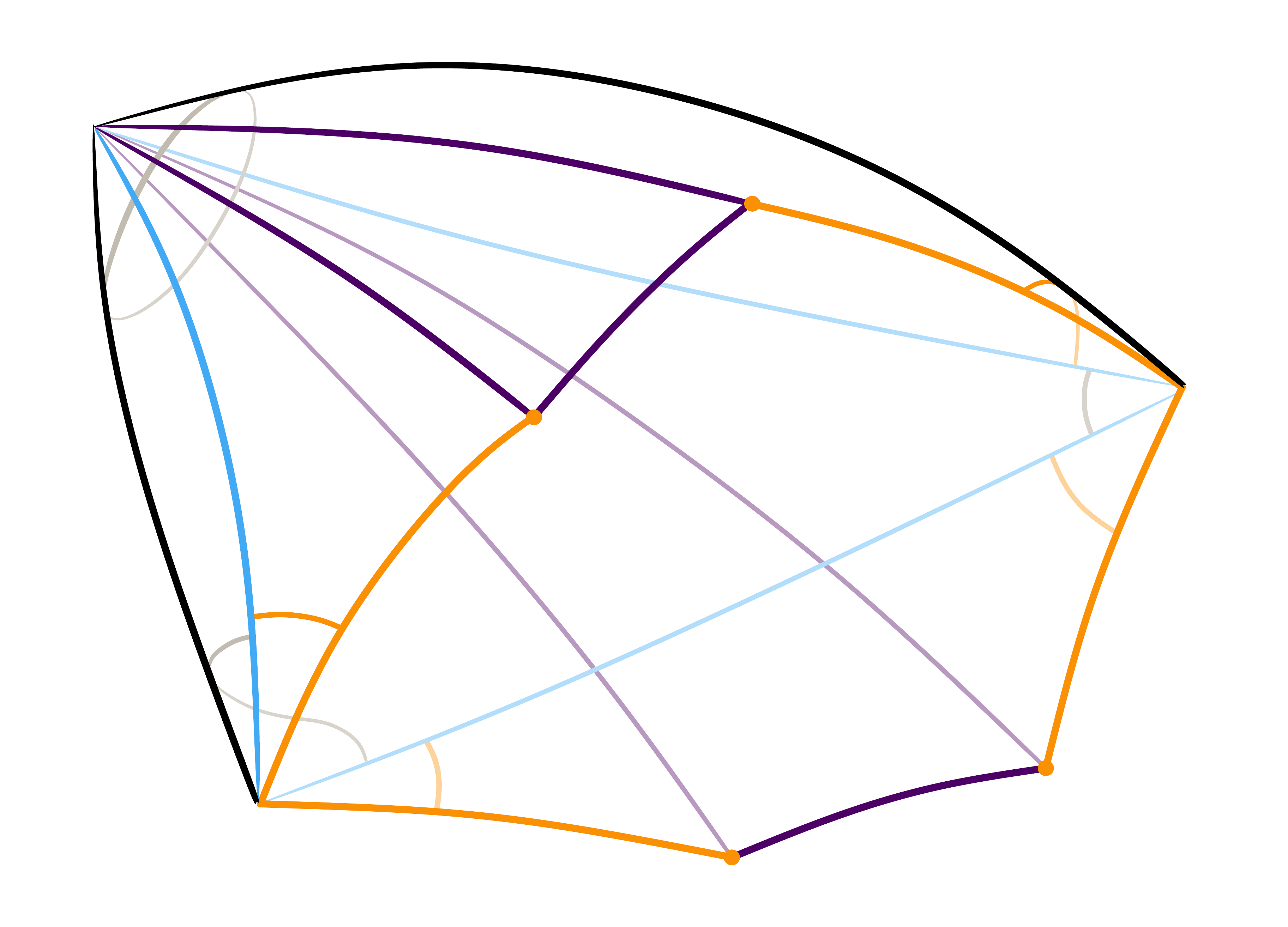}};

\begin{scope}[shift={(-0.1, 0.3)}]
\node[plum, anchor=north east] at (1.95, 2.75) {$\lambda_a$};
\node[plum, anchor=north east] at (2.9, 3.35) {$\lambda_b$};
\node[plum, anchor=north west] at (3.2, 2.75) {$\xi$};
\end{scope}
\node[lightplum, anchor=south west] at (3.2, 0.75) {$\lambda'_a$};
\node[lightplum, anchor=north east] at (4.35, 1.45) {$\lambda'_b$};
\node[plum, anchor=north west] at (4.2, 0.55) {$\xi$};
\node[apricot, anchor=south east] at (2, 1.95) {$\ell_a$};
\node[apricot, anchor=north east] at (4.19, 3.7) {$\ell_b$};
\node[apricot, anchor=south east] at (2.85, 0.1) {$\ell_a$};
\node[apricot, anchor=north east] at (5.62, 1.5) {$\ell_b$};
\node[sky, anchor=north east] at (1.05, 2.7) {$d_a$};
\node[lightsky, anchor=north east] at (4, 2.9) {$d_b$};
\node[lightsky, anchor=north east] at (3.45, 1.85) {$d_{ab}$};

\node[lightapricot, anchor=west] at (2.15, 0.77) {$\phi_a$};
\node[apricot, anchor=west] at (1.6, 1.31) {\tiny $\beta_a/2-\phi_a$};
\node[darksmoked, anchor=east] at (0.87, 1.23) {\small $\beta_a/2$};
\node[apricot, anchor=west] at (5.2, 3.2) {$\phi_b'$};
\node[lightapricot, anchor=east] at (6.6, 1.9) {\small $\beta_b'/2-\phi_b'$};
\node[darksmoked, anchor=east] at (5.15, 2.5) {\small $\beta_b'/2$};
\end{scope}
\end{tikzpicture}
\end{center}
We already know from~\eqref{eq:length-kinked-geodesic-isolated} that $c_{a}$ and $c_{b}$ depend continuously on the samosa assembly parameters. A trigonometric computation, analogous to the one in Lemma~\ref{lem:slit-angle}, shows that
\begin{align*}
\lambda_a&=\arccosh\big (\cosh(d_a)\cosh(\ell_{a})-\cos(\beta_a/2-|\phi_{a}|)\sinh(d_a)\sinh(\ell_{a})\big) \\
    \lambda_b&=\arccosh\big (\cosh(d_b)\cosh(\ell_{b})-\cos(\phi_{b}')\sinh(d_b)\sinh(\ell_{b})\big).
\end{align*}
We can find $\lambda_a'$ and $\lambda_b'$, using similar trigonometry, after cutting the samosa along the slits and the arc joining the whole singularities. This lets us flatten out the triangle whose sides are $\zeta'_a$, $\zeta'_b$, and the arc joining the whole singularities. Studying the resulting picture, we can figure out that
\begin{align*}
\lambda_a'&=\arccosh\big (\cosh(d_a)\cosh(\ell_{a})-\cos(\beta_{a}/2+|\phi_{a}|)\sinh(d_a)\sinh(\ell_{a})\big) \\
    \lambda_b'&=\arccosh\big (\cosh(d_b)\cosh(\ell_{b})-\cos(\beta_{b}'-|\phi_{b}'|)\sinh(d_b)\sinh(\ell_{b})\big).
\end{align*}
The value of $\xi$ is one the four side lengths of a hyperbolic quadrilateral. We can write a manifestly continuous formula for it in terms of the lengths $\ell_{a}, \ell_{b}$, $d_{ab}$ of the three other sides, and the two interior angles $\phi_{a}$ and $\beta_b'/2-\phi_{b}'$. Doing these computations, which we omit, proves that all of the edge lengths in the triangulation of each V-pieces depend continuously on the samosa assembly parameters.

\subsubsection{Triangulating a V-piece in Case~$\text{V}_2$}\label{sec:proof-thm-2-triangulating-V-pieces-case-V2}
If, instead, we're in Case~$\text{V}_2$, then either the arcs $\zeta_a$, $\zeta'_b$ or the arcs $\zeta'_a$, $\zeta_b$ are blocked by the slits, and therefore not realized as unbroken geodesic arcs on $\AssembSphereExt{\Upsilon}$. We replace the two broken arcs with the geodesic arcs that connect the two whole singularities by enclosing the boundary components.
\begin{center}
\begin{tikzpicture}[scale=1.33, every node/.style={inner sep=0.2mm}]

\node[anchor=south west, inner sep=0mm] {\includegraphics[width=8cm]{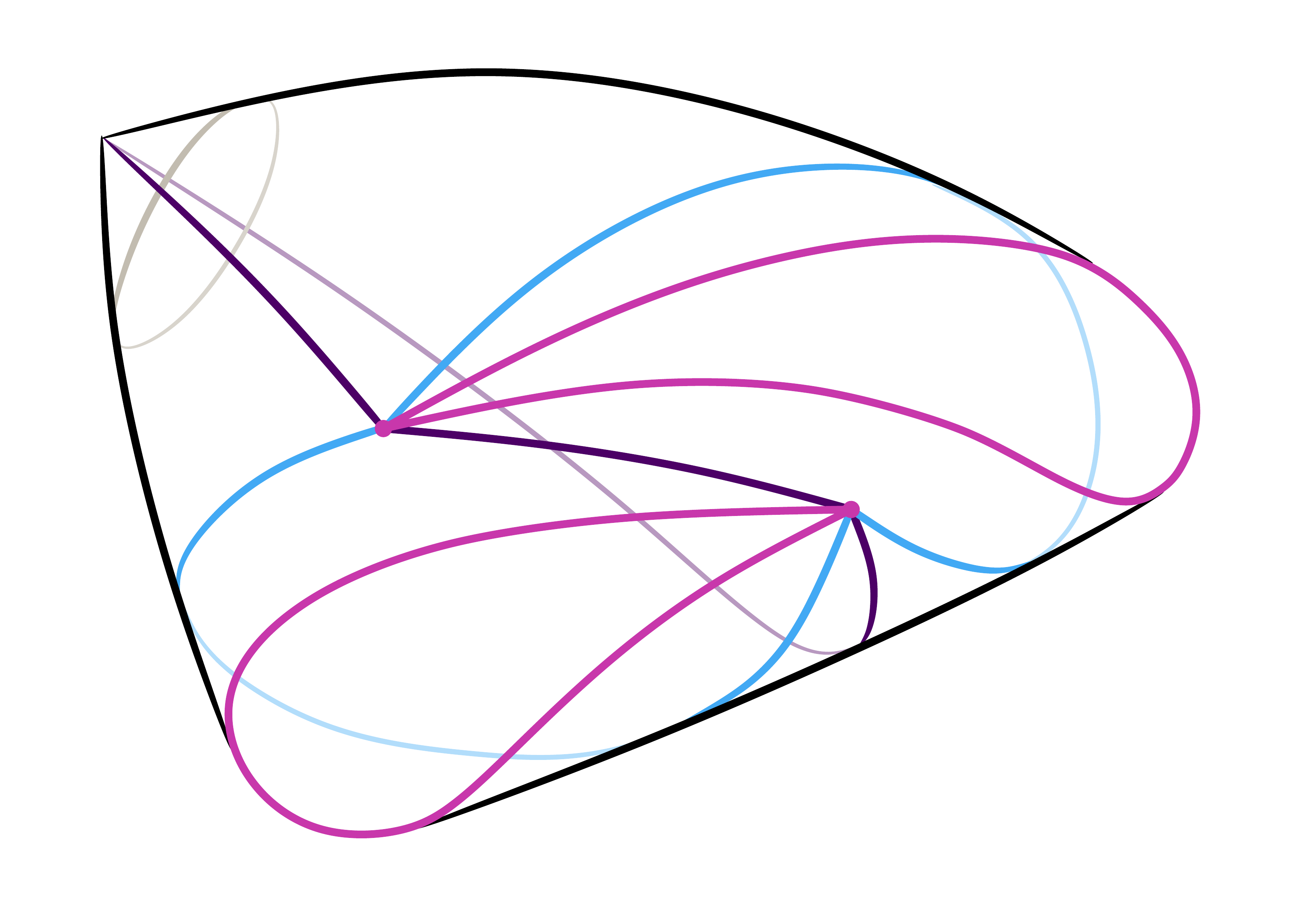}};

\node[mauve, anchor=north east] at (1.5, 0.4) {$c_a$};
\node[mauve, anchor=north east] at (5.9, 2.5) {$c_b$};
\node[plum, anchor=north east] at (1.4, 2.75) {$\lambda_b$};
\node[plum, anchor=north west] at (4.0, 1.3) {$\lambda_a'$};
\node[plum, anchor=north west] at (2.0, 2.17) {$\xi$};
\end{tikzpicture}
\end{center}
You can find the lengths of these new triangulation edges through trigonometric calculations, similar to the ones we did in Case~$\text{V}_1$. After some algebra, you'll again find manifestly continuous formulas for the lengths in terms of the samosa assembly parameters. This is enough to show that $R$ is continuous at the samosa assembly $\sigma$ that we started with.

\subsubsection{Continuity of the inverse map}\label{sec:proof-thm-2-continuity-inverse-map}
We now prove that the inverse map of the realization map, defined on $R(\Assemb_{\alpha, \Upsilon}^\varepsilon)\subset\HypCone{\alpha}(\AssembSphereExt{\Upsilon})$, is also continuous. To do so, we show that the samosa assembly parameters $(\beta_b,\ell_b,\phi_b,\phi'_b)_{b\in {\mathcal B}}$ of $R^{-1}([\HypConeStruct])$ depend continuously on $[\HypConeStruct]\in R(\Assemb_{\alpha, \Upsilon}^\varepsilon)$. We can relate $\ell_b$ to $c_b$ and $\kappa_b$ by~\eqref{eq:slit-length-as-function-of-h}. Proposition~\ref{length-homeo} says that $c_b$ depends continuously on $[\HypConeStruct]$. We can use hyperbolic trigonometry to express the kink angle $\kappa_b$ in terms of triangulation edge lengths, so Proposition~\ref{length-homeo} implies that it too depends continuously on $[\HypConeStruct]$. This shows that $\ell_b$ depends continuously on $[\HypConeStruct]$ for every $b\in {\mathcal B}$. The value of $\beta_b$ is expressed continuously in terms of $c_b$ and $\kappa_b$ by~\eqref{eq:beta-as-function-of-h}, showing that $\beta_b$ varies continuously in $[\HypConeStruct]$.

Finally, to prove that the slit angles $\phi_b$ and $\phi_b'$ depend continuously on $[\HypConeStruct]$, we use~\eqref{eq:slit-angle-1-as-function-of-h} and~\eqref{eq:slit-angle-2-as-function-of-h}. These equations express the slit angles as continuous functions of $\alpha$, $\beta$, $\ell$, $\lambda$, and sometimes $\xi$ and $\eta$. We already established the continuous dependence of $\ell_b$ and $\beta_b$ on $[\HypConeStruct]$. From Proposition~\ref{length-homeo}, we deduce that both $\lambda$ and $\xi$ depend continuously on $[\HypConeStruct]$. Since $\eta$ is an angle of the triangulation, the hyperbolic law of cosines---formula~\eqref{eq:hyperbolic-law-of-cosines-lengths} from Appendix~\ref{apx:trig-formulae}---shows that it too depends continuously on $[\HypConeStruct]$. This finishes the proof that $R^{-1}$ is continuous.

\subsubsection{Smoothness of the realization map and its inverse}\label{sec:proof-thm-2-smoothness}
It remains to be shown that if Conjecture~\ref{smooth-atlas} holds, the realization map is a diffeomorphism onto its image. The argument is similar to the proof that the realization map is continuous.

Using Conjecture~\ref{smooth-atlas}, we can prove that $R$ is smooth by showing that the edge lengths of various triangulations depend smoothly on the samosa parameters. We use~\eqref{eq:length-kinked-geodesic-isolated}, \eqref{eq:length-lambda-geodesic-1-isolated}, \eqref{eq:length-lambda-geodesic-2-isolated}, \eqref{eq:length-delta-geodesic-isolated} for that. The non-smoothness of \eqref{eq:length-lambda-geodesic-2-isolated} at $\phi_b = 0$ does no harm because we're working on the domain $\Assemb_{\alpha, \Upsilon}^\varepsilon$, which prescribes a hemisphere for each slit, and hence a sign for each slit angle.

To prove that the inverse map is smooth, it's enough to express the samosa parameters as smooth functions of the same triangulation edge lengths. Here, we use~\eqref{eq:slit-length-as-function-of-h}, \eqref{eq:beta-as-function-of-h}, \eqref{eq:slit-angle-1-as-function-of-h}, \eqref{eq:slit-angle-2-as-function-of-h}. For instance, we see from~\eqref{eq:beta-as-function-of-h} that $\beta_b$ depends smoothly on the parameters $\kappa_b$ and $c_b$, except maybe when $\kappa_b=0$ or when
\[ \sin(\kappa_b/2)^2\cosh(c_b)-\cos(\kappa_b/2)^2=-1. \]
Using trigonometric identities, we can see that $\sin(\kappa_b/2)^2\cosh(c_b)-\cos(\kappa_b/2)^2\geq -1$, with equality if and only if $\kappa_b=0$. To investigate the situation at $\kappa_b = 0$, we can write
\[ \sin(\beta_b)=\sin(\kappa_b/2)\frac{\sinh(c_b)}{\sinh(\ell_b)} \]
using the hyperbolic law of sines---formula~\eqref{eq:hyperbolic-law-of-sines} from Appendix~\ref{apx:trig-formulae}. This alternate expression for $\beta_b$ shows that it depends smoothly on $\kappa_b$, $c_b$ and $\ell_b$ around $\kappa_b=0$ too, because $c_b$ and $\ell_b$ stay bounded away from zero in this region. The treatment of the other variables is similar.

\section{Unfolding hyperbolic cone spheres}\label{sec:unfolding}
\subsection{Overview}
Sometimes, a samosa assembly can be unfolded onto the hyperbolic plane, giving a convenient net for the hyperbolic cone sphere that the samosa assembly represents. This net, as we'll see in the next section, is related to Maret's triangle chain parameterization of DT representations. It may also be convenient for numerical work, because it gives a fundamental domain for the hyperbolic cone sphere and enables explicit computations of the holonomy map through the triangle chain construction.
\subsection{Unfolding samosas}\label{sec:hamantash}
Some samosas can be cut and unfolded onto the hyperbolic plane in a particularly neat and consistent way.
\begin{defn}\label{defn:hamantash}
A {\em hamantash} is a slit samosa that can be unfolded in a certain way. It's defined by the following properties.
\begin{itemize}
\item All of its slits are in the same hemisphere. Recall that the equator is in both hemispheres.
\item Each slit can be extended geodesically to the equator without hitting another slit.
\end{itemize}
\end{defn}
Intuitively, a hamantash is a slit samosa whose slits lie in the same hemisphere and are short enough to avoid each other's continuations. We can make this intuition precise by recasting the continuation condition as an upper bound on each slit length. The bounds depend on of the number of slits. In this paper, we'll only consider hamantashen with one or two slits. For one slit, the bound is given by Lemma \ref{lem:bound-on-l-one-slit}.

\begin{lem}\label{lem:bound-on-l-two-slits}
A samosa with two slits in the same hemisphere is a hamantash as long as its slit angles satisfy $\phi_a\neq 0$ and $\phi_b'\neq \pm\beta_b'/2$, and its slit lengths $\ell_a$ and $\ell_b$ don't exceed the respective bounds $\ell_{a,\textup{max}}$ and $\ell_{b,\textup{max}}$ given by the formulas
\begin{align*}
\coth (\ell_{a,\textup{max}})&=\frac{1}{\sinh(d)}\big(\cos(\phi_a)\cosh(d)+|\sin(\phi_a)|\cot(\beta_b'/2-|\phi_b'|)\big) \\
\coth (\ell_{b,\textup{max}})&=\frac{1}{\sinh(d)}\big(\cos(\beta_b'/2-|\phi_b'|)\cosh(d)+\sin(\beta_b'/2-|\phi_b'|)\cot(|\phi_a|)\big),
\end{align*}
where $d$ is the length of the equator segment joining the two slit corners. If $\phi_a=0$ or $\phi_b'= \pm\beta_b'/2$, then one of the slit lies on an equator segment and its continuation meets the other slit at its starting point.
\end{lem}
\begin{proof}
Apply the four-parts formula, listed as formula~\eqref{eq:four-part-formula} in Appendix~\ref{apx:trig-formulae}.
\begin{center}
\begin{tikzpicture}[scale=1, every node/.style={inner sep=1mm}]
\node[anchor=south west, inner sep=0mm] at (-2.44, -8.47) {\reflectbox{\includegraphics[width=4.5cm]{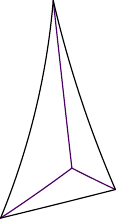}}};
\coordinate (A) at (-2.4, -7.33);
\coordinate (B) at (2.04, -8.43);
\coordinate (M) at (-0.72, -6.5);

\draw[apricot] (A) ++(22:0.6) arc (22:-10:0.6) node[inner sep=1.5mm, anchor=100] {$\phi_a$};
\draw (A) ++(-9:0.3) arc (-9:61:0.3) node[anchor=-10] {$\beta_a/2$};
\draw (B) ++(116:0.3) arc (116:163:0.3) node[anchor=60] {$\beta_b'/2$};
\draw[apricot] (B) ++(144:0.6) arc (144:115:0.6) node[anchor=200] {$\phi_b'$};

\path (A) edge[draw=none, bend right=1] node[pos=0.45, inner sep=1.3mm, anchor=north] {$d$} (B);
\path[plum] (A) -- (M) node[inner sep=0.5mm, pos=0.5, anchor=north west] {$\ell_{a,\text{max}}$};
\path[plum] (B) -- (M) node[inner sep=0mm, pos=0.6, anchor=south west] {$\ell_{b,\text{max}}$};
\end{tikzpicture}
\end{center}
\end{proof}

Unfolding a hamantash produces a geodesic polygon (not necessarily convex). To unfold a hamantash with one slit, we extend the slit geodesically until it hits the equator. Then we cut the hamantash along the extended slit and the equator segment that it hits. This splits the hemisphere containing the slit into two flaps, which we flip over the equator, producing a hyperbolic pentagon. In the limiting case where the slit runs along the equator, we end up cutting along two segments of the equator and flipping an entire hemisphere over the third segment, unfolding the hamantash into a hyperbolic quadrilateral.
\begin{center}
\vspace{2mm}
\begin{tikzpicture}[scale=0.9]
\node[anchor=south west, inner sep=0mm] {\includegraphics[width=7.2cm]{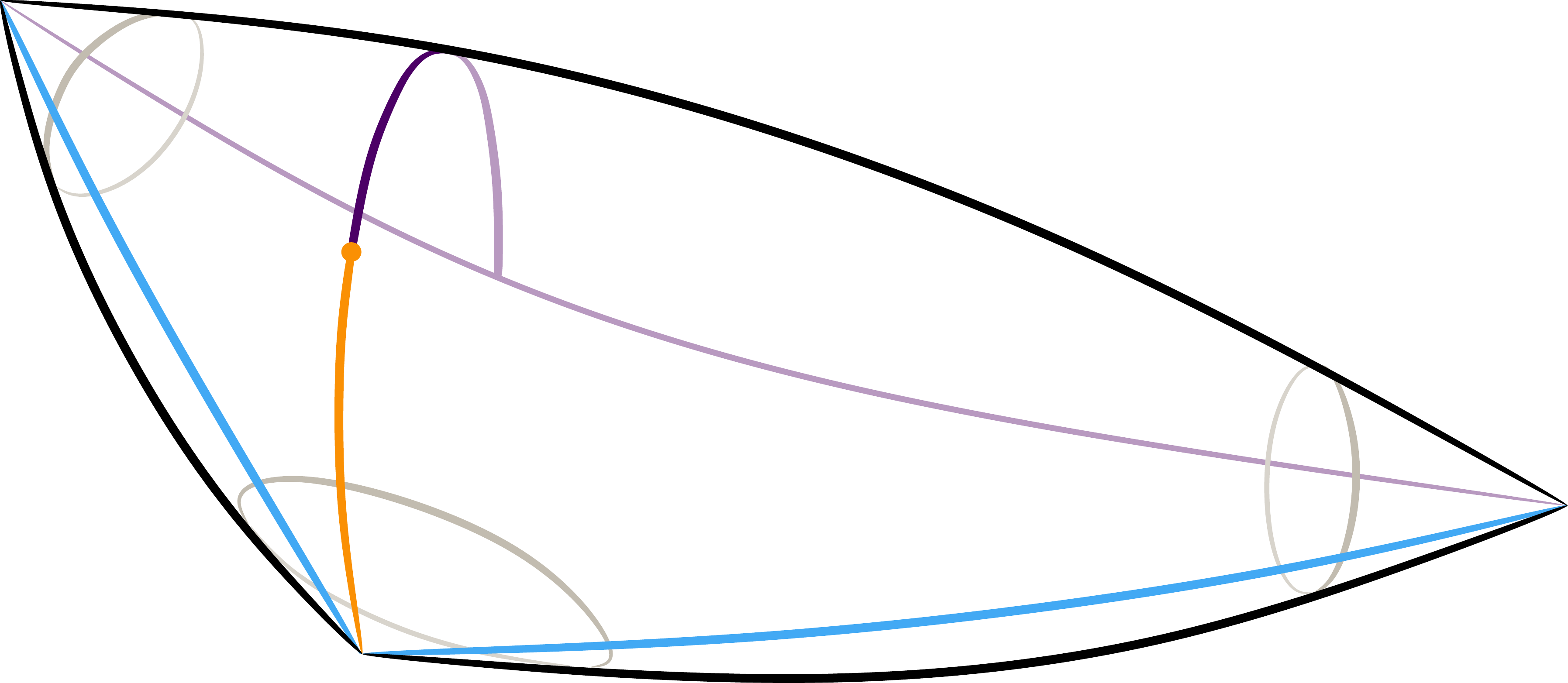}};
\node[anchor=south west, inner sep=0mm] at (-2.5, -6) {\reflectbox{\includegraphics[width=5.4cm, angle=-90]{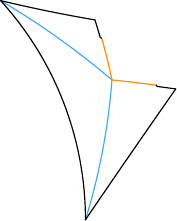}}};
\draw[plum] (0.6, 1) edge[->, bend right=10] node[circle, anchor=-45] {unfold} (-0.9, -0.7);
\end{tikzpicture}
\end{center}

To unfold a hamantash with two slits, we extend the slits geodesically until they meet at a point, and we add a geodesic segment from the meeting point to the un-slit corner. Then we cut the hamantash along the extended slits and the added segment. This splits the hemisphere containing the slits into three flaps, which we flip over the equator, producing a hyperbolic hexagon. In the special case where one or both of the slits run along the equator, the hamantash unfolds into a hyperbolic pentagon or quadrilateral, respectively.
\begin{center}
\vspace{2mm}
\begin{tikzpicture}[scale=0.9]
\node[anchor=south west, inner sep=0mm] {\includegraphics[width=5.4cm]{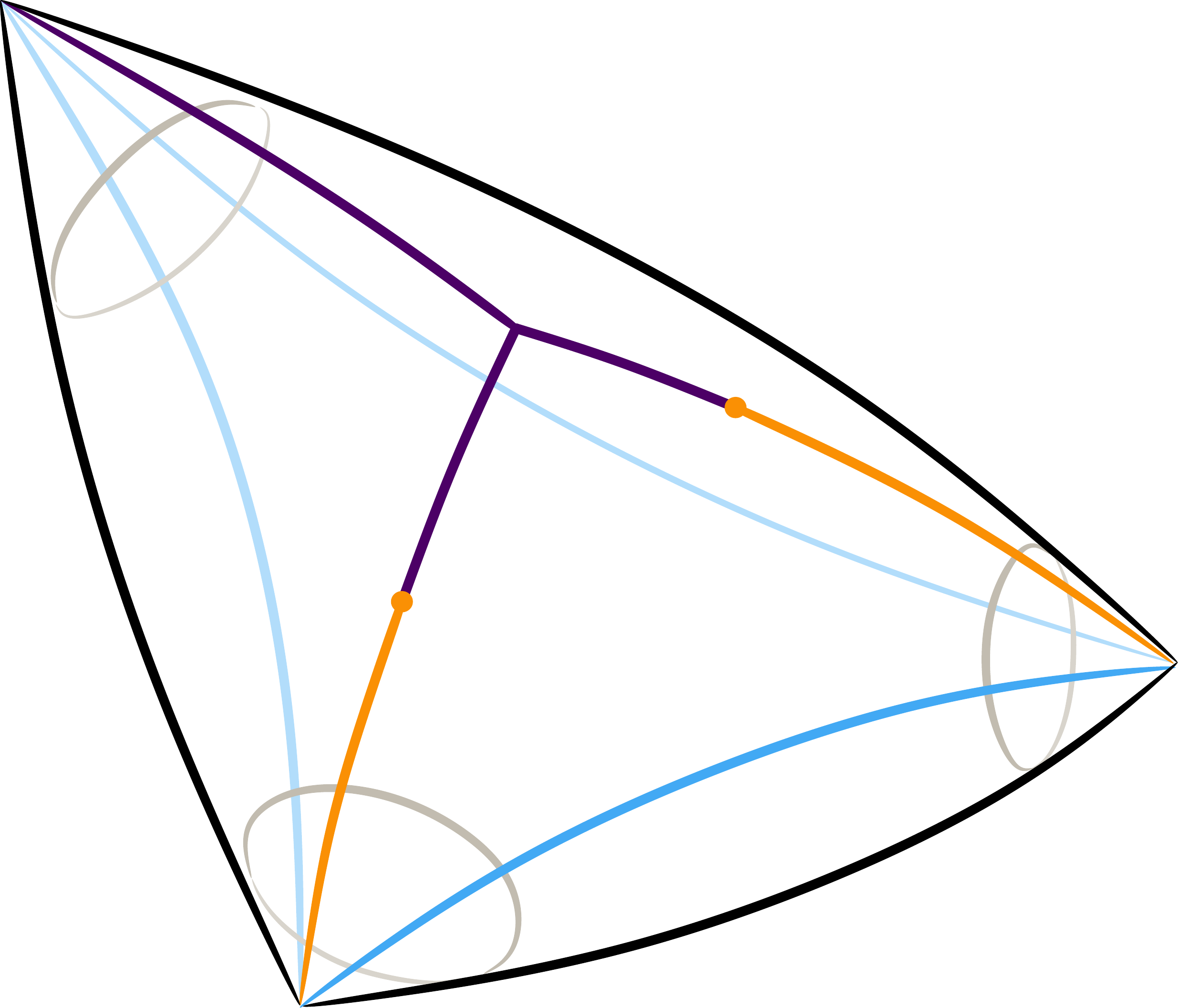}};
\node[anchor=south west, inner sep=0mm] at (-6, -1) {\reflectbox{\includegraphics[width=4.05cm]{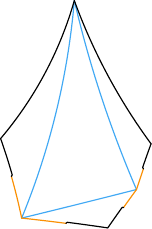}}};
\draw[plum] (-0.25, 3.6) edge[->, bend right=10] node[circle, anchor=-70] {\textnormal{unfold}} (-2, 3);
\end{tikzpicture}
\vspace{1mm}
\end{center}

We embed an unfolded hamentash in the hyperbolic plane by putting its interior side face-up. Equivalently, we embed it by an orientation-reversing local isometry.

\begin{rem}
If a hamantash has its slits in the northern hemisphere, the southern hemisphere stays intact when we unfold. By definition, the equator runs counterclockwise around the southern hemisphere as seen from the outside of the hamantash, so our orientation-reversing isometry embeds it as a clockwise loop in the hyperbolic plane. Similarly, if a hamantash has its slits in the southern hemisphere, the equator gets embedded as a counterclockwise loop in the hyperbolic plane.

If a hamantash has all of its slits on the equator, both of the conditions above hold. Both hemispheres stay intact when we unfold, embedding two copies of the equator in the hyperbolic plane---one clockwise and one counterclockwise.
\end{rem}
\subsection{Unfolding samosa assemblies}\label{sec:hamantash-assemblies}
Now that we know how to unfold a single samosa, we can say what it means for a samosa assembly to be unfoldable.
\begin{defn}\label{defn:hamantash-assembly}
A samosa assembly is called a {\em hamantash assembly} if the realization of each pair of pants into a slit samosa actually produces a hamantash (Definition~\ref{defn:hamantash}).
\end{defn}
The subspace of hamantash assemblies is denoted $\Haman_{\alpha, \Upsilon}\subset \Assemb_{\alpha, \Upsilon}$. A hamantash assembly can be unfolded to a collection of hyperbolic polygons embedded in the hyperbolic plane. Using side identifications, all the pieces can be glued back together to produce the realization of the hamantash assembly. The process of gluing the hamantashen along their slits commutes with the processes of unfolding and re-folding the hamantashen. In particular, we can glue the hamantashen along their slits while they're laid out, unfolded, on the hyperbolic plane. This produces a net for the realization---although the net might overlap itself.
\begin{center}
\vspace{2mm}
\reflectbox{\includegraphics[width=9cm]{fig/net-chain.pdf}}
\vspace{2mm}
\end{center}
When none of the slits lie on equator segments opposite from un-slit corners, as discussed in Section~\ref{sec:marking}, the realized hamantash assembly comes with a canonical isotopy class of homeorphisms to the generic realization $\AssembSphereExt{\Upsilon}$. Thus, the process of unfolding, gluing along the slits, and then re-folding is the restriction $R \colon \Haman_{\alpha, \Upsilon}^\text{gen} \to \HypCone{\alpha}(\AssembSphereExt{\Upsilon})$ of the realization map, where $\Haman_{\alpha, \Upsilon}^\text{gen} = \Assemb_{\alpha, \Upsilon}^\text{gen} \cap \Haman_{\alpha, \Upsilon}$.

\section{Parameterizing DT representations}\label{chap:parametrizing-DT}
\subsection{The holonomies of hamantashen}\label{sec:holonomy-hamantash}
When we use hamantash assemblies to parameterize hyperbolic cone structures on $\AssembSphereExt{\Upsilon}$, the holonomy map has a nice description. It just erases some details of the unfolded hamantash assembly, turning it into a triangle chain that specifies a DT representation as per Section~\ref{sec:dt-review}.

Given a chained pants assembly $\Upsilon$, let $\Haman_{\alpha, \Upsilon}^{\textrm{north}}\subset \Haman_{\alpha, \Upsilon}$ be the subset consisting of hamantash assemblies whose slits all lie in northern hemispheres. The realization map extends continuously from $\Haman_{\alpha, \Upsilon}^\text{gen}$ to $\Haman_{\alpha, \Upsilon}^\text{north}$. To see why, first recall from Section~\ref{sec:marking} why the realization map fails to extend from $\Assemb_{\alpha, \Upsilon}^\text{gen}$ to $\Assemb_{\alpha, \Upsilon}$. The problem is that a non-generic samosa assembly comes with several isometry classes of identifications with $\AssembSphereExt{\Upsilon}$, and the choice between them can't be made continuously throughout $\Assemb_{\alpha, \Upsilon}$. In $\Haman_{\alpha, \Upsilon}^\text{north}$, we can make a continuous choice of identification by pushing all the slits infinitesimally into the interior of the northern hemisphere.

The holonomy map
\[ \Haman_{\alpha, \Upsilon}^{\textrm{north}}\overset{R}{\longrightarrow}\HypCone{\alpha}(\AssembSphereExt{\Upsilon}) \overset{\hol}{\longrightarrow} \RepDT{\alpha}(\AssembSpherePk{\Upsilon}). \]
can be understood as follows. Unfolding a hamantash assembly produces a collection of hyperbolic polygons embedded in the hyperbolic plane. These polygons are glued together as described after Definition~\ref{defn:hamantash-assembly}. An unfolded hamantash in $\Haman_{\alpha, \Upsilon}^{\textrm{north}}$ consists of a triangle, which is the southern hemisphere, and one to three flaps, which are the pieces of the northern hemisphere. For every pants curve in $\mathcal{B}$, there are two triangles meeting at a shared vertex, with the adjacent flaps glued together along the cut edges of a slit. Erasing the flaps gives a triangle chain, laid out along the spine of the hamantash assembly, with the geometric features described in Section~\ref{sec:dt-review}. Since all the triangles in the chain are southern hemispheres, they all lie on the same side of the spine introduced in Section~\ref{sec:marking}.
\begin{center}
\begin{tikzpicture}
\node[anchor=west] at (0.5, 0) {\reflectbox{\includegraphics[width=6cm]{fig/net-chain.pdf}}};
\node[anchor=east] at (-0.5, 0) {\reflectbox{\includegraphics[width=6cm]{fig/triangle-chain.pdf}}};

\draw (1.5,1.5) edge[bend right=30,->] node[midway, anchor=south] {erase the flaps} (-2,1.5);
\end{tikzpicture}
\end{center}
We'll show in Theorem~\ref{thm:holonomy-hamantash} that for each $b\in \mathcal{B}$, the action and angle coordinates $\beta_b$ and $\gamma_b$ of the triangle chain can be written in terms of hamantash assembly parameters as $\beta_b$ and $\phi'_b+\beta_b/2-\phi_b$, respectively. These hamantash assembly parameters normally can't be added to each other, because $\beta_b$ is real-valued, while $\phi_b$ and $\phi'_b$ take values in circles of different circumferences. On $\Haman_{\alpha, \Upsilon}^\textnormal{north}$, however, we can take $\phi_b$ and $\phi'_b$ to be valued in the real intervals $[0, \beta_b/2]$ and $[0, \beta'_b/2]$, respectively.
\begin{center}
\begin{tikzpicture}[scale=.9, every node/.style={inner sep=0.5mm}]
\node[anchor=south west, inner sep=0mm] at (-4.96, -3.07) {\reflectbox{\includegraphics[width=9cm]{fig/net-gluing.pdf}}};
\coordinate (S) at (-0.25, 0.75);
\coordinate (T) at (0.48, -0.67);
\draw[apricot] ++(68:0.3) arc (68:103:0.3) node[at start, right] {\small $\beta_b/2-\phi_b$};
\draw[apricot] ++(109:0.4) arc (109:158:0.4) node[at end, below left] {\small $\phi_b'$};
\draw[sky] ++(68:0.67) arc (68:155:0.67) node[near end, above left] {$\gamma_b$};
\draw[apricot] ++(-20:0.4) arc (-20:-50:0.4) node[near start, below right] {\small $\phi_b$};
\end{tikzpicture}
\end{center}
\begin{thm}\label{thm:holonomy-hamantash}
Let $\Upsilon$ be a chained pants assembly, and $\mathcal{B}$ its set of pants curves. Recall from Section~\ref{sec:marking} that $\pi_1\AssembSpherePk{\Upsilon}$ comes with a distinguished geometric presentation, whose standard pants decomposition is $\mathcal{B}$. Use the associated action-angle coordinates to parameterize $\IntRepDT{\alpha}{\mathcal B}(\AssembSpherePk{\Upsilon})$, as described in Corollary~\ref{cor:action-angle} and the paragraphs before it.

Under these conditions, the diagram below is well-defined and commutative. This implies, among other things, that the image of the holonomy map falls within $\IntRepDT{\alpha}{\mathcal B}(\AssembSpherePk{\Upsilon})$. Furthermore, the map $\Haman_{\alpha, \Upsilon}^\textnormal{north} \to \overset{\circ}{\Delta} \times [0, \pi]^\mathcal{B}$ is surjective.
\begin{center}
\begin{tikzcd}[row sep=1.5cm]
\Haman_{\alpha, \Upsilon}^{\textnormal{north}} \arrow[rr, bend left, "\textnormal{holonomy map}"] \arrow[r, "R"'] \arrow[d, "{(\beta,\;\phi'+\beta/2-\phi)}"'] & \HypCone{\alpha}(\AssembSphereExt{\Upsilon}) \arrow[r, "\hol"'] & \RepDT{\alpha}(\AssembSpherePk{\Upsilon}) \arrow[d, "{(\beta,\;\gamma)}"] \\
\overset{\circ}{\Delta} \times [0, \pi]^\mathcal{B} \arrow[rr] & & \Delta \times (\R/2\pi\Z)^\mathcal{B}
\end{tikzcd}
\end{center}
On the right-hand side of the diagram, $\beta$, $\gamma$ are the action-angle coordinates associated with $\mathcal{B}$. On the left-hand side, $\beta$, $\phi$, $\phi'$ are hamantash assembly parameters, which we treat as real-valued in the way we described before stating this theorem.
\end{thm}
\begin{proof}
Pick a hamantash assembly $H\in\Haman_{\alpha, \Upsilon}^\textnormal{north}$. As explained above, the realization of $H$ comes with an isotopy class of homeomorphisms to $\AssembSphereExt{\Upsilon}$, giving a hyperbolic cone structure $R(H) \in \HypCone{\alpha}(\AssembSphereExt{\Upsilon})$ on $\AssembSphereExt{\Upsilon}$. Let's fix one of those homeomorphisms, identifying the realization of $H$ with the hyperbolic cone sphere $(\AssembSphereExt{\Upsilon}, h_\HypConeStruct)$ given by some cone structure $\HypConeStruct \in R(H)$.

Let $c_1, \ldots, c_n$ be the elements of $\pi_1\AssembSphereExtPk{\Upsilon}$ that push forward to the distinguished geometric presentation of $\pi_1\AssembSpherePk{\Upsilon}$, described in Section~\ref{sec:marking}. Recall that each loop encloses a different puncture of $\AssembSpherePk{\Upsilon}$, identifying the set of punctures with $\{1, \ldots, n\}$. We'll use this labeling of the punctures throughout the proof.

Unfold $(\AssembSphereExt{\Upsilon}, h_\HypConeStruct)$ into a bunch of polygons and glue the polygons together as described above. We get a larger polygon $P$ which is immersed in $\HH$ by an orientation-reversing local isometry, potentially overlapping itself. The marked points of $(\AssembSphereExt{\Upsilon}, h_\HypConeStruct)$ unfold to vertices of $P$, which we'll call $C_1,\ldots,C_n$. The angle of $P$ at $C_i$ is equal to $2\pi-\alpha_i$. If we erase the flaps of $P$, as described above, we get a triangle chain $T$. On other other hand, if we remove the vertices of $P$ and develop the resulting polygon $P'$ across its sides, we get a universal covering $\widetilde{P} \to \AssembSphereExtPk{\Upsilon}$ which is a local isometry with respect to $\HypConeStruct$. Continuing the immersion $P' \to \HH$ along this covering gives an orientation-reversing developing map $\widetilde{P} \to \HH$ for $\HypConeStruct$, as defined in Section~\ref{sec:space-of-hyperbolic-cone-metrics}. Post-compose with a flip to get an ordinary developing map, and let $\rho_P \maps \pi_1\AssembSphereExtPk{\Upsilon} \to \psl$ be the corresponding holonomy representation.

By construction, $\widetilde{P}$ is tiled with copies of $P'$, which are indexed by the elements of $\pi_1\AssembSphereExtPk{\Upsilon}$. We can define $P'_\zeta$ as the tile you end up in if you follow a path on $\widetilde{P}$ that projects to $\zeta \in \pi_1\AssembSphereExtPk{\Upsilon}$. The action of $\zeta$ on $\widetilde{P}$ brings $P'_\zeta$ back to $P'$. Correspondingly, the action of $\rho_P(\zeta)$ on $\HH$ brings the image of $P'_\zeta$ back to the image of $P'$ under the developing map.

Recall that each loop $c_i$ runs mostly through the southern hemisphere of a samosa, briefly entering the northern hemisphere when it loops counterclockwise around the un-slit corner $i$. As a result, when we unfold $(\AssembSphereExt{\Upsilon}, h_\HypConeStruct)$, the loop $c_i$ is only cut in one place: across the sides of $P$ adjacent to $C_i$. That means the tile $P'_{c_i} \subset \widetilde{P}$ is adjacent to $P'$. Since $P \to \HH$ is orientation-reversing, we cross from the image of $P'$ into the image of $P'_{c_i}$ by going clockwise around $C_i$. Thus, to bring the image of $P'_{c_i}$ back to the image of $P'$, we rotate it by $2\pi - \alpha_i$ counterclockwise around $C_i$. Composing $P \to \HH$ with a flip to get an orientation-preserving chart, we learn that $\rho_P(c_i)$ is the clockwise rotation of angle $2\pi - \alpha_i$ around $C_i$.

Recall from Section~\ref{sec:space-of-hyperbolic-cone-metrics} that $\rho_P$ factors through the projection $\pi_1\AssembSphereExtPk{\Upsilon} \to \pi_1\AssembSpherePk{\Upsilon}$ to become a representation of $\pi_1\AssembSpherePk{\Upsilon}$, which we also call $\rho_P$. The holonomy map $\hol \circ R$ sends $H$ to the conjugacy class of $\rho_P \maps \pi_1\AssembSpherePk{\Upsilon} \to \psl$. With that in mind, let's go back to the triangle chain $T$ that we got by erasing the flaps of $P$. This triangle chain describes a representation $\rho_T \maps \pi_1\AssembSpherePk{\Upsilon} \to \psl$, as discussed in Section~\ref{sec:dt-review}. By definition, $\rho_T(c_i)$ is the counterclockwise rotation of angle $\alpha_i$ around around $C_i$. In light of the previous paragraph, that means $\rho_T = \rho_P$. The holonomy map therefore sends $H$ to the conjugacy class of $\rho_T$. You can now compare the angles in $T$ that give the action-angle coordinates of $\rho_T$ with the angles in $P$ that give the hamantash parameters $\beta, \phi, \phi'$. You'll see that the action coordinates are $\beta$ and the angle coordinates are $\phi'+\beta/2-\phi$, as we wanted to show.

The range of the admissible values of $\beta$ for hamantash assemblies is the same as the one for samosa assemblies, which is given by the inequalities after Definition~\ref{defn:samosa-assembly}. The inequalities define the interior of a polytope that corresponds to the polytope of Corollary~\ref{cor:action-angle} showing that the image of the holonomy map is contained in $\IntRepDT{\alpha}{\mathcal B}(\AssembSpherePk{\Upsilon})$. Since every slit lies in the northern hemisphere, we can take $\phi_b'$ and $\phi_b$ to be valued in $[0, \beta'_b/2]$ and $[0, \beta_b/2]$, respectively. The angle $\gamma_b = \phi_b'+\beta_b/2-\phi_b$ therefore lies in $[0, \pi]$.\end{proof}

\begin{question}\label{q:get-coordinates-for-Teich-?}
A consequence of Theorem~\ref{thm:holonomy-hamantash} is that the slit lengths $\ell$ have no influence of the holonomy of the cone structure. Similarly, changing the slit angles $\phi$ and $\phi'$ while preserving their difference leaves the holonomy unchanged. However, both of these changes affect the conformal structure on $\AssembSphere{\Upsilon}$. It would be interesting to investigate whether the slit lengths and some combination of the slit angles could meaningfully parameterize all or part of $\Teich(\AssembSphere{\Upsilon})$ through the mapping
\[
\Haman_{\alpha, \Upsilon}^\textnormal{gen} \longrightarrow \HypCone{\alpha}(\AssembSphereExt{\Upsilon}) \overset{\tau}{\longrightarrow} \Teich (\AssembSphere{\Upsilon})
\]
(Section~\ref{sec:space-of-hyperbolic-cone-structures}). If they could, how would this parameterization relate to the traditional Fenchel-Nielsen coordinates on Teichmüller space?
\end{question}
\begin{rem}
Each action coordinate $\beta_b$ generates a Hamiltonian ``twist flow'' on the space of DT representations, which increases the angle coordinate $\gamma_b$ at unit speed. We can locally lift this flow to the space of samosa assemblies by decreasing $\phi_b$ at unit speed, increasing $\phi'_b$ at unit speed, or doing some convex combination of those. Since Theorem~\ref{thm:holonomy-hamantash} only applies to $H\in\Haman_{\alpha, \Upsilon}^\textnormal{north}$, it would take some thought to continue the flow when a slit leaves the northern hemisphere. To keep the flow continuous, we expect to have to adjust the presentation of $\pi_1\AssembSpherePk{\Upsilon}$ by a Dehn twist each time a slit goes through an equator segment opposite an un-slit corner, based on the discussion in Section~\ref{sec:marking}. This shouldn't conflict with the fact that $\phi_b$ and $\gamma_b$ have periods $\beta_b$ and $2\pi$, respectively, because a Dehn twist along $b$ only changes $\gamma_b$ by $\beta_b$.
\end{rem}
\subsection{Finding a suitable pants decomposition}\label{sec:no-overlap-no-degenerate-triangles}
We want to use the description of the image of the holonomy map in Theorem~\ref{thm:holonomy-hamantash} to realize every DT representation $\rho$ as the holonomy of a hyperbolic cone metric built from a hamantash assembly. In order to do that, we must first find a geometric presentation of $\pi_1\SpherePk$ and a compatible pants decomposition of $\SpherePk$ for which the action-angle coordinates of $\rho$, whose construction was described in Section~\ref{sec:dt-review}, are in the image of the holonomy map of Theorem~\ref{thm:holonomy-hamantash}. The next three propositions guarantee the existence of the desired presentation of $\pi_1\SpherePk$.
\begin{prop}\label{prop:non-degenerate-pants-decomp}
For any DT representation, there exists a chained pants decomposition of $\SpherePk$ for which every triangle in the associated triangle chain has positive area.
\end{prop}
\begin{proof}
 We'll make use of the following two properties of triangle chains from \cite{action-angle}.
\begin{enumerate}
  \item\label{got-non-degen} For any pair of pants decomposition and any DT representation, there is always at least one non-degenerate triangle in the associated chain. This follows from the fact that the total area of the triangles in the chain is a positive constant (essentially equal to the Toledo number of the representation).
  \item\label{all-or-none} A triangle in a chain is degenerate if and only if all three vertices coincide. This is a consequence of the fact that if three elliptic elements of $\psl$ can be multiplied to get the identity, their fixed points either coincide or form a non-degenerate triangle. 
\end{enumerate}

Assume that $n\geq 4$, since the claim is obvious for $n=3$. For concreteness, identify $\SphereMk$ with the one-point compactification of the plane, with the marked points laid out around a circle. Label the marked points $1,\ldots,n$, going counterclockwise around the circle. The marked points then form a cyclically ordered set. Introduce a standard geometric presentation of $\pi_1 \SpherePk$ by putting the base point at the center of the circle and representing the generator $c_i$ by a path that runs outward along a radius of the circle, loops counterclockwise around $i$, and then returns along a radius to the center. More generally, for any set $I$ of consecutive punctures, let $c_I$ be the element of $\pi_1 \SpherePk$ represented by a path that runs outward along a radius, loops counterclockwise around $I$, and then returns along a radius to the center.
\begin{center}
\begin{tikzpicture}[decoration={markings, mark=at position 0.5 with {\arrow[scale=1.5]{Latex[]}}}]
\newcommand{\npunk}{5}
\newcommand{\punkgap}{360/\npunk}

\foreach \k in {1, ..., \npunk} {
  \pgfmathsetmacro{\theta}{360*\k/\npunk:1}
  \fill (\k*\punkgap:1.5) circle (0.08) ++(\k*\punkgap:0.3) node {$\k$};
  \draw[mauve, -{Latex[length=2.5mm, flex]}] (0, 0) .. controls +({(\k-0.45)*\punkgap}:1.0) and +({\k*\punkgap-90}:1.2) .. (\k*\punkgap:2.4) node[outer sep=0.8mm, anchor=\k*\punkgap+180] {$c_\k$};
  \draw[mauve] (\k*\punkgap:2.4) .. controls +({\k*\punkgap+90}:1.2) and +({(\k+0.45)*\punkgap}:1.0) .. (0, 0);
}
\fill[mauve] circle (0.08);

\begin{scope}[shift={(6, 0)}]
\draw[mauve, -{Latex[length=2.5mm, flex]}] (0, 0) .. controls +(3.55*\punkgap:1.0) and +({4*\punkgap-90}:1.2) .. (4*\punkgap:2.5) .. controls +({4*\punkgap+90}:1.0) and +({5*\punkgap-90}:1.0) .. (5*\punkgap:2.5) node[outer sep=0.8mm, anchor=5*\punkgap+180] {$c_{\{4,5,1\}}$};
\draw[mauve] (5*\punkgap:2.5) .. controls +({5*\punkgap+90}:1.0) and +({1*\punkgap-90}:1.0) .. (1*\punkgap:2.5) .. controls +({1*\punkgap+90}:1.2) and +(1.45*\punkgap:1.0) .. (0, 0);
\fill[mauve] circle (0.08);
\foreach \k in {1, ..., \npunk} {
  \pgfmathsetmacro{\theta}{360*\k/\npunk:1}
  \fill (\k*\punkgap:1.5) circle (0.08) ++(\k*\punkgap:0.3) node {$\k$};
}
\end{scope}
\end{tikzpicture}
\end{center}
When $I$ is the whole set of punctures, there are many paths that fit this description, but they all represent the identity in $\pi_1 \SpherePk$. If $I$ and $J$ are adjacent, disjoint sets of consecutive punctures, with $I$ clockwise of $J$, then $c_I c_J = c_{I \cup J}$.\footnote{In fundamental group product expressions, we concatenate the loops from left to right, so $c_I c_J$ means following $c_I$ and then $c_J$.}

The problem of finding the desired pants decomposition for $\rho\colon \pi_1\SpherePk\to \operatorname{PSL}_2\R$ can be phrased as a combinatorial game. Before we explain the rules, we need to paint $\HH$ so that every point is a different color. We then give each puncture $i$ the color of the fixed point of $\rho(c_i)$. More generally, we give each set of consecutive punctures $I$ the color of the fixed point of $\rho(c_I)$. We know that $\rho(c_I)$ is elliptic because DT representations are totally elliptic, and $c_I$ is represented by a simple curve. Observation~\eqref{all-or-none} gives this coloring a useful combinatorial property.
\begin{lem}\label{all-or-none-color}
Suppose $I$ and $J$ are adjacent, disjoint sets of consecutive punctures. Then $I$, $J$, and $I \cup J$ have either all different colors or all the same color.
\end{lem}
\begin{proof}
By switching the labels of the sets if necessary, we can assume that $I$ is clockwise of $J$. We observed earlier that $c_I c_J = c_{I \cup J}$, implying that $\rho(c_I), \rho(c_J)$, and $\rho(c_{I\cup J})^{-1}$ are three elliptic elements of $\psl$ that can be multiplied to get the identity. Observation~\eqref{all-or-none} then gives the desired conclusion.
\end{proof}

Now we can explain the rules of the game. At the beginning of the game, we have the punctures laid out in a circle on the table, matching their arrangement on $\SpherePk$. On the first move, we remove two neighbouring punctures from the table; we're disqualified from winning if we remove two punctures of the same color. Now the circle of punctures has a gap in it. On each subsequent move, we widen the gap by removing one of the two punctures next to it; we're disqualified if we take a puncture of the same color as the set of missing punctures. The game ends when there are two punctures left on the table. If we've been disqualified, or if we're left with two punctures of the same color as the set of missing punctures, we lose. Otherwise, we win. By Lemma~\ref{all-or-none-color}, winning implies that the two remaining punctures and the set of missing punctures all have different colors.

Playing this game is equivalent to building a chained pants decomposition of $\SpherePk$. Each move adds a pants curve, which is determined by the set of missing punctures at the end of the move: calling the set of missing punctures $I$, we add a curve that represents the fundamental group element $c_I$. We lose if we ever make a pair of pants that $\rho$ associates with a degenerate triangle.

Now, let's show that we can always win. By Observation~\eqref{got-non-degen}, we can always make a first move without being disqualified. Now, imagine that we haven't been disqualified yet, but our previous move has left us in a hopeless position: either all of our available moves will disqualify us, or we've ended the game in a losing position. That means the punctures next to the gap both have the same color as the set of missing punctures. If our previous move was not the first move, let $z$ be the puncture we removed on that move, and let $I$ be the set of missing punctures at the beginning of that move. If the previous move was the first move, let $z$ be one of the two punctures we removed, and let $I$ be the set containing the other puncture we removed. In either case, the current set of missing punctures is $I \cup \{z\}$.

Since our previous move didn't disqualify us, $I$ and $z$ have different colors. Let $y$ and $x$ be the punctures next to the gap, with $y$ adjacent to $I$ and $x$ adjacent to $z$. Our hopeless position implies that $y$ has the same color as $I \cup \{z\}$, and Lemma~\ref{all-or-none-color} tells us that $I \cup \{z\}$ has a different color from both $I$ and $z$, so we know that $y$ has a different color from $I$. Thus, on our previous move, we could've removed $y$ instead of $z$ without being disqualified.

Let's go back to the previous move and remove $y$ instead of $z$. Recalling that $I$ and $z$ have different colors, while $I \cup \{z\}$ and $y$ have the same color, we can deduce through the following lemma that $I \cup \{y\}$ and $z$ have different colors.
\begin{lem}
Suppose $I$, $J$, $J'$ are disjoint sets of punctures, with $J$ and $J'$ both adjacent to $I$. Suppose $I$ and $J$ have different colors, while $I \cup J$ and $J'$ have the same color. Then $I \cup J'$ and $J$ have different colors.
\end{lem}
\begin{proof}
If $I \cup J'$ and $J$ had the same color, Lemma~\ref{all-or-none-color} would imply that $I \cup J \cup J'$ had the same color as $I \cup J'$ and $J$. Hence, it's enough to show that $I \cup J \cup J'$ has a different color from $J$.

Lemma~\ref{all-or-none-color} implies that $I \cup J$ has a different color from both $I$ and $J$. It also implies that $I \cup J \cup J'$ has the same color as $I \cup J$ and $J'$. Hence, $I \cup J \cup J'$ has a different color from both $I$ and $J$.
\end{proof}
We now see that by removing $y$ instead of $z$, we've put ourselves in a better position---one where either we can move without being disqualified or we've finished the game in a winning position. If the game hasn't ended yet, the difference in color between $I \cup \{y\}$ and $z$ implies that we can remove $z$ without being disqualified. If the game has ended, the same difference in color implies that we're in a winning position. This shows that the hopeless position we first imagined can always be avoided by changing the previous move.
\end{proof}
The chained pants assembly produced by Propostion~\ref{prop:non-degenerate-pants-decomp} has a defect. When producing the triangle chain, there is no guarantee that all the triangles will be clockwise oriented. This is because we didn't adapt the system of generators of $\pi_1\SpherePk$ to make the pants decomposition standard. We correct this defect with the next proposition. The idea is to move along the triangle chain and flip every ill-oriented triangle by conjugating the corresponding generator, while preserving the same pants decomposition. At the end of the process, the pants decomposition becomes standard for the new system of generators ensuring that every triangle in the chain is now clockwise oriented, as explained in Section~\ref{sec:dt-review}.
\begin{prop}\label{prop:non-degenerate-standard-pants-decomp}
For every DT representation, there exists a geometric presentation of $\pi_1\SpherePk$ such that every triangle in the triangle chain associated with its standard pants decomposition has positive area.  
\end{prop}
\begin{proof}
We start by applying Proposition~\ref{prop:non-degenerate-pants-decomp} to find a chained pants decomposition $\mathcal B$ for which all triangles in the chain have positive area. In the proof of Propsotion~\ref{prop:non-degenerate-pants-decomp} we used a particular system of generators $c_1,\ldots,c_n$. We now switch to a new generating set $c_1',\ldots,c_n'$ for which the pants decomposition $\mathcal B$ is the standard one. It's constructed inductively, like this:
\begin{itemize}
    \item In the first move of the game of Proposition~\ref{prop:non-degenerate-pants-decomp}, we removed two neighbouring punctures, say $i$ and $i+1$. Let's define $c_1'=c_i$ and $c_2'=c_{i+1}$. The first pants curve is represented by the fundamental group element $c_{\{i,i+1\}}=c_ic_{i+1}$ and also by its inverse 
    \[
    b_1=(c_{\{i,i+1\}})^{-1}=(c_2')^{-1}(c_1')^{-1}.
    \]
    \item In the second move of the game, we removed either puncture $i-1$ or $i+2$. If we removed puncture $i+2$, then we simply define $c_3'=c_{i+2}$. The second pants curve is then represented by the fundamental group element $c_{\{i,i+1,i+2\}}=c_ic_{i+1}c_{i+2}$. In this case we also define $b_2=(c_{\{i,i+1,i+2\}})^{-1}$. If instead we removed puncture $i-1$, then we alternatively define $c_3'=b_1c_{i-1}b_1^{-1}$. In this case, the second pants curve is represented by $c_{\{i-1,i,i+1\}}=c_{i-1}c_ic_{i+1}$
    and we define $b_2=(c_{\{i-1,i,i+1\}})^{-1}$. In both cases,
    \[
    b_2=(c_3')^{-1}b_1=(c_3')^{-1}(c_2')^{-1}(c_1')^{-1}.
    \]
    \item Now, suppose we've just defined $c_j'$ with $3\leq j\leq n-3$ by modifying the fundamental group generator that we removed on the $(j-1)$th move. Let $I$ be the set of missing punctures at the beginning of the $j$th move and $z$ be the puncture removed on the $j$th move. Part of our induction hypotheses is that the last pants curve that we built is represented by the fundamental group element $b_{j-1}=c_I^{-1}=(c_j')^{-1}\cdots (c_2')^{-1}(c_1')^{-1}$. If $z$ was directly counterclockwise of $I$, then we let $c_{j+1}'=c_z$. If $z$ was clockwise of $I$, then we let $c_{j+1}'=b_{j-1}c_zb_{j-1}^{-1}$. In both cases, the new pants curve is represented by 
    \[
    b_j=(c_{I\cup \{z\}})^{-1}=(c_{j+1}')^{-1}(c_j')^{-1}\cdots (c_2')^{-1}(c_1')^{-1}.
    \]
    \item Once we've defined $c_{n-2}'$, then we finish the construction as follows. The last pants curve we built is represented by the fundamental group element 
    \[
    b_{n-3}=(c_I)^{-1}=(c_{n-2}')^{-1}\cdots (c_2')^{-1}(c_1')^{-1},
    \]
    where $I$ is now a set of all but two neighbouring punctures which we denote $x$ and $y$. Assume that $x$ is directly counterclockwise of $I$. We define $c_{n-1}'=c_x$ and $c_n'=c_y$. Since $x$ is directly counterclockwise of $I$ and $c_1\cdots c_n=1$, it holds that $(c_I)^{-1}=c_xc_y=c_{n-1}'c_n'$. Comparing with the expression of $b_{n-3}$ above, we conclude that $c_1'\cdots c_n'=1$.
\end{itemize}
We have produced a collection $c_1',\ldots,c_n'$ of fundamental group element satisfying $c_1'\cdots c_n'=1$. By construction, each of the $c_j'$ is conjugate to one of the generators $c_1,\ldots, c_n$. Moreover, the pants curves of $\mathcal B$ are represented by $b_j=(c_{j+1}')^{-1}\cdots (c_2')^{-1}(c_1')^{-1}$, making $\mathcal B$ the standard pants decomposition of $\SpherePk$ for the geometric generators $c_1',\ldots,c_n'$ of $\pi_1\SpherePk$.
\end{proof}
We can even go one step further than Proposition~\ref{prop:non-degenerate-standard-pants-decomp} to make sure that adjacent triangles in a chain don't overlap. Two triangles are said to \emph{overlap} if their interiors aren't disjoint. Note that non-overlapping triangles are allowed to intersect along their boundaries. In terms of the action-angle coordinates of~\cite{action-angle}, two non-degenerate adjacent triangles with common vertex corresponding to the pants curve $b$ overlap if and only if $\pi<\gamma_b<2\pi$ when $\gamma_b$ is seen as a number in $[0,2\pi)$. When $\gamma_b\in \{0,\pi\}$, the triangles touch along their boundaries.
\begin{prop}\label{prop:no-overlap}
For every DT representation $\rho\colon \pi_1\SpherePk\to \psl$, there exists a geometric presentation of $\pi_1\SpherePk$ such that every triangle in the triangle chain associated with its standard pants decomposition has positive area and no pair of adjacent triangles overlap.\footnote{In fact, we'll prove something slightly stronger, as discussed in Remark~\ref{rem:no-overlap-more-precise}.}
\end{prop}
\begin{proof}
We start from the geometric presentation $c_1,\ldots,c_n$ provided by Proposition~\ref{prop:non-degenerate-standard-pants-decomp}. It guarantees that all the triangles in the chain associated to $\rho$ and built using the standard pants decomposition $b_1,\ldots, b_{n-3}$ are non-degenerate. Suppose that the two triangles that share the vertex corresponding to the fixed point of $\rho(b_i)$ overlap. In terms of action-angle coordinates, this means that $\gamma_i \in (\pi, 2\pi)$.

Consider the fundamental group automorphism $\tau_i\colon \pi_1\SpherePk\to \pi_1\SpherePk$ defined by
\[ \tau_i(c_j) = \begin{cases}
c_j & j \le i+1 \\
b_i c_j b_i^{-1} & j \ge i+2.
\end{cases} \]
We can see that $\tau_i(c_1)\cdots\tau_i(c_n)=1$ by unpacking the definition of $\tau_i$ and recalling that $b_i=c_{i+1}^{-1} c_i^{-1} \cdots c_1^{-1}$. (Under the classical Dehn--Nielsen--Baer correspondence, explained for instance in~\cite[§8]{mcg-primer}, the automoprhism $\tau_i$ comes from the Dehn twist along the simple closed curve $b_i$.) Changing the generating family of $\pi_1\SpherePk$ from $\{c_1,\ldots,c_n\}$ to $\tau_i(\{c_1,\ldots,c_n\})$ only affects one action-angle coordinate of $\rho$: the angle coordinate $\gamma_i$. The change in $\gamma_i$ is given by the formula
\begin{align*}
\gamma_i^\text{new}(\rho) & = \gamma_i^\text{old}(\rho\circ \tau_i) \\
& = \gamma_i^\text{old}(\rho) - \beta_i.
\end{align*}
Similarly, changing the generating family of $\pi_1\SpherePk$ from $\{c_1,\ldots,c_n\}$ to $\tau_i^{-1}(\{c_1,\ldots,c_n\})$ gives
\begin{align*}
\gamma_i^\text{new}(\rho) & = \gamma_i^\text{old}(\rho\circ \tau_i^{-1}) \\
& = \gamma_i^\text{old}(\rho) - \beta'_i,
\end{align*}
leaving all of the other action-angle coordinates unchanged.

Suppose $\beta_i \in (0, \pi]$. Since our assumptions imply that $\gamma_i \in [\pi, 2\pi)$, there's an integer $n \geq 1$ with $\gamma_i - n\beta_i \in [0, \pi)$. We change the generating family to $\tau_i^n(\{c_1,\ldots,c_n\})$, giving $\gamma_i^\text{new} = \gamma_i^\text{old} - n\beta_i$, so we end up with $\gamma_i^\text{new} \in [0, \pi)$, as desired. On the other hand, suppose $\beta_i \in (\pi, 2\pi)$. That means $\beta'_i \in (0, \pi]$, and there's an integer $n' \geq 1$ with $\gamma_i - n'\beta'_i \in [0, \pi)$. We change the generating family to $\tau_i^{-n}(\{c_1,\ldots,c_n\})$, so we end up with $\gamma_i^\text{new} \in [0, \pi)$ in the same way.

In the triangle chain built from the new generating family, the triangles meeting at the fixed point $\rho(b_i)$ no longer overlap. Since the change of generating family only affects $\gamma_i$, leaving the other action-angle coordinates of $\rho$ unchanged, it does not create any new overlaps between other pairs of adjacent triangles. Therefore, by repeating the process above for each pair of adjacent triangles, we can get rid of all the overlaps.
\end{proof}
\begin{rem}\label{rem:no-overlap-more-precise}
Our proof of Proposition~\ref{prop:no-overlap} actually shows something a bit stronger: it produces a geometric presentation of $\pi_1\SpherePk$ with $\gamma_i\in[0,\pi)$ for every $i=1,\ldots,n-3$. This excludes not only overlaps, but also some almost-overlaps: the ones where $\gamma_i = \pi$ for some $i$. This stronger result will be useful in the proof of Corollary~\ref{cor:surjectivity-holonomies}.
\end{rem}
\subsection{Building a hyperbolic cone sphere with prescribed holonomy}
We know from Theorem~\ref{thm:holonomy-hamantash} that for each pants assembly $\Upsilon$, with underlying set of pants curves $\mathcal B$, the hamantash assemblies $\Haman_{\alpha, \Upsilon}^{\textnormal{north}}$ realize a large, neatly shaped piece of the dense open subset $\IntRepDT{\alpha}{\mathcal B}(\AssembSpherePk{\Upsilon}) \subset \RepDT{\alpha}(\AssembSpherePk{\Upsilon})$. We'll now show that if we iterate over all pants assemblies $\Upsilon$, the images of the holonomy maps $\Haman_{\alpha, \Upsilon}^{\textnormal{north}} \longrightarrow \RepDT{\alpha}(\AssembSpherePk{\Upsilon})$ cover all of $\RepDT{\alpha}(\SpherePk)$. 
\begin{cor}\label{cor:surjectivity-holonomies}
Every DT representation $\rho\colon \pi_1\SpherePk\to \operatorname{PSL}_2\R$ is the holonomy of a hyperbolic cone metric constructed from a hamantash assembly.
\end{cor}
\begin{proof}
We apply Proposition~\ref{prop:no-overlap} to find a geometric presentation of $\pi_1\SpherePk$ such that all the triangles in the triangle chain associated with $\rho$ for the standard pants decomposition $\mathcal{B}$ have positive area and don't overlap with adjacent triangles. The pants curves of $\mathcal B$ are represented by the fundamental group elements $b_1,\ldots,b_{n-3}$. Let $\beta_1,\ldots,\beta_{n-3}$ and $\gamma_1,\ldots,\gamma_{n-3}$ be the action-angle coordinates of $\rho$ for the pants decomposition $\mathcal B$. In the triangle chain associated with $\rho$, no consecutive triangles overlap, so $\gamma_i\in [0,\pi]$ for every $i=1,\ldots, n-3$. We can even adapt the geometric presentation to make sure that $\gamma_i\in [0,\pi)$ for every $i=1,\ldots, n-3$, as discussed in Remark~\ref{rem:no-overlap-more-precise}.

From a combinatorial point of view, the pants decomposition $\mathcal{B}$ gives us a bunch of pairs of pants with a relation that says certain pairs of cuffs should be glued together. Each pair of pants comes with an orientation, and the pants are glued with matching orientations. Adding a little extra structure will turn this data into a genus-0 pants assembly, as defined in Section~\ref{sec:assembly-instructions-parametrization}. First, for each pair of pants, we need to give a cyclic ordering of the cuffs. We do this by recalling that each pair of pants corresponds to a triangle in the triangle chain, and its three cuffs correspond to the three vertices of the triangle. We use the clockwise ordering of the vertices as the cyclic ordering of the cuffs. Next, we need to orient each pants curve. We use the orientations of the loops $b_1, \ldots, b_{n-3}$.

We've now defined a genus-0 pants assembly $\Upsilon$ that encodes the pants decomposition $\mathcal{B}$ of $\SpherePk$. Our next goal is to identify $\SpherePk$ with the generic realization sphere $\AssembSpherePk{\Upsilon}$. Recall from Section~\ref{sec:marking} that $\pi_1\AssembSpherePk{\Upsilon}$ comes with a distinguished geometric presentation, whose standard pants decomposition is $\mathcal{B}$. There's a unique isotopy class of homeomorphisms $\AssembSpherePk{\Upsilon} \to \SpherePk$ that sends this geometric presentation to the chosen geometric presentation of $\pi_1\SpherePk$. This is the desired identification.

Now, let's build a hamantash assembly in $\Haman_{\alpha,\Upsilon}^{\textnormal{north}}$ from the triangle chain associated with~$\rho$. As discussed in Section~\ref{sec:holonomy-hamantash}, we think of each triangle in the triangle chain as the southern hemisphere of a samosa, which is sitting on the plane with its interior side facing upward. To complete the samosa, we need to add a northern hemisphere: a copy of the same triangle with the opposite orientation. We place the northern hemisphere triangle on top of the southern hemisphere triangle and glue the edges. This turns the triangle chain into a sequence of samosas.

Next, from each shared vertex $B_i$ of the triangle chain, we draw slits on the northern hemispheres of the adjacent samosas. We choose the slit angles $\phi_i \in (0, \beta_i/2]$ and $\phi'_i \in  [0, \beta'_i/2)$ so that $\gamma_i = \phi'_i + \beta_i/2 - \phi_i$. We can always choose the slit angles this way, because the range of the expression $\phi'_i + \beta_i/2 - \phi_i$ with $\phi_i \in (0, \beta_i/2]$ and $\phi'_i \in [0, \beta'_i/2)$ is the interval $[0,\pi)$, which matches the range of $\gamma_i$. If $\gamma_i$ is $0$, there's only one choice: $\phi_i = \beta_i/2$ and $\phi_i'=0$. On the other hand, if $\gamma_i\in (0,\pi)$, there's a continuum of possible choices for $\phi_i$ and $\phi'_i$.  We give the slits on both sides of $B_i$ the same length $\ell_i$, so they can be cut and glued together. This turns our sequence of samosas into a samosa assembly whose underlying pants assembly is $\Upsilon$. Finally, because we made sure that $\phi_i$ is never equal to $0$ and $\phi_i'$ is never equal to $\beta_i'/2$, we can pick $\ell_1, \ldots, \ell_{n-3}$ small enough to ensure that our samosa assembly is a hamantash assembly in $\Haman_{\alpha,\Upsilon}^{\textnormal{north}}$ (Lemmas~\ref{lem:bound-on-l-one-slit} and~\ref{lem:bound-on-l-two-slits}). Label the resulting hamantash assembly~$H_\rho$.

We can now apply Theorem~\ref{thm:holonomy-hamantash}, recalling that we've identified $\SpherePk$ with $\AssembSpherePk{\Upsilon}$, and conclude that the image of $H_\rho$ under the holonomy map has the same action-angle coordinates as $\rho$. In other words, the holonomy map sends $H_\rho$ to the conjugacy class of $\rho$.
\end{proof}

\section{Allowing degenerate configurations}
To parameterize hyperbolic cone structures in the larger space $\AugHypCone{\alpha}(\SphereExtMk)$, where whole singularities are allowed to coalesce with other singularities (Section~\ref{sec:space-of-hyperbolic-cone-structures}), one might try to extend the realization process to degenerate samosa assemblies, whose slits are allowed to end at the ends of other slits and at fractional singularities (Definition~\ref{defn:samosa-assembly}). Intuitively, most degenerate samosa assemblies should have realizations in $\AugHypCone{\alpha}(\AssembSphereExt{\Upsilon})$, but some degenerate samosa assemblies might need to be excluded from the domain $\AugAssemb_{\alpha, \Upsilon}$ of the extended realization map. For example, since fractional singularities aren't allowed to merge in $\AugHypCone{\alpha}(\AssembSphereExt{\Upsilon})$, it might not be possible to realize samosa assemblies where two slits that end on fractional singularities are meant to be glued together.
\begin{question}\label{q:degenerate-configurations-?}
Does the realization map from Section~\ref{sec:marking} extend meaningfully to a realization map $\overline{R}\colon \AugAssemb_{\alpha, \Upsilon} \to \AugHypCone{\alpha}(\AssembSphereExt{\Upsilon})$, where the domain is expanded to include some degenerate samosa assemblies? If so, can $\overline{R}$ be used to chart $\AugHypCone{\alpha}(\AssembSphereExt{\Upsilon})$?
\end{question}
\begin{rem}
Even if an extended realization map could produce new hyperbolic cone structures, it wouldn't produce any new holonomy representations, because every DT representation is already realized by a non-degenerate samosa assembly (Corollary~\ref{cor:surjectivity-holonomies}).
\end{rem}
Any degenerate samosa assembly can be seen as a limit of non-degenerate ones by shortening the slits that end at fractional singularities and at the ends of other slits. One could thus start studying Question~\ref{q:degenerate-configurations-?} by verifying that the realized sequence of hyperbolic cone structures has a limit in $\AugHypCone{\alpha}(\AssembSphereExt{\Upsilon})$. This limit might then be interpreted as the realization of the initial degenerate samosa assembly. Note that all the realized hyperbolic cone spheres in the sequence would have the same holonomy, since the samosa assemblies in the sequence differ only in their slit lengths. This suggests a relationship between Questions~\ref{q:get-coordinates-for-Teich-?} and~\ref{q:degenerate-configurations-?}.
\appendix
\section{Trigonometric formulas}\label{apx:trig-formulae}
The following classical formulas will be useful for our study of hyperbolic cone metrics. We assume that we're working with a non-degenerate hyperbolic triangle of side lengths $(a,b,c)$ and interior angles $(\alpha, \beta, \gamma)$. 
\subsubsection*{Hyperbolic laws of cosines}
\begin{align}
    \cos(\gamma)&=\frac{\cosh(a)\cosh(b)-\cosh(c)}{\sinh(a)\sinh(b)} \tag{C.1}\label{eq:hyperbolic-law-of-cosines-lengths},\\
    \cosh(c)&= \frac{\cos(\alpha)\cos(\beta)+\cos(\gamma)}{\sin(\alpha)\sin(\beta)}. \tag{C.2}\label{eq:hyperbolic-law-of-cosines-angles}
\end{align}
\subsubsection*{Hyperbolic law of sines}
\begin{equation}
    \sin(\alpha)=\sinh(a)\frac{\sin(\beta)}{\sinh(b)} \tag{S}\label{eq:hyperbolic-law-of-sines}.
\end{equation}
\subsubsection*{Four-parts formula}
\begin{equation}
    \cos (\gamma)\cosh(a)=\sinh(a)\coth(b)-\sin(\gamma)\cot(\beta). \tag{F}\label{eq:four-part-formula}
\end{equation}
\section{Barycentric coordinates in the hyperbolic plane}\label{apx:barycentric-coordinates}
\subsection{Convex combinations of points in the hyperbolic plane}
The barycentric coordinates on hyperbolic triangles used in the proof of Proposition~\ref{length-homeo} are defined in terms of ``convex combinations'' of points in the hyperbolic plane, analogous to the usual convex combinations in a vector space. This approach is taken from the work of Stahl~\cite{stahl}.
\begin{defn}
A \emph{point-mass} is a pair $(X,x)$ consisting of a point $X$ in the hyperbolic plane and positive real number $x$. A \emph{point-mass system} is a finite collection of point-masses.
\end{defn} 
A point-mass system has a \emph{center of mass}, which is the analogue of a weighted sum in a vector space. The center of mass of $(X,x)$ and $(Y,y)$ is defined to be the point-mass $(Z, z)$ given by the unique point $Z$ on the geodesic segment $XY$ satisfying
\[
x\sinh\big(d(X,Z)\big)=y\sinh\big(d(Y,Z)\big)
\]
and the real number $z = x\cosh\big(d(X,Z)\big)+y\cosh\big(d(Y,Z)\big)$. The center of mass can be seen as the result of a binary operation: $(Z,z) = (X,x)\ast (Y,y)$. In the special case where the points are the same and only the masses differ, the $\ast$ operation simplifies to $(X,x) \ast (X,x') = (X,x+x')$.

Stahl proved that the $\ast$ operation is commutative and associative~\cite[Proposition~3.3]{stahl}, leading to a well-defined notion of \emph{center of mass} for an arbitrarily large point-mass system $\{(X_1,x_1), \ldots, (X_n,x_n)\}$, given by the point-mass
\[
(C,c)=(X_1,x_1)\ast\cdots\ast (X_n,x_n).
\]
The mass $c$ can be expressed as $c=\sum_i x_i\cosh\big(d(X_i,C)\big)$~\cite[Theorem~3.7]{stahl}.
\subsection{Parametrizing hyperbolic segments}
The center of mass gives a smooth parametrization of the hyperbolic segment between two points $X$ and $Y$ via the map $t\mapsto Z_t$ defined by 
\[
(Z_t,z_t)=(X,1-t)\ast (Y,t).
\]
Letting $Z_0=X$ and $Z_1=Y$, we get a map $Z_t\colon [0,1]\to \HH$. It's possible to write an explicit formula for $Z_t$. We'll do this in the Poincaré disk model for convenience, choosing a point of view where $X$ is at the origin and $Y$ is on the positive real axis, with coordinate $\eta \in (0, 1)$. After some computations, we get
\begin{equation}\label{eq:parameterization-hyperbolic-segment}
Z_t=\frac{1+\eta^2(2t-1)-\sqrt{(1-\eta^2)(1-\eta^2(2t-1)^2)}}{2\eta t}.
\end{equation}
The function $t\mapsto Z_t$ is a smooth function $(0,1) \to \HH$ whose derivatives all have right and left limits as $t$ approaches $0$ and $1$ respectively. As a family parameterized by $Y$, the maps $t\mapsto Z_t$ depend continuously on the variable $\eta$ in the $\mathcal{C}^\infty$ topology on $\mathcal{C}^\infty \big((0,1),\HH\big)$. 

\subsection{Parametrizing hyperbolic triangles}
Our study of hyperbolic cone structures and their relation to triangulations (Propostition~\ref{length-homeo}) requires parameterizing hyperbolic triangles in way that the induced parameterization on each side only depends on the length of the corresponding side. This is precisely what convex combinations of points do. We'll denote the standard 2-simplex in $\R^3$ by $\Delta=\{(t_1,t_2,t_3)\in (0, 1)^3 : t_1+t_2+t_3=1\}$.
\begin{defn}\label{defn:barycentric-coordinates}
The \emph{barycentric coordinates} on the hyperbolic triangle with vertices $(X_1,X_2,X_3)$ are given by the map $f\colon \Delta\to \HH$ that sends $(t_1,t_2,t_3)$ to the point $Z$ defined by the relation
\[
(Z,z)= (X_1,t_1)\ast (X_2,t_2)\ast (X_3,t_3),
\]
discarding the mass $z$. The image of $f$ is the interior of the triangle $(X_1,X_2,X_3)$.
\end{defn}
Barycentric coordinates for hyperbolic triangles are analogous in various ways to the classical barycentric coordinates for Euclidean triangles. For instance, Ungar proved that the point $C$ defined by 
\[
(C,c)=(X_1,1/3)\ast(X_2,1/3)\ast (X_3,1/3)
\]
coincides with the intersection of the medians in the triangle $(X_1,X_2,X_3)$~\cite{ungar}. 

The map $f$ from Definiton~\ref{defn:barycentric-coordinates} extends naturally to the closed simplex $\overline{\Delta}$. The induced parametrizations of the sides of $f(\overline{\Delta})$ only depend on the sides' lengths, in the sense that if $f$ and $f'$ are the barycentric coordinate parameterizations of two triangles $(X_1,X_2,X_3)$ and $(X_1',X_2',X_3')$ with $d(X_1,X_2)=d(X_1',X_2')$, then the orientation-preserving isometry $A$ of $\HH$ that sends $X_1$ to $X_1'$ and $X_2$ to $X_2'$ also sends $f(t_1,t_2,0)$ to $f'(t_1,t_2,0)$ for every coordinate triple $(t_1,t_2,0) \in \overline{\Delta}$.

You can prove that $f$ is a smooth map by thinking of $f(t_1,t_2,t_3)$ as the center of mass of $(X_1,t_1)$ and $(X_2,t_2)\ast (X_3,t_3)$, describing it as the composition of two smooth maps of the form~\eqref{eq:parameterization-hyperbolic-segment}. Similarly, using the continuity of the $\ast$ operation, you can show that the map $f$ depends continuously on the vertices $(X_1,X_2,X_3)$ in the $\mathcal{C}^\infty$ topology on $\mathcal{C}^\infty(\Delta, \HH)$. 

\bibliographystyle{amsalpha}
\bibliography{references.bib}

\providecommand{\bysame}{\leavevmode\hbox to3em{\hrulefill}\thinspace}
\providecommand{\MR}{\relax\ifhmode\unskip\space\fi MR }
\providecommand{\MRhref}[2]{%
  \href{http://www.ams.org/mathscinet-getitem?mr=#1}{#2}
}
\providecommand{\href}[2]{#2}
\begin{thebibliography}{McO88}

\bibitem[BG99]{BenGol}
Robert~L. Benedetto and William~M. Goldman, \emph{The topology of the relative
  character varieties of a quadruply-punctured sphere}, Experiment. Math.
  \textbf{8} (1999), no.~1, 85--103. \MR{1685040}

\bibitem[BH99]{metric-non-pos}
Martin Bridson and Andr\'{e} Haefliger, \emph{Metric spaces of non-positive
  curvature}, Grundlehren der mathematischen {W}issenschaften, vol. 319,
  Springer-Verlag, 1999.

\bibitem[CDF14]{branch-sing}
Gabriel Calsamiglia, Bertrand Deroin, and Stefano Francaviglia, \emph{Branched
  projective structures with {F}uchsian holonomy}, Geom. Topol. \textbf{18}
  (2014), no.~1, 379--446. \MR{3159165}

\bibitem[DP07]{DryPar}
Emily~B. Dryden and Hugo Parlier, \emph{Collars and partitions of hyperbolic
  cone-surfaces}, Geom. Dedicata \textbf{127} (2007), 139--149. \MR{2338522}

\bibitem[DT19]{DeTh19}
Bertrand Deroin and Nicolas Tholozan, \emph{Supra-maximal representations from
  fundamental groups of punctured spheres to {$\operatorname{PSL}(2, \R)$}},
  Ann. Sci. {\'E}c. Norm. Sup{\'e}r. (4) \textbf{52} (2019), no.~5, 1305--1329.
  \MR{4057784}

\bibitem[Far21]{Far21}
Gianluca Faraco, \emph{Geometrisation of purely hyperbolic representations in
  {$\rm PSL_2\Bbb R$}}, Adv. Geom. \textbf{21} (2021), no.~1, 99--108.
  \MR{4203288}

\bibitem[FM12]{mcg-primer}
Benson Farb and Dan Margalit, \emph{A primer on mapping class groups},
  Princeton Mathematical Series, vol.~49, Princeton University Press,
  Princeton, NJ, 2012. \MR{2850125}

\bibitem[FST08]{triangulations}
Sergey Fomin, Michael Shapiro, and Dylan Thurston, \emph{Cluster algebras and
  triangulated surfaces. {I}. {C}luster complexes}, Acta Math. \textbf{201}
  (2008), no.~1, 83--146. \MR{2448067}

\bibitem[Ham03]{Ham03}
Ursula Hamenst\"{a}dt, \emph{Length functions and parameterizations of
  {T}eichm\"{u}ller space for surfaces with cusps}, Ann. Acad. Sci. Fenn. Math.
  \textbf{28} (2003), no.~1, 75--88. \MR{1976831}

\bibitem[Hir94]{hirsch-diff-topology}
Morris~W. Hirsch, \emph{Differential topology}, Graduate Texts in Mathematics,
  vol.~33, Springer-Verlag, New York, 1994, Corrected reprint of the 1976
  original. \MR{1336822}

\bibitem[KP10]{arc+curve}
Mustafa Korkmaz and Athanase Papadopoulos, \emph{On the arc and curve complex
  of a surface}, Math. Proc. Cambridge Philos. Soc. \textbf{148} (2010), no.~3,
  473--483. \MR{2609303}

\bibitem[Mar24]{action-angle}
Arnaud Maret, \emph{Action-angle coordinates for surface group representations
  in genus zero}, J. Symplectic Geom. \textbf{22} (2024), no.~5, 937--999
  (English).

\bibitem[Mat12]{Mat12}
Daniel~V. Mathews, \emph{Hyperbolic cone-manifold structures with prescribed
  holonomy {II}: higher genus}, Geom. Dedicata \textbf{160} (2012), 15--45.
  \MR{2970041}

\bibitem[McO88]{McOw88}
Robert~C. McOwen, \emph{Point singularities and conformal metrics on {R}iemann
  surfaces}, Proc. Amer. Math. Soc. \textbf{103} (1988), no.~1, 222--224.
  \MR{938672}

\bibitem[Mon16]{Mondello}
Gabriele Mondello, \emph{Topology of representation spaces of surface groups in
  {${\rm PSL}_2(\Bbb R)$} with assigned boundary monodromy and nonzero {E}uler
  number}, Pure Appl. Math. Q. \textbf{12} (2016), no.~3, 399--462.
  \MR{3767231}

\bibitem[Ngu21]{cone-metric}
Dat~Pham Nguyen, \emph{On {H}yperbolic {C}one {M}etrics, {$\rm PSL_2\Bbb
  R$}-{C}haracter {V}arieties, and {B}ranched {C}overings}, ProQuest LLC, Ann
  Arbor, MI, 2021, Thesis (Ph.D.)--Stanford University. \MR{4326834}

\bibitem[Sch93]{Sch93}
Paul Schmutz, \emph{Die {P}arametrisierung des {T}eichm\"{u}llerraumes durch
  geod\"{a}tische {L}\"{a}ngenfunktionen}, Comment. Math. Helv. \textbf{68}
  (1993), no.~2, 278--288. \MR{1214232}

\bibitem[Sta07]{stahl}
Saul Stahl, \emph{Mass in the hyperbolic plane}, 2007,
  \href{https://arxiv.org/abs/0705.3448}{\texttt{arXiv:0705.3448}}.

\bibitem[Thu98]{Thurston}
William~P. Thurston, \emph{Shapes of polyhedra and triangulations of the
  sphere}, Geometry and Topology Monographs \textbf{1} (1998), 511--549.

\bibitem[Tro91]{Tro91}
Marc Troyanov, \emph{Prescribing curvature on compact surfaces with conical
  singularities}, Trans. Am. Math. Soc. \textbf{324} (1991), no.~2, 793--821.

\bibitem[Ung04]{ungar}
Abraham~A. Ungar, \emph{The hyperbolic triangle centroid}, Commentat. Math.
  Univ. Carol. \textbf{45} (2004), no.~2, 355--369 (English).

\end{thebibliography}

\end{document}